\documentclass{amsart}
\usepackage{amsmath}
\usepackage{amssymb}
\usepackage{amsfonts}
\usepackage{graphicx}
\usepackage{texdraw}
\usepackage{graphpap}
\usepackage{wasysym}
\usepackage{pxfonts}
\usepackage[OT2,T1]{fontenc}

\DeclareSymbolFont{cyrletters}{OT2}{wncyr}{m}{n}
\DeclareMathSymbol{\berd}{\beta}{cyrletters}{"42}

\input txdtools

 \newtheorem{thm}{Theorem}[section]
 \newtheorem{cor}[thm]{Corollary}
 \newtheorem{lem}[thm]{Lemma}
 \newtheorem{prop}[thm]{Proposition}
 \theoremstyle{definition}
 \newtheorem{defn}[thm]{Definition}

 \theoremstyle{remark}
 
 \newtheorem{rem}[thm]{Remark}

\numberwithin{equation}{section} \numberwithin{figure}{section}

\newcommand{\e}{\mathrm e}

\newcommand{\C}{{\mathbb C}}
\newcommand{\D}{{\mathbb D}}

\newcommand{\T}{{\mathbb T}}
\newcommand{\R}{{\mathbb R}}

\newcommand{\calA}{{\mathcal A}}

\newcommand{\dA}{{\diff A}}

\newcommand{\calP}{{\mathcal P}}

\newcommand{\setS}{{\mathcal S}}

\newcommand{\Ordo}{\mathrm O}
\newcommand{\ordo}{\mathrm o}

\newcommand\Qfun{V}
\newcommand\vol{\mathrm{vol}}
\newcommand\clos{\text{\rm clos}}
\newcommand\Phull{\operatorname{\text{\rm phull}}}
\newcommand\Pol{\mathrm{Pol}}

\newcommand{\probab}{\mathrm{prob}}
\newcommand\probc{{\probab}_c(\C)}

\newcommand{\re}{\operatorname{Re}}
\newcommand{\im}{\operatorname{Im}}

\newcommand{\diff}{{\mathrm d}}
\newcommand{\imag}{{\mathrm i}}

\newcommand{\expect}{\mathbb{E}}
\newcommand{\EE}{\mathcal{E}}

\newcommand{\dist}{\operatorname{dist\,}}
\newcommand{\supp}{\operatorname{supp}}

\newcommand{\trace}{\mathrm{tr}}
\newcommand{\Sub}{\mathrm{Sub}}
\newcommand{\Obs}{\mathrm{Obst}}

\newcommand{\Int}{\text{\rm int}}
\newcommand{\sm}{\setminus}
\newcommand{\ti}{\tilde}

\begin{document}
\title[Coulomb gas ensembles and Laplacian growth]
{Coulomb gas ensembles and Laplacian growth}

%----------Author 1
\author[Hedenmalm]
{H\aa{}kan Hedenmalm}

\address{H{\aa}kan Hedenmalm\\ Department of Mathematics\\
KTH Royal Institute of Technology\\
S -- 100 44 Stockholm\\
SWEDEN}

\email{haakanh@math.kth.se}

\thanks{The first author is supported by the G\"oran 
Gustafsson Foundation (KVA) and Vetenskapsr\aa{}det (VR). 
The second author is supported by NSF Grant No. 0201893.}

%----------Author 2
\author[Makarov]
{Nikolai Makarov}

\address{Nikolai Makarov\\ Department of Mathematics\\ California Institute of
Technology\\ Pasadena\\ CA 91125\\ USA}

\email{makarov@caltech.edu}

%\subjclass{Primary 35B65, 30C35; Secondary 30C55,30C85}

%\keywords{Coulomb gas, Laplacian growth, Hele-Shaw flow}

\begin{abstract} 
We consider weight functions $Q:\C\to \R$ that are locally in a suitable
Sobolev space, and impose a logarithmic growth condition from below. We
use $Q$ as a confining potential in the model of one-component plasma 
($2$-dimensional Coulomb gas), and study the configuration of the electron  
cloud as the number $n$ of electrons tends to infinity, while the confining 
potential is rescaled: we use $mQ$ in place of $Q$ and let $m$ tend to
infinity as well. We show that if $m,n$ tend to infinity in a proportional
fashion, with $n/m\to t$ , where $0<t<+\infty$ is fixed, then the electrons 
accumulate on a compact set $S_t$, which we call the {\em droplet}. The set
$S_t$ can be obtained as the coincidence set of an obstacle problem, if we
remove a small set (the shallow points). Moreover, on the droplet $S_t$, 
the density of electrons is asymptotically $\Delta Q$. The growth of the 
droplets $S_t$ as $t$ increases is known as Laplacian growth. It is 
well-known that Laplacian growth is unstable. To analyze this feature, we 
introduce the notion of a {\em local droplet}, which involves removing part 
of the obstacle away from the set $S_t$. The local droplets are no longer 
uniquely determined by the time parameter $t$, but at least they may be 
partially ordered. We show that the growth of the local droplets may be 
terminated in a maximal local droplet, or by the droplets' growing to 
infinity in some direction (``fingering'').
\end{abstract}

%%% ----------------------------------------------------------------------
\maketitle
%%% ----------------------------------------------------------------------

\addtolength{\textheight}{2.2cm}

\section{Overview} \label{intro}

\subsection{Outline of the paper}
In Sections \ref{sec-qcl} and \ref{sec-J}, we study the one-component plasma 
(Coulomb gas ensemble) in two dimensions, and find the quasi-classical limit 
as the number $n$ of electrons tends to infinity while the confining potential 
is rescaled: $mQ$ replaces $Q$, where $m$ tends to infinity, so that 
$n/m\to t$. 
This was obtained previously by Johansson \cite{J} in the one-dimensional 
context. It turns out that Johansson's proof carries through with only minor 
modifications also in the two-dimensional case, as was explained earlier in 
our arXiv preprint \cite{HM}. Here, we make an effort to obtain the 
result under minimal smoothness and growth assumptions on the potential $Q$. 

In Section \ref{sec-obst}, we connect the equilibrium measure with an obstacle
problem, and show how to apply the Kinderlehrer-Stampacchia-Caffarelli theory
to obtain a priori smoothness of the solutions to the obstacle problem. 
We also show that the density of the equilibrium measure is given by 
$\Delta Q$ on the droplet, which permits us to reduce the complexity of the 
equilibrium measure to the study of its support (the droplet). Here, $\Delta
:=\partial\bar\partial$ is a quarter of the usual Laplacian.
The droplet is shown to equal the coincidence set for
the associated obstacle problem, if we remove the so-called shallow points. 
For smooth strictly convex $Q$, the topology of the droplets is shown to be
simple.

In Section \ref{sec-ld}, we introduce the notion of local droplets, which are
obtained when we pass from the potential $Q$ to its localization $Q_\Sigma$
for subsets $\Sigma\subset\C$ (cf., e.g. \cite{EF}). The local droplets 
are partially ordered, and in Section \ref{sec-cld}, we study maximal 
domination chains of local droplets. The maximal domination chains either 
end in a maximal local droplet, or grow to infinity. The local droplets appear 
to be natural from the point of view of physics (see, e.g., \cite{W}). They 
are also natural from the mathematical point of view: the description of all 
possible local droplets is exactly the inverse problem of potential theory.

One purpose with the material on domination chains of droplets in Section 
\ref{sec-cld} is to provide a natural setting to analyze Laplacian growth 
(i.e., the Hele-Shaw equation), which is known to be unstable in the forward 
time direction. This is explained in Section \ref{sec-hse}. 
The domination chains of droplets are interesting in part because of 
their integrability nature, especially in the case of potentials $Q$ with 
$\Delta Q=\mathrm{constant}>0$ near the local droplet (such $Q$ will be called 
constant strength potentials). 
This will be the topic of a forthcoming paper, where we will discuss the 
algebraic-geometric nature of maximal local droplets for constant strength 
potentials.
\medskip

\subsection{Comments on the exposition}
While a few of the results covered in this paper are essentially understood, 
we believe the reader will appreciate a rather self-contained and easily
accessible exposition. As for the treatment of Johansson's
theorems in Section \ref{sec-J}, the extension to the two-dimensional setting
requires some care about details, and as far as we know, no general proof has 
been available so far beyond the arXiv preprint \cite{HM}, where an excessive 
regularity condition was made to simplify the presentation (here, we remove 
that condition by modifying the smoothing argument of Johansson's paper 
\cite{J}; see Subsection \ref{subsec-entr}). 

The connection between equilibrium measures and obstacle problems is known 
(see, e.g., \cite{ST}). However, it is perhaps less well known that the 
Kinderlehrer-Stampacchia-Caffarelli theory (see \cite{KS}; cf. also \cite{HS}, 
where the same technique was used) allows us to develop an understanding of 
the equilibrium measures in terms of their supports, the droplets. This 
contrasts with the one-dimensional theory, where a lot of the difficulty is 
to determine the density of the equilibrium measure. 
As for the treatment of the Hele-Shaw equation, our approach based on 
equilibrium measures and obstacle problems allows us to develop the theory 
with low regularity. The standard approach to Hele-Shaw flow theory is to use
(partial) balayage and variational inequalities, see, e.g., \cite{GB}. 
We prefer the obstacle problem approach because it is more intuitive and 
geometrically appealing.

\subsection{Acknowledgements}
We thank Kurt Johansson for helpful comments in connection with the previous
arXiv preprint \cite{HM}, and Serguei Shimorin for help with the proofreading. 

%\newpage
% 1. This paper is an expanded version of ArXiv's paper. We added the 
%"dynamical" part emphasizing the  relation to LG (Laplacian growth), and 
%also we included several examples.
%The importance of LG connection can be explained as follows. Eventually, 
%we want to understand the quantum version asymptotically as $\hbar\to 0$. 
%It is a general principle that successive terms of the asymptotic expansion 
%can be described as a response of the previous terms to the variation of 
%the potential, see discussion the last section. LG characterizes exactly  
%The response of   the quasi-classical limit is exactly LG. Mention the 
%continuation of this topic in the papers [AHM] (with Yacin).

%\bs 2. We give a self-contained survey of LG. Our approach is based on 
%equilibrium measures and an obstacle problem rather than HS equation or 
%partial balayage. This approach is certainly not new but I'm not sure if it 
%is presented anywhere as a complete theory. Also, for (rather) general 
%potential several results are probably new. 

%\bs 3. The first three sections are essentially our ArXiv paper. I 
%only suggest some minor changes.

% -- The existence of the scaling limit is essentially (verbatim?) 
%Johansson's theorem. What should we do with it? 

% -- The third section is about Caffarelli theory. It would probably 
%make  sense to give some details about the proof.

\section{Quasi-classical limit of Coulomb gas ensembles}
\label{sec-qcl}

\subsection{One-component plasma (OCP)} 
In the 2-dimensional Cou\-lomb gas model (or rather OCP, the 
one-component plasma model), we have  $n$ electrons located at points 
$\{z_j\}_{j=1}^{n}$ in the complex plane, influenced by an external 
field. The potential of interaction is
\[
\log\frac{1}{|z_j-z_k|^2},\qquad j\neq k,\,\,\,j,k\in\{1,\ldots,n\},
\]
while the external field potential is denoted by $\Qfun(z)$. The function
\[
\Qfun:\,\C\to\R\cup\{+\infty\}
\]
is lower semi-continuous and sufficiently large to keep the electrons at 
finite distances. We shall supply the precise condition shortly.
%\[Q(z)-\log|z|^2\to+\infty,\qquad \text{as}\quad|z|\to+\infty.\]
The combined potential energy resulting from particle interaction and
the external potential is the function 
$\EE_\Qfun:\C^n\to\R\cup\{+\infty\}$ given by
\[
\EE_\Qfun(z)=\frac12\sum_{j,k:j\ne k}\log\frac1{|z_j-z_k|^2}+
\sum_j \Qfun(z_j),\qquad z=(z_1,\ldots,z_n)\in\C^n,
\]
where the summation indices $j,k$ are assumed confined to the set 
$\{1,\ldots,n\}$. In any reasonable gas dynamics model, the low energy
states are supposed to be more likely than the high energy states. 
For a positive constant $\beta$, let $Z_n=Z_{n,\beta,\Qfun}$ denote the
constant
\[
Z_{n}=\int_{\C^n}~\e^{-\frac{\beta}{2}\,\EE_\Qfun}~\diff \vol_{2n},
\]
where $\vol_{2n}$ denotes the standard volume measure in $\C^n
\cong\R^{2n}$.
We suppose that $0<Z_n<+\infty$, which means that the potential $\Qfun$
imposes a weak localization restraint on the plasma cloud. The
corresponding Gibbs model then gives the joint density of states
\[
\frac{1}{Z_{n}}\,\e^{-\frac{\beta}{2}\,\EE_\Qfun(z)}.
\]
where $\beta$ has the interpretation as the inverse temperature.
In terms of the usual van der Monde expression
\[
\triangle(z)=\prod_{j,k:j<k}(z_k-z_j),
\]
we may write 
\[
Z_{n}=
\int_{\C^n}|\triangle(z)|^\beta \e^{-\frac{\beta}{2}\sum_j\Qfun(z_j)}~\diff
\vol_{2n}(z).
\]
We thus introduce a probability point process
\[
\Pi_n\equiv\Pi_{n,\beta,\Qfun}\in\probab(\C^n)
\]
($\probab(\C^n)$ is the convex set of all Borel probability measures on $\C^n$)
by setting
\[
\diff\Pi_n(z)=\frac{\e^{-\frac{\beta}{2}\,\EE_\Qfun(z)}}{Z_{n}}\,
\diff \vol_{2n}(z)=\frac{|\triangle(z)|^\beta}{Z_n} 
\e^{-\frac{\beta}{2}\sum_j\Qfun(z_j)}~\diff \vol_{2n}(z), \qquad z\in\C^n.
\]

\subsection{Marginal measures.}
For integers $k=1,\ldots,n$, we define the marginal probability measure
$\Pi_n^{(k)}\in\probab(\C^k)$ by setting
\[
\Pi_n^{(k)}(e)=\Pi_n(e\times\C^{n-k}),
\]
for Borel measurable subsets $e\subset\C^k$; in particular, $\Pi_n^{(n)}=
\Pi_n$.  
The associated measures
\[
\Gamma_n^{(k)}=\frac{n!}{(n-k)!}\,\Pi_n^{(k)}
\]
are known as {\em intensity (or correlation) measures}.
For $k=n$, we have $\Gamma_n^{(n)}=n!\,\Pi_n$, which is why we simplify the 
notation and write $\Gamma_n=\Gamma_n^{(n)}$. 
On the other hand, for $k=1$, we have ($\expect$ is the expectation operation)
\[
\Gamma_n^{(1)}(e)=\expect\big[\#\{j:z_j\in e\}\big],
\]
where it is tacitly assumed that $j$ is confined to the set $\{1,\ldots,n\}$,
and $\#$ denotes counting measure. 
In more explicit form, we have, for $n=2$ and $k=1$,
\begin{equation}
\diff\Gamma_2^{(1)}(\zeta)=
\frac{2\int_\C|\zeta-\xi|^\beta\diff\mu(\xi)}
{\int_{\C^2}|\xi-\eta|^\beta\diff\mu(\xi)\diff\mu(\eta)}
\diff\mu(\zeta),\qquad
\zeta\in\C,
\label{eq-G21}
\end{equation}
where $\diff\mu(\xi)=\e^{-\frac{\beta}{2}\Qfun(\xi)}\diff \vol_2(\xi)$.
More generally, for $1\le k\le n$ and a Borel subset $e\subset\C^k$, we have
\[
\Gamma^{(k)}_n(e)=\expect[\#\{(j_1,\ldots,j_k)\in\,\text{perm}(k,n):
\,\,(z_{j_1},\ldots,z_{j_k})\in e\}],
\]
where $\text{perm}(k,n)$ stands for the collection of all permutations of 
length $k$ of the set $\{1,\ldots,n\}$. 

\begin{rem} \rm
In the above definition of the probability measure $\Pi_n$, we realize that
\[
\e^{-\frac{\beta}{2}\sum_j\Qfun(z_j)}\diff \vol_{2n}(z)=\diff\mu(z_1)\cdots
\diff\mu(z_n),\qquad z=(z_1,\ldots, z_n),
\]
where
\[
\diff\mu(\xi)=\e^{-\frac{\beta}{2}\Qfun(\xi)}\diff \vol_2(\xi),
\qquad\xi\in\C.
\]
Most of the above discussion does not depend on this particular structure of 
the measure $\mu$, and we are free to consider more general 
measures.
For instance, this allows us to include the one-dimensional theory in the 
model.
\end{rem}

\subsection{The random normal matrix model} If $\beta=2$, then the
probability measure
\[
\diff \Pi_n(z)=\frac{|\triangle(z)|^2}{Z_n} e^{-\sum_j\Qfun(z_j)}
\diff \vol_{2n}(z)
\]
with normalization constant
\[
Z_n=\int_{\C^n}|\triangle(z)|^2 e^{-\sum_j\Qfun(z_j)}
\diff \vol_{2n}(z)
\]
describes the distribution of the eigenvalues of $n\times n$ 
Random Normal Matrices (RNM) with joint probability measure proportional to
\[
e^{-\trace\,\Qfun(M)}\diff M,
\]
where ``tr'' is the trace, and $\diff M$ stands for the natural ``Haar-type'' 
measure on the submanifold of all complex-valued $n\times n$ matrices $M$ 
with $M^*M=MM^*$ (the normal matrices). In this case the point process is 
determinantal:
\begin{equation}
\diff\Gamma_n^{(k)}(z)=
\det\left[ K_n(z_i,z_j)\right]_{i,j=1}^k e^{-\sum_j\Qfun(z_j)}\,
\diff\vol_{2k}(z),
\label{eq-detK}
\end{equation}
where $K_n$ is the reproducing kernel in the polynomial Bargmann-Fock space
\[
{\Pol}_n=
{\rm span}\{1,z,\dots, z^{n-1}\}\subset L^2(\C,\e^{-\Qfun}).
\]
We thus consider $\Pol_n$ as a finite-dimensional linear 
subspace of $L^2(\C, e^{-\Qfun})$ (linearity is always with respect to the 
field $\C$), and the Gram-Schmidt procedure supplies, for 
$j=0,\ldots, n-1$, polynomials $p_j$ of degree $j$ and norm $1$ such that 
$p_j\perp p_k$ for $j\neq k$.
In terms of these orthogonal polynomials, we have
\begin{equation}
K_n(z,w)=\sum_0^{n-1} p_j(z)\,\bar p_j(w).
\label{eq-K}
\end{equation}
The algebraic mechanism behind the formula for the correlation measure
$\Gamma^{(m)}_n$ is well understood. See, for instance, Mehta's book \cite{M}.

\subsection{Aggregation of quantum droplets} For reasons that will become
clearer later on, we shall regard the point process $\Gamma_n=n!\Pi_n$
(or, equivalently, $\Pi_n$) as a quantum droplet. We are interested in the 
transition $\Gamma_n\to\Gamma_{n+1}$, which corresponds to adding one more 
electron to the droplet. A direct comparison of the processes $\Gamma_n$ and 
$\Gamma_{n+1}$ is not possible, and we are led to consider marginal 
intensities. The following lemma for $\beta=2$ has the interpretation that 
if we add an electron, the expected number of $k$-tuples of electrons 
increases everywhere in $\C^k$.  

\begin{lem} If  $\beta=2$, then
\[\forall k, \qquad \Gamma_n^{(k)}\le
\Gamma_{n+1}^{(k)}.\]
\label{lem-2.2}
\end{lem}

\begin{proof}
In view of \eqref{eq-K}, we have
\[
\left[ K_{n+1}(z_i,z_j)\right]_{i,j=1}^k=
\left[K_{n}(z_i,z_j)\right]_{i,j=1}^k+\left[p_{n}(z_i)\bar p_n(z_j)
\right]_{i,j=1}^k,
\]
where all matrices involved are positive (semi)definite (the rightmost 
matrix has rank 1). As we compare with \eqref{eq-detK}, we realize that
the desired assertion
\[
\det\left[ K_{n}(z_i,z_j)\right]_{i,j=1}^k\le
\det\left[K_{n+1}(z_i,z_j)\right]_{i,j=1}^k
\]
is an immediate consequence of the minimax principle (see, e. g., the books
of Dunford, Schwarz \cite{DS} and Gohberg, Krein \cite{GK}).
\end{proof}

\begin{rem}
This ``aggregation'' property might well be true for all 
$\beta\le 2$ but it certainly fails for $\beta>2$. We consider
the illuminating special case $\Gamma_1^{(1)}\le\Gamma_{2}^{(1)}$, which 
in the notation of \eqref{eq-G21} asserts that
\begin{equation}
\int_{\C^2}|\xi-\eta|^\beta\diff\mu(\xi)\diff\mu(\eta)\le
2 \mu(\C)\int_\C |\zeta-\xi|^\beta\diff\mu(\xi),\qquad \zeta\in\C.
\label{eq-test}
\end{equation}
The measure $\diff\mu(\xi)=\e^{-\frac{\beta}{2}\Qfun(\xi)}\diff\vol_2(\xi)$ 
can essentially be replaced by an fairly arbitrary positive Borel measure (with
finite moments). As we plug in the choice 
$\diff\mu=\diff\delta_0+\diff\delta_{1}$, we see that \eqref{eq-test} is 
equivalent to 
\[
|\zeta|^\beta+|\zeta-1|^\beta\ge\frac12,\qquad \zeta\in\C.
\]
With $\zeta=\frac12$ this gives $\beta\le 2$.
In fact, it is possible to show that the inequality 
$\Gamma_1^{(1)}\le\Gamma_{2}^{(1)}$ holds generally for $0<\beta\le2$. 
We outline the argument. It suffices to consider $z=0$ in \eqref{eq-test}, 
and to show that
\begin{equation}
\int_{\C^2}\big\{|\xi|^\beta+|\eta|^\beta-|\xi-\eta|^\beta\big\}
\diff\mu(\xi)\diff\mu(\eta)\ge0
\label{eq-test2}
\end{equation}
for all positive measures $\mu$ with finite moments. 
For $0<\beta\le1$ the $L^\beta$ triangle inequality shows that the integrand
on the left hand side is positive point-wise, and the assertion is immediate.
We turn to the remaining case $1<\beta<2$. One first establishes with the 
methods of Calculus that
\[
(1+t+2x)^{\beta/2}\le 1+t^{\beta/2}+\beta x,\qquad 
-\sqrt{t}\le x\le\sqrt{t}, \,\,0<t<+\infty,
\]
which in complex form becomes
\[
|1+\tau|^\beta\le 1+|\tau|^\beta+\beta\re \tau,\qquad \tau\in\D,
\]
where $\D$ denotes the open unit disk in $\C$. By homogenization, this 
inequality leads to
\[
|\xi-\eta|^\beta\le |\xi|^\beta+|\eta|^\beta-\beta\min\{|\xi|^{\beta-2},
|\eta|^{\beta-2}\}\,\re(\bar\eta\xi),\qquad \xi,\eta\in\C,
\]
so that
\[
|\xi|^\beta+|\eta|^\beta-|\xi-\eta|^\beta\ge
\beta\min\{|\xi|^{\beta-2},|\eta|^{\beta-2}\}\,\re(\bar\eta\xi),\qquad 
\xi,\eta\in\C.
\]
So, to get \eqref{eq-test2} it suffices to obtain
\[
\int_{\C^2}
\min\{|\xi|^{\beta-2},|\eta|^{\beta-2}\}\re(\bar\eta\xi)\,
\diff\mu(\xi)\diff\mu(\eta)\ge0.
\]
But this is an immediate consequence of Schur's product theorem for positive
definite matrices (in this case we have ``continuous'' matrices), as
both $\min\{|\xi|^{\beta-2},|\eta|^{\beta-2}\}$ and $\re(\bar\eta\xi)$ express
positive definite kernels.
\end{rem}

\subsection{Scaling and the class of weights}
 
If we keep the confining potential $\Qfun$ fixed, and
let $n$ (the number of electrons) grow, the process $\Pi_n$ will generically
grow beyond any confinement. For this reason, it is necessary to jack up the
confinement as $n$ grows. This is achieved by putting $\Qfun=mQ$, where $m$
is a scaling parameter and 
\[
Q:\,\C\to\R\cup\{+\infty\}
\] 
is a fixed potential, assumed to be lower semi-continuous. To avoid degeneracy,
we must suppose that $Q<+\infty$ at least on a set of positive area. 
From well-known physical considerations, it is natural 
to let $m$ be essentially proportional to $n$. As we are free to pick $Q$ 
as we like, we may assume that the proportionality
constant is $1$, that is, that $m=n+\ordo(n)$ as $n\to+\infty$.
The growth requirement on $Q$ which conforms with this normalization is
\begin{equation}
Q(z)-\log|z|^2\to+\infty\quad\text{as}\,\,\,|z|\to+\infty.
\label{eq-2.6}
\end{equation}
\medskip

\subsection{The equilibrium measure} 
\label{sub-equil}
We consider the limit of the point 
processes 
\[
\Pi_{mQ,n}\qquad \text{as}\quad n\to+\infty\quad\text{while}\,\,\,
\frac{n}{m}\to 1,
\]
while assuming that $Q$ grows in accordance with \eqref{eq-2.6}.
In this case we have convergence of the saddle point configurations. More 
precisely, the probability measures
\begin{equation}
\sigma_n=\frac1n\sum_j\delta_{z_j},
\label{eq-sigma}
\end{equation}
which minimize the functionals (we write $z=(z_1,\ldots,z_n)$) 
\[
I^\#_{mQ,n}[\sigma_n]:=\frac2{n(n-1)}\EE_{mQ}(z)
=\frac1{n(n-1)}\sum_{j,k:j\ne k}\log\frac1{|z_j-z_k|^2}+\frac{2m}{n(n-1)}
\sum_j Q(z_j),
\]
converge as $n\to+\infty$ while $m=n+\ordo(n)$ in the weak-star sense of 
measures to the unique probability measure $\sigma=\hat\sigma_Q$ which 
minimizes the weighted logarithmic energy
\begin{equation}
\label{eq-intQ}
I_Q[\sigma]:=\int_{\C^2}\log\frac1{|\xi-\eta|^2}
\diff\sigma(\xi)\diff\sigma(\eta)+2\int_\C Q\diff\sigma.
\end{equation}
This comes as no big surprise given the striking similarity of the expressions 
$I^\#_{mQ,n}[\sigma_n]$ and $I_Q[\sigma]$. 
The configuration of points corresponding to a minimizer $\sigma_n$ is known 
as a collection of {\em weighted Fekete points}, and the measure 
$\hat\sigma_Q$ is called the {\em equilibrium measure}. The existence and 
uniqueness of the minimizing measure $\hat\sigma_Q$ is due to Frostman.
Let $\probc$ denote the convex body of all compactly supported Borel 
probability measures on $\C$. 

\begin{thm} [Frostman] There exists a unique equilibrium measure $\hat\sigma
=\hat\sigma_Q$ such that
\[
I_Q[\hat\sigma]=\inf_\sigma I_Q[\sigma],
\]
the infimum being taken over all compactly supported probability measures 
$\sigma$. 
\label{thm-frost1}
\end{thm}

For the proof, we refer to \cite{ST}. We will write
\begin{equation}
\gamma(Q):=I_Q[\hat\sigma_O],\quad \gamma^*(Q):=\gamma(Q)-\int_\C
Q\diff\hat\sigma_Q,
\label{eq-gQ}
\end{equation}
for the (modified) Robin constants involved. Let 
\begin{equation}
L_Q(\xi,\eta):=\log\frac1{|\xi-\eta|^2}+Q(\xi)+Q(\eta),
\label{eq-LQ}
\end{equation}
and observe that for probability measures $\sigma$, we have
\begin{equation}
\label{eq-intQ2}
I_Q[\sigma]=\int_{\C^2}L_Q(\xi,\eta)\,
\diff\sigma(\xi)\diff\sigma(\eta).
\end{equation}
Next, we introduce the weighted potential
\[
U_Q^\sigma(\xi)=\int_\C L_Q(\xi,\eta)\diff\sigma(\eta)
\]
and observe that since
\begin{equation}
\label{eq-intQ3}
I_Q[\sigma]=\int_{\C}U_Q^\sigma(\xi)\diff\sigma(\xi)
\end{equation}
we expect that the energy minimizer $\sigma=\hat\sigma_Q$ should have 
$U_Q^\sigma$ constant on the support 
\[
S=S_Q:=\supp\hat\sigma_Q,
\] 
and that constant should also equal the minimum value of $U_Q^\sigma$.
We will at times use the notation $\hat\sigma_Q=\hat\sigma[Q]$ and $S_Q=S[Q]$. 
We use q.e. as short-hand for {\em quasi-everywhere}.

\begin{thm} [Frostman]
The support $S_Q$ of the equilibrium measure $\hat\sigma_Q$ 
is compact. Moreover, if $\gamma(Q)$ is as in \eqref{eq-gQ}, then
$U^{\hat\sigma_Q}_Q\ge\gamma(Q)$ q.e. on $\C$, while 
$U^{\hat\sigma_Q}_Q\le\gamma(Q)$ at each point of $S_Q$.
The value $\gamma(Q)$ equals the minimal energy $I_Q[\hat\sigma_Q]$. 
\label{thm-frost2}
\end{thm}

For the proof, we refer to \cite{ST}. The number $\e^{-\gamma(Q)}$ is said to
be the weighted capacity.

In terms of the usual logarithmic potential
\[
U^\sigma(\xi)=\int_\C \log\frac{1}{|\xi-\eta|^2}\diff\sigma(\eta),
\]
we see that for a compactly supported probability measure $\sigma$,
\[
U^{\sigma}_{Q}(\xi)=\int_\C L_{Q}(\xi,\eta)\diff\sigma(\eta)
=U^{\sigma}(\xi)+Q(\xi)+\int_\C Q\diff\sigma,
\]
which allows us to write Frostman's Theorem \ref{thm-frost2} in the following
form.

\begin{thm} [Frostman]
The support $S_Q$ of the equilibrium measure $\hat\sigma_Q$ 
is compact. Moreover, if $\gamma^*(Q)$ is as in \eqref{eq-gQ}, then
$U^{\hat\sigma_Q}+Q\ge\gamma^*(Q)$ q.e. on $\C$, while 
$U^{\hat\sigma_Q}+Q\le\gamma^*(Q)$ at each point of $S_Q$.
\label{thm-frost3}
\end{thm}

Let $\hat\sigma_{mQ,n}$ denote the probability measure $\sigma_n$ given by
\eqref{eq-sigma} corresponding to a weighted Fekete point configuration (i.e.,
a minimizing configuration).
The convergence to the global energy minimizing measure is as follows.  

\begin{thm}[Fekete, Totik] We have the convergence
\[\hat\sigma_{mQ,n}\to\hat\sigma_Q\quad\text{as}\,\,\,n\to+\infty\,\,\,
\text{while}\,\,\,m=n+\ordo(n)
\]
in the weak-star sense of measures. Moreover, we have convergence in energy:
\[
I^\#_{mQ,n}[\hat\sigma_{mQ,n}]\to I_Q[\hat\sigma_Q]=\gamma(Q),
\quad\text{as}\,\,\,n\to+\infty\,\,\,\text{while}\,\,\,m=n+\ordo(n).
\]  
\label{thm-FT}
\end{thm}

For the proof, we refer to \cite{ST}, p. 145.

\subsection{Johansson's marginal measure theorem for the plane}
For a probability measure $\sigma\in\probab(\C)$ and an integer 
$k=1,2,3,\ldots$, we denote by 
$\sigma^{\otimes k}\in\probab(\C^k)$ the product measure given by
\[
\diff\sigma^{\otimes k}(z_1,\ldots,z_k)=\diff\sigma(z_1)\cdots
\diff\sigma(z_k).
\]

\begin{defn} We say that $Q$ has {\em extra growth} provided that
\begin{equation}
Q(z)\ge(1+\delta_0)\log(1+|z|^2)-C_0,\qquad z\in\C,
\label{eq-Q1}
\end{equation}
holds for some small but positive value of $\delta_0$ and some (positive) 
real constant $C_0$. Moreover, we say that $Q$ is {\em regular} provided that 
it is bounded and continuous in an open neighborhood of 
$\setS_Q=\supp\hat\sigma_Q$.
\label{def-reg}
\end{defn}

\begin{thm} 
Suppose $Q$ is regular with extra growth.
Then, for every $k=1,2,3,\ldots$, we have the convergence
\[
\Pi^{(k)}_{mQ,n} \to \hat\sigma_Q^{\otimes k} \quad\text{as}\,\,\,
n\to+\infty\,\,\,\text{while}\,\,\,m=n+\ordo(n),
\]
in the weak-star sense of measures.
\label{thm-J}
\end{thm}

\begin{rem}
(i) Johansson \cite{J} proves his theorem in the degenerate real line case 
when $Q(\xi)=+\infty$ for $\xi\in\C\setminus\R$ (the Hermitian matrix case). 
This can be viewed as a limit case of our considerations. However, the
approach of Johansson's proof can be modified so as to include the complex 
plane case stated here. We indicate the necessary modifications in an appendix 
below.

\noindent(ii) An alternative formulation of Theorem \ref{thm-J} runs as 
follows.
As $n\to+\infty$ while $m=n+\ordo(n)$, the random variables $z_1,\dots, z_k$ 
on $(\C^n,\Pi_{mQ,n})$ are asymptotically i.i.d. with law $\hat\sigma_Q$.

\noindent(iii) We now find an application of Theorem \ref{thm-J} to linear
statistics.
Let $\mathrm{C}_b(\C^k)$ denote the Banach space of bounded
continuous functions in $\C^k$. Moreover, let the trace $\trace_n f$ of the 
function $f\in \mathrm{C}_b(\C)$ be given by
\[
\trace_n f=\sum_j f(z_j),
\]
where the sum as usual runs over $j=1,\ldots,n$ and $z_1,\ldots,z_n$ are 
random variables with joint probability $(\C^n,\Pi_{mQ,n})$. 
For each $j=1,\ldots,n$, we have, in view of Johansson's marginal measure 
theorem, for $f\in\mathrm{C}_b(\C)$,
\[
\expect[f(z_j)]=\int_\C f(\xi)\diff\Pi_{mQ,n}^{(1)}(\xi)~\to~
\int_\C f(\xi)\diff\hat\sigma_Q(\xi)=:\langle f,\hat\sigma_Q\rangle,
\]
as $n\to+\infty$ while $m=n+\ordo(n)$. By forming the average over $j$, we
get, for $f\in\mathrm{C}_b(\C)$, 
\[
\expect\left[\frac1n \trace_n f\right]
=\int_\C f(\xi)\diff\Pi_{mQ,n}^{(1)}(\xi)~\to~\langle f,\hat\sigma_Q\rangle,
\]
as $n\to+\infty$ while $m=n+\ordo(n)$.
There is an analogous statement which holds for functions 
$f\in \mathrm{C}_b(\C^k)$ and involves the measure $\hat\sigma^{\otimes k}_Q$ 
in place of $\hat\sigma_Q$. This more general statement allows us to obtain
that for $f\in\mathrm{C}_b(\C)$, ($k,k'$ are fixed integers $\ge0$)
\[
\expect
\bigg[\bigg(\frac1n \trace_n f\bigg)^k\bigg(\frac1n \trace_n \bar f\bigg)^{k'}
\bigg]~\to~
\left(\langle f,\hat\sigma_Q\rangle\right)^k
\left(\langle \bar f,\hat\sigma_Q\rangle\right)^{k'},
\]
as $n\to+\infty$ while $m=n+\ordo(n)$. Here, as usual, $\bar f$ is the
function whose values are complex conjugate to those of $f$. 
This expresses that $\frac1n\trace_n f$ tends to the constant value 
$\langle f,\hat\sigma_Q\rangle$ in all moments as $n\to+\infty$ 
while $m=n+\ordo(n)$, and hence in particular, we have convergence in 
distribution (as in the weak law of large numbers).

\noindent(iv) We remark that Theorem \ref{thm-J} holds independently of the 
value of the inverse temperature $\beta$. 
However, for $\beta=2$, much more precise statements have been obtained 
recently in \cite{AHM1}, \cite{AHM2}, \cite{AHM3}. 
The reason why this is possible is the determinantal property \eqref{eq-detK}. 
To give some hints about the results, we introduce the fluctuation
\[
\mathrm{fl}_nf=\trace_n f-n\langle f,\hat\sigma_Q\rangle.
\]
In view of (iv), we know that $\frac1n\mathrm{fl}_nf\to0$ in moments and hence
in distribution as $n\to+\infty$ while $m=n+\ordo(n)$. Next, suppose 
$n\to+\infty$ while $m=n+\ordo(1)$, which means that $m$ is kept much closer to
$n$ than before, and suppose also that the function $Q$ is real-analytically 
smooth with $\Delta Q>0$ in the interior of $S_Q$ (we recall that $S_Q$ is 
the support of the equilibrium measure $\hat\sigma_Q$).
In analogy with the CLT (central limit theorem),  it is shown in \cite{AHM1}, 
\cite{AHM2} that under some additional assumptions, the stochastic variable 
$\mathrm{fl}_nf$ converges in distribution to a real-valued Gaussian with 
expectation $e_f$ and variance $v_f$, 
\[
e_f=\frac{1}{2\pi}\int_{S_Q}f\,\Delta\log\Delta Q\,\diff\vol_2,\quad
v_f=\frac{1}{4\pi}\int_{S_Q}|\nabla f|^2\diff\vol_2,
\]
provided the function $f$ is real-valued, $C^\infty$-smooth, and is supported 
in the interior of $S_Q$. The extension to general test functions $f$ is 
obtained in \cite{AHM3}; the general formulae for $e_f,v_f$ include 
boundary effects. 
\end{rem}

\subsection{Johansson's free energy theorem for the plane}

We recall the expression for the normalization constant 
\[
Z_{m,n}=\int_{\C^n}|\triangle(z)|^\beta\e^{-\frac{\beta}{2} m\sum_j Q(z_j)}
\diff\vol_{2n}(z),
\]
which we write in the form
\begin{equation}
Z_{m,n}=\int_{\C^n}\exp\Bigg\{-\frac{\beta}{4}\sum_{j,k:j\ne k}
L_Q(z_j,z_k)+\frac{\beta}{2}(n-m-1)\sum_j Q(z_j)\Bigg\}
\diff\vol_{2n}(z),
\label{eq-Z}
\end{equation}
where $L_Q$ is as in \eqref{eq-LQ}. The quantity 
\[
\frac{1}{n(n-1)}\log Z_{m,n}
\]
has in the physics literature acquired the name {\em free energy} (frequently
$n^2$ is used in place of $n(n-1)$; asymptotically, there is no difference). 
See Definition \ref{def-reg} for the terms {\em regular} and 
{\em extra growth}.

\begin{thm} 
Suppose $Q$ is regular with extra growth. Then
\[
\frac{1}{n(n-1)}\log Z_{m,n}\to-\frac{\beta}{4}\gamma(Q)=
-\frac{\beta}{4}I_Q[\hat\sigma_Q]\quad\text{as}\,\,\,
n\to+\infty\,\,\,\text{while}\,\,\,m=n+\ordo(n).
\]
\label{thm-J-free}
\end{thm}

\begin{rem}
As with Theorem \ref{thm-J}, Johansson \cite{J} proves his theorem in the 
degenerate real line case when $Q(\xi)=+\infty$ for $\xi\in\C\setminus\R$ 
(the Hermitian matrix case). This can be viewed as a limit case of our 
considerations. However, the approach of Johansson's proof can be modified 
so as to include the complex plane case stated here. We indicate the 
necessary modifications in an appendix below.
\end{rem}

\subsection{Aggregation of equilibrium measures} 
\label{sub2.9}
We now look at the quasi-classical limit of the evolution of quantum 
droplets (the addition of more electrons to the droplet). This will allow us
to understand how the quantum process is related to a growth process of 
Hele-Shaw type for compact sets in the plane. 

We restrict our attention to the potentials $Q$ that satisfy a scale invariant
version of the growth condition \eqref{eq-2.6}, namely
\[
Q(z)-A\log |z|\to+\infty,\quad\text{as}\,\,\, |z|\to+\infty,
\]
no matter how big the positive parameter $A$ gets.
We will be interested in the evolution of positive measures
\[
\hat\sigma_t\equiv \hat\sigma_t[Q]:=t\hat\sigma_{Q/t},
\]
where $t$ ranges over $0<t<+\infty$. We write $S_t=S_t[Q]$ for the support 
of the measure $\hat\sigma_t[Q]$ (i.e., $S_t=S_{Q/t}$). 
Note that $\hat\sigma_t[Q]$ has total mass $t$. 
The process of increasing the parameter $t$ has the following interpretation.
We consider the limit process of letting $n\to+\infty$ while $m=n/t+\ordo(n)$.
In other words, $m\to+\infty$ while $n=mt+\ordo(m)$. An increase of $t$ 
therefore has the interpretation of increasing the total number of electrons 
$n$ for fixed $m$. To rescale, we introduce $m'=mt$, so that the relationship
reads $n=m'+\ordo(m')$. Since $mQ=m'Q/t$, rescaling also means we must 
replace $Q$ by $Q/t$.
In view of Johansson's theorem, we find that
\[
\lim\,\expect~\frac{\#\{{\rm electrons~in~} e\}}m=
t\lim\, \expect~\frac{\#\{{\rm electrons~in~} e\}}{m'}
=t\hat\sigma_{Q/t}(e)=\hat\sigma_t[Q](e).
\]
In other words, the growth process $\hat\sigma_t[Q]$ is the quasi-classical
limit of the growth process of adding electrons to the quantum droplet. 
We shall see that if $Q$ is $C^2$-smooth, the measure $\hat\sigma_t[Q]$ is 
determined uniquely by its support $S_t[Q]$. We understand the set $S_t[Q]$ 
as a (classical) droplet; cf. Subsection \ref{subsec-5.1} for a precise 
definition. 

\begin{cor} 
The family of measures $\hat\sigma_t[Q]$ is monotonically increasing in $t$. 
\label{cor-2.12}
\end{cor}

\begin{proof} This is true for  quantum droplets if $\beta=2$; the 
quasi-classical limit does not depend on $\beta$.
\end{proof}

\begin{rem}
\label{rem-2.13}
It is not hard to write down a potential theoretic proof of this fact; see
Proposition \ref{prop-4.15}.
\end{rem}

\section{Appendix: proof of Johansson's marginal measure and free energy
theorems for the plane}
\label{sec-J}

\subsection{Fekete configurations}
\label{subsec-3.1}
The approach to prove Johansson's theorem in this setting is to show that
point configurations $(z_1,\ldots,z_n)$ whose associated energy functional
\[
I^\sharp_{n,(n-1)Q}[\sigma_n]=\frac1{n(n-1)}\sum_{j,k:j\neq k}
L_Q(z_j,z_k)
\]
deviate substantially from the minimum are highly unlikely. We note that since 
$m=n+\ordo(n)$ is assumed, the choice to replace $m$ by $n-1$ in the energy
is reasonable. Let write 
\[
\hat\sigma_n:=\hat\sigma_{n,(n-1)Q}
\]
for the minimizing (Fekete) measure in the context of Theorem \ref{thm-FT},
and we also write 
\[
I^\sharp_{n}[\hat\sigma_n]:=I^\sharp_{n,(n-1)Q}[\hat\sigma_{n,(n-1)Q}]
\]
for the associated energy. By \cite{ST}, pp. 143--145, the sequence of energies
$I^\sharp_{n}[\hat\sigma_n]$ is decreasing in $n$, and converges to 
$I_Q[\hat\sigma_Q]=\gamma(Q)$ as $n\to+\infty$ (cf. Theorem \ref{thm-FT}). 

\subsection{An entropy estimate}\label{subsec-entr}
We introduce an auxiliary Borel measurable function 
$\phi:\C\to[0,+\infty)$ with
\begin{equation}
\int_\C\phi\diff\vol_2=1,\quad \int_\C(Q+|\log\phi|)\phi\diff\vol_2
<+\infty,
\label{eq-intphi}
\end{equation}
with the understanding that $\phi\log\phi=0$ at points where 
$\phi=0$. We sort of artificially smuggle it into the expression \eqref{eq-Z} 
for $Z_{m,n}$:
$$Z_{m,n}=\int_{\C^n}\exp\Bigg\{-\frac{\beta}{4}\sum_{j,k:j\ne k}
L_Q(z_j,z_k)+\frac{\beta}{2}(n-m-1)\sum_j Q(z_j)-\sum_j\log\phi(z_j)\Bigg\}
\prod_j\phi(z_j)\,\diff\vol_{2n}(z).$$
Now, by Jensen's inequality, we have, with 
$\diff \sigma_\phi=\phi\diff\vol_2$,
\begin{multline}
\log Z_{m,n}\\
\ge\log\int_{\C^n}\Bigg\{-\frac{\beta}{4}\sum_{j,k:j\ne k}
L_Q(z_j,z_k)+\frac{\beta}{2}(n-m-1)\sum_j Q(z_j)-\sum_j\log\phi(z_j)\Bigg\}
\prod_j\phi(z_j)\,\diff\vol_{2n}(z)
\\
=-\frac{\beta n(n-1)}{4}\int_{\C^2}L_Q(\xi,\eta)
\diff\sigma_\phi(\xi)\diff\sigma_\phi(\eta)
+\frac{\beta}{2} n(n-m-1)\int_\C Q\diff\sigma_\phi
-n\int_\C\log\phi\,\diff\sigma_\phi,
\label{eq-estZ1}
\end{multline}
where we used repeatedly that $\sigma_\phi$ is a probability measure. 
We rewrite this as
\begin{equation*}
\frac{1}{n(n-1)}\log Z_{m,n}\ge-\frac{\beta}{4}I_Q[\sigma_\phi]
+\beta\frac{n-m-1}{2(n-1)}\int_\C Q\diff\sigma_\phi
-\frac{1}{n-1}\int_\C\log\phi\diff\sigma_\phi,
\end{equation*}
This gives (as $n\to+\infty$ while $m=n+\ordo(n)$)
\begin{equation}
\liminf \frac{1}{n(n-1)}\log Z_{m,n}\ge -\frac{\beta}{4}\,I_Q[\sigma_\phi].
\label{eq-estbelow}
\end{equation}
The condition on $\phi$ that $\phi\log\phi\in L^1(\C)$ is of entropy type,
and this is the reason why we call \eqref{eq-estbelow} an {\em entropy 
estimate}. We would like to plug in the choice $\sigma_\phi=\hat\sigma_Q$ into 
the entropy estimate \eqref{eq-estbelow} to get an effective bound. At this
point, we do not know enough about $\hat\sigma_Q$ to be sure whether it is 
of the form $\sigma_\phi$ with $\phi$ meeting \eqref{eq-intphi}. To remedy 
this, we consider the function $\phi_r:\C\to[0,+\infty)$ given by ($0<r\le1$) 
$$\phi_r(\xi)=\frac{1}{\pi r^2}\int_{\D(\xi,r)}\diff\hat\sigma_Q(\eta)
=\frac{1}{\pi r^2}\int_{\tau\in\D(0,r)}\diff\hat\sigma_Q(\xi-\tau);$$
this amounts to convolution with the normalized characteristic function of
the disk $\D(0,r)$. The corresponding measure 
$$\diff\sigma_r:=\diff\sigma_{\phi_r}=\phi_r\diff\vol_2$$
is a compactly supported (Borel) probability measure, with density
$\phi_r\in L^\infty(\C)$, so that \eqref{eq-intphi} holds with
$\phi_r$ in place of $\phi$. By the standard properties of convolutions, 
$\sigma_r\to\hat\sigma_Q$ in the weak-star sense of measures as $r\to0$. 
We claim that we also have convergence in energy,
\begin{equation}
I_Q[\sigma_{r}]\to I_Q[\hat\sigma_Q]=\gamma(Q)\quad\text{as}\,\,\,r\to0.
\label{eq-energyconv}
\end{equation}
Suppose for the moment that we have obtained \eqref{eq-energyconv}.
Then we find from the approximation procedure that
\begin{equation}
\liminf \frac{1}{n^2}\log Z_{m,n}\ge -\frac{\beta}{4}\,I_Q[\hat\sigma_Q]
=-\frac{\beta}{4}\,\gamma(Q).
\label{eq-estbelowstrong}
\end{equation}
To obtain \eqref{eq-energyconv}, we note that interchanging the order of 
integration gives
$$I_Q[\sigma_{r}]-I_Q[\hat\sigma_Q]=2\int_\C Q(\diff\sigma_r-\diff\hat\sigma_Q)
+\int_{\C^2}\Lambda_r(\xi,\eta)
\diff\hat\sigma_Q(\xi)\diff\hat\sigma_Q(\eta),$$
where 
$$\Lambda_r(\xi,\eta)=\frac{2}{\pi^2r^4}\int_{(\tau,\tau')\in\D(0,r)^2}
\big[\log\big|(\xi+\tau)-(\eta+\tau')\big|-\log|\xi-\eta|\big]
\diff\vol_2(\tau)\diff\vol_2(\tau').$$
The support of $\sigma_r$ is at most within distance $r$ from the support
$\setS_Q$ of $\hat\sigma_Q$, so in view of the assumption that $Q$ be bounded 
and continuous in a fixed neighborhood of $\setS_Q$, we get 
\begin{equation}
\int_\C Q(\diff\sigma_r-\diff\hat\sigma_Q)= 
\int_\C Q\diff\sigma_r-\int_\C Q\diff\hat\sigma_Q\to 0\quad\text{as}\,\,\,
r\to0.
\label{eq-conv}
\end{equation}
Next, we rewrite the expression for $\Lambda_r$:
\begin{multline*}
\Lambda_r(\xi,\eta)=\frac{2}{\pi^2r^4}\int_{(\tau,\tau')\in\D(0,r)^2}
\log\bigg|1+\frac{\tau-\tau'}{\xi-\eta}\bigg|
\diff\vol_2(\tau)\diff\vol_2(\tau')\\
=\frac{2}{\pi^2r^4}\int_{\tau''\in\D(0,2r)}
\vol_2\big(\D(0,r)\cap\D(\tau'',r)\big)\,
\log\bigg|1+\frac{\tau''}{\xi-\eta}\bigg|\diff\vol_2(\tau'').
\end{multline*}
We use that the common area of the two intersecting circular disks is
$$\vol_2\big(\D(0,r)\cap\D(\tau'',r)\big)=
2r^2\arccos\frac{|\tau''|}{2r}
-r|\tau''|\sqrt{1-\frac{|\tau''|^2}{4r^2}}$$
to get
\begin{equation*}
\Lambda_r(\xi,\eta)=\frac{4}{\pi^2r^2}\int_{\tau''\in\D(0,2r)}
\Bigg\{\arccos\frac{|\tau''|}{2r}
-\frac{|\tau''|}{2r}\sqrt{1-\frac{|\tau''|^2}{4r^2}}\Bigg\}
\log\bigg|1+\frac{\tau''}{\xi-\eta}\bigg|\diff\vol_2(\tau'').
\end{equation*}
The identity
$$\int_{-\pi}^{\pi}\log\big|1+\lambda \e^{\imag \theta}\big|\,\diff\theta
=2\pi\log^+|\lambda|,\qquad\lambda\in\C,$$
where for real $x\ge0$, $\log^+x=\max\{0,\log x\}$, shows that
$$\Lambda_r(\xi,\eta)=0,\qquad \text{if}\,\,\,2r\le |\xi-\eta|,$$
while
$$\Lambda_r(\xi,\eta)=\frac{8}{\pi r^2}\int_{|\xi-\eta|}^{2r}
\Bigg\{\arccos\frac{s}{2r}
-\frac{s}{2r}\sqrt{1-\frac{s^2}{4r^2}}\Bigg\}
\log\frac{s}{|\xi-\eta|}\,s\diff s\quad\text{if}\,\,\,|\xi-\eta|<2r.$$
In the latter case, we may use that for $|\xi-\eta|<s<2r$, 
\[
0\le\arccos\frac{s}{2r}-\frac{s}{2r}\sqrt{1-\frac{s^2}{4r^2}}\le\frac{\pi}2,
\quad 0\le\log\frac{s}{|\xi-\eta|}\le \log\frac{2r}{|\xi-\eta|},
\]
to conclude that
\begin{equation*}
0\le\Lambda_r(\xi,\eta)\le
8\log\frac{2r}{|\xi-\eta|}\quad\text{if}\,\,\,
|\xi-\eta|<2r.
\end{equation*}
It follows that generally, we have
\begin{equation}
0\le\Lambda_r(\xi,\eta)\le8\log^+\frac{2r}{|\xi-\eta|}.
\label{eq-Lr}
\end{equation}
The measure $\hat\sigma_Q$ has compact support and finite logarithmic energy,
\[
\int_{\C^2}\log\frac1{|\xi-\eta|}\diff\hat\sigma_Q(\xi)
\diff\hat\sigma_Q(\eta)<+\infty,
\]
so that if we use \eqref{eq-Lr} and the Lebesgue's domintated convergence 
theorem, we see that
$$\int_{\C^2}\Lambda_r(\xi,\eta)\,\diff\hat\sigma_Q(\xi)
\diff\hat\sigma_Q(\eta)\to0\quad\text{as}\,\,\, r\to0.$$
As we combine this with \eqref{eq-conv}, the claimed energy convergence 
\eqref{eq-energyconv} is immediate, and hence \eqref{eq-estbelowstrong} 
follows.

\subsection{Low probability of high energy configurations}
In view of \eqref{eq-estbelowstrong}, we have 
\begin{equation}
\frac{1}{n(n-1)}\log Z_{m,n}\ge -\frac{\beta}{4}\,(\gamma(Q)+\varepsilon),
\label{eq-estbelowstrong2}
\end{equation}
for fixed positive $\varepsilon$ and large enough $n$. In this context, 
we think of $m=m_n$ as (fixed) sequence which depends on $n$, with 
$m=m_n=n+\ordo(n)$.

We put
\begin{equation}
\label{eq-G}
G(\xi,\eta):=
\log\frac{(1+|\xi|^2)(1+|\eta|^2)}{|\xi-\eta|^2}\ge0,\qquad \xi,\eta\in\C.
\end{equation}
In view of the assumed extra growth \eqref{eq-Q1}, we have
\begin{multline}
L_Q(\xi,\eta)=\log\frac{1}{|\xi-\eta|^2}+Q(\xi)+Q(\eta)
\\
\ge
\log\frac{1}{|\xi-\eta|^2}+\frac{\delta_0}{1+\delta_0}[Q(\xi)+Q(\eta)]+
\log[(1+|\xi|^2)(1+|\eta|^2)]-\frac{2C_0}{1+\delta_0}
\\
=G(\xi,\eta)+\frac{\delta_0}{1+\delta_0}
[Q(\xi)+Q(\eta)]-\frac{2C_0}{1+\delta_0}
\label{eq-LQest}
\end{multline}
where $\delta_0$ and $C_0$ are as in \eqref{eq-Q1}.
To simplify the notation, we write, with $z=(z_1,\ldots,z_n)$, 
\[
L_Q^{\langle\!\langle n\rangle\!\rangle}(z)=
\sum_{j,k:j\ne k}L_Q(z_j,z_k),\quad\text{and}\quad
G^{\langle\!\langle n\rangle\!\rangle}(z)=\sum_{j,k:j\ne k}G(z_j,z_k)\ge0,
\]
where it is assumed that $j$ and $k$ range over $\{1,\ldots,n\}$. These
expressions are of ``double trace type'' associated with the
functions $L_Q$ and $G$ (see \eqref{eq-LQ} and \eqref{eq-G}). 
We also have the   ``trace type'' expressions (with $z=(z_1,\ldots,z_n)$)
\[
Q^{\langle n\rangle}(z)=\sum_j Q(z_j)\quad\text{and}\quad
\Lambda^{\langle n\rangle}(z):=
\sum_j\log(1+|z_j|^2).
\]
It now follows from \eqref{eq-LQest} that
\begin{equation}
\label{eq-LQest2}
L_Q^{\langle\!\langle n\rangle\!\rangle}(z)\ge
G^{\langle\!\langle n\rangle\!\rangle}(z)+\frac{2\delta_0(n-1)}{1+\delta_0}
Q^{\langle n\rangle}(z)-\frac{2C_0}{1+\delta_0}n(n-1),
\end{equation}
while a direct application of the extra growth condition 
\eqref{eq-Q1} leads to
\begin{equation}
\label{eq-exgrow}
Q^{\langle n\rangle}(z)\ge(1+\delta_0)\Lambda^{\langle n\rangle}(z)-C_0n.
\end{equation}
%These two estimates combine to yield that
%\begin{equation}
%\label{eq-LQest3}
%L_Q^{\langle\!\langle n\rangle\!\rangle}(z)\ge
%G^{\langle\!\langle n\rangle\!\rangle}(z)+2\delta_0(n-1)
%\Lambda^{\langle n\rangle}(z)-2C_0 n(n-1).
%\end{equation}

The point with introducing this notation is that
\eqref{eq-Z} simplifies to
\begin{equation}
Z_{m,n}=\int_{\C^n}\exp\Bigg\{-\frac{\beta}{4}
L_Q^{\langle\!\langle n\rangle\!\rangle}(z)+\frac{\beta}{2}
(n-m-1)Q^{\langle n\rangle}(z)\Bigg\}\diff\vol_{2n}(z),
\label{eq-Z'}
\end{equation}
while the probability density becomes
\begin{equation}
\diff\Pi_{mQ,n}(z)=\frac{1}{Z_{m,n}}\exp\Bigg\{-\frac{\beta}{4}
L_Q^{\langle\!\langle n\rangle\!\rangle}(z)+\frac{\beta}{2}
(n-m-1)Q^{\langle n\rangle}(z)\Bigg\}\diff\vol_{2n}(z),
\label{eq-probdens}
\end{equation}
As mentioned in Subsection \ref{subsec-3.1}, we have the estimate
\begin{equation}
\frac1{n(n-1)}L_Q^{\langle\!\langle n\rangle\!\rangle}(z)\ge\gamma(Q),
\qquad z=(z_1,\ldots,z_n)\in\C^n.
\label{eq-Fekete1}
\end{equation}
We introduce the set
\begin{equation}
\calA(n,\epsilon)=\bigg\{z\in\C^n:\,\frac{1}{n(n-1)}
L_Q^{\langle\!\langle n\rangle\!\rangle}(z)
\le\gamma(Q)+\epsilon\bigg\},
\label{eq-setA}
\end{equation}
where $\epsilon$ is a positive real number.

\begin{prop}
There exists a positive integer $N_0$, which depends on $\epsilon>0$ 
but not on $a\ge0$, such that 
$$\Pi_{mQ,n}(\C^n\setminus\calA(n,\epsilon+a))\le\e^{-\beta an(n-1)/8},\qquad
n\ge N_0,$$
provided the sequence $m=m_n=n+\ordo(n)$ is kept fixed.
\label{prop-3.1}
\end{prop}

\begin{proof}
By definition, we have
\begin{equation}
\frac{1}{n(n-1)}L_Q^{\langle\!\langle n\rangle\!\rangle}(z)
>\gamma(Q)+\epsilon+a,\qquad z\in\C^n\setminus\calA(n,\epsilon+a).
\label{eq-prop-1}
\end{equation}
We rewrite \eqref{eq-LQest2} as
\begin{equation}
\frac{1}{n(n-1)}\,L^{\langle\!\langle n\rangle\!\rangle}_Q(z)
\ge \frac{2\delta_0}{(1+\delta_0)n}Q^{\langle n\rangle}(z)-
\frac{2C_0}{1+\delta_0},\qquad z\in\C^n,
\label{eq-prop-2}
\end{equation}
and form a convex combination of \eqref{eq-prop-1} and \eqref{eq-prop-2} 
(we keep $\theta$ fixed with $0<\theta<1$)
\begin{equation}
\frac{1}{n(n-1)}\,L^{\langle\!\langle n\rangle\!\rangle}_Q(z)
\ge(1-\theta)(\gamma(Q)+\epsilon+a)
+\frac{\theta}{1+\delta_0}\bigg\{\frac{2\delta_0}{n}
Q^{\langle n\rangle}(z)-2C_0\bigg\},\quad z\in\C^n\setminus\calA(n,\epsilon+a).
\label{eq-prop-3}
\end{equation}
The exponent in the density defining $\Pi_{mQ,n}$ is (cf. \eqref{eq-Z'})
\[
-\frac{\beta}{4}\sum_{j,k:j\ne k}
L_Q(z_j,z_k)+\frac{\beta}{2}(n-m-1)\sum_j Q(z_j)
=-\frac{\beta}{4}L_Q^{\langle\!\langle n\rangle\!\rangle}(z)
+\frac{\beta}{2}(n-m-1)Q^{\langle n\rangle}(z),
\]
and in view of the estimate \eqref{eq-prop-3} we get
\begin{multline}
-\frac{\beta}{4}L_Q^{\langle\!\langle n\rangle\!\rangle}(z)
+\frac{\beta}{2}(n-m-1)Q^{\langle n\rangle}(z)
\le-\frac{\beta}{4} n(n-1)(1-\theta)(\gamma(Q)+\epsilon+a)
\\
-\frac{\beta}{2}\bigg\{\theta(n-1)\frac{\delta_0}{1+\delta_0}-(n-m-1)\bigg\}
Q^{\langle n\rangle}(z) +\frac{C_1\theta\beta}{4} n(n-1),
\qquad z\in\C^n\setminus\calA(n,\epsilon+a).
\label{eq-prop-4}
\end{multline}
If 
\begin{equation}
\frac{m}{n-1}>1-\frac{\theta\delta_0}{1+\delta_0},
\label{eq-prop-4.1}
\end{equation}
holds, which is bound to be the case for big enough $n$ (provided $\theta$
is kept away from $0$), since $m=n+\ordo(n)$,
the expression in front of $Q^{\langle n\rangle}(z)$ on the right hand side of
\eqref{eq-prop-4} is negative, and we may apply 
\eqref{eq-exgrow} to \eqref{eq-prop-4}, and arrive at
\begin{multline}
-\frac{\beta}{4}L_Q^{\langle\!\langle n\rangle\!\rangle}(z)
+\frac{\beta}{2}(n-m-1)Q^{\langle n\rangle}(z)
\\
\le-\frac{\beta}{4} n(n-1)(1-\theta)(\gamma(Q)+\epsilon+a)
-\frac{\beta}{2}\big\{\theta(n-1)\delta_0-(1+\delta_0)(n-m-1)\big\}
\Lambda^{\langle n\rangle}(z)
\\
+\frac{C_0\beta\theta}{2}n(n-1)-\frac{C_0\beta}{2} n(n-m-1),
\qquad z\in\C^n\setminus\calA(n,\epsilon+a).
\label{eq-prop-5}
\end{multline}
As a consequence, we find that
\begin{multline}
\Pi_{mQ,n}(\C^n\setminus\calA(n,\epsilon+a))
=\frac{1}{Z_{m,n}}\int_{\C^n\setminus\calA(n,\epsilon+a)}
\exp\Bigg\{-\frac{\beta}{4}L_Q^{\langle\!\langle n\rangle\!\rangle}(z)
+\frac{\beta}{2}(n-m-1)Q^{\langle n\rangle}(z)\Bigg\}\diff\vol_{2n}(z)
\\
\le\frac{1}{Z_{m,n}}\exp\Bigg\{
-\frac{\beta}{4} n(n-1)(1-\theta)(\gamma(Q)+\epsilon+a)
+\frac{C_0\beta\theta}{2}n(n-1)
-\frac{C_0\beta}{2} n(n-m-1)\Bigg\}
\\
\times\Bigg\{\int_\C(1+|\xi|^2)^{-\frac{\beta}{2}\{\theta(n-1)\delta_0
-(1+\delta_0)(n-m-1)\}}\diff\vol_2(\xi)\Bigg\}^n.
\label{eq-prop-6}
\end{multline}
An exercise involving polar coordinates convinces us that for $\alpha>1$, 
\begin{equation}
\int_\C(1+|\xi|^2)^{-\alpha}\diff\vol_2(\xi)=\frac{\pi}{\alpha-1},
\label{eq-polcoord}
\end{equation}
and we see that \eqref{eq-prop-6} entails that
\begin{multline}
\Pi_{mQ,n}(\C^n\setminus\calA(n,\epsilon+a))
\\
\le\frac{1}{Z_{m,n}}\exp\bigg\{
-\frac{\beta}{4} n(n-1)(1-\theta)(\gamma(Q)+\epsilon+a)
+\frac{C_0\beta\theta}{2}n(n-1)
-\frac{C_0\beta}{2} n(n-m-1)\bigg\}
\\
\times\bigg\{\frac{2\pi}{\beta\{\theta(n-1)\delta_0
-(1+\delta_0)(n-m-1)\}-2}\bigg\}^n,
\label{eq-prop-7}
\end{multline}
provided that 
\begin{equation*}
\frac{m}{n-1}>1-\frac{\theta\delta_0}{1+\delta_0}
+\frac{2}{\beta(1+\delta_0)(n-1)}.
\end{equation*}
Let us assume slightly more, namely that 
\begin{equation}
\frac{m}{n-1}>1-\frac{\theta\delta_0}{1+\delta_0}
+\frac{2(1+\pi)}{\beta(1+\delta_0)(n-1)},
\label{eq-prop-8}
\end{equation}
which is a little stronger than \eqref{eq-prop-4.1}, and holds for big enough 
$n$ (as long as $\theta$ is kept away from $0$), since $m=n+\ordo(n)$.
This allows us to get rid of the last factor in the right hand side of 
\eqref{eq-prop-7}:
\begin{multline}
\Pi_{mQ,n}(\C^n\setminus\calA(n,\epsilon+a))
\\
\le\frac{1}{Z_{m,n}}\exp\bigg\{
-\frac{\beta}{4}n(n-1)(1-\theta)(\gamma(Q)+\epsilon+a)
+\frac{C_0\beta\theta}{2}n(n-1)
-\frac{C_0\beta}{2} n(n-m-1)\bigg\}.
\label{eq-prop-9}
\end{multline}
We finally implement the estimate \eqref{eq-estbelowstrong2}, and get
\begin{multline}
\Pi_{mQ,n}(\C^n\setminus\calA(n,\epsilon+a))
\\
\le\exp\bigg\{\frac{\beta}{4}n(n-1)\bigg[\theta\gamma(Q)
-(1-\theta)(\epsilon+a)+\varepsilon
+2\theta C_0
-2C_0\bigg(1-\frac{m}{n-1}\bigg)\bigg]\bigg\}.
\label{eq-prop-10}
\end{multline}
The constant $C_0$ is assumed positive, and we may therefore 
pick a small $\theta$, $0<\theta<\frac12$, such that
\[
\theta[\gamma(Q)+2C_0]\le \frac{\epsilon}{2}.
\]
Since $m=n+\ordo(n)$, it follows from \eqref{eq-prop-10} that
\begin{equation}
\Pi_{mQ,n}(\C^n\setminus\calA(n,\epsilon+a))
\le\exp\bigg\{\frac{\beta}{4}n(n-1)\bigg[-(1-\theta)a
-\bigg(\frac12-\theta\bigg)\epsilon+\varepsilon
+\ordo(1)\bigg)\bigg]\bigg\}.
\label{eq-prop-11}
\end{equation}
Also, by choosing $\varepsilon$ sufficiently small, we can make sure that
\[
-\Big(\frac12-\theta\Big)\epsilon+\varepsilon
+\ordo(1)\le0
\]
for big $n$, so that \eqref{eq-prop-11} gives
\begin{equation}
\Pi_{mQ,n}(\C^n\setminus\calA(n,\epsilon+a))
\le\exp\bigg\{-\frac{\beta}{4}(1-\theta)a n(n-1)\bigg\}\le
\exp\bigg\{-\frac{\beta}{8}an(n-1)\bigg\},
\label{eq-prop-12}
\end{equation}
as claimed.
\end{proof}

\subsection{The proof of Johansson's free energy theorem}
%In view of \eqref{eq-estbelowstrong}, we have 
%\begin{equation}
%\frac{1}{n(n-1)}\log Z_{m,n}\ge -\frac{\beta}{4}\,\big(\gamma(Q)+\varepsilon
%\big),
%\label{eq-estbelowstrong2}
%\end{equation}
%for fixed positive $\varepsilon$ and large enough $n$. In this context, 
%we think of $m=m_n$ as (fixed) sequence which depends on $n$, with 
%$m=m_n=n+\ordo(n)$.
The claim is that
\begin{equation}
\frac{1}{n(n-1)}\log Z_{m,n}\to -\frac{\beta}{4}\,\gamma(Q)\quad
\text{as}\,\,\,n\to+\infty\,\,\,\text{while}\,\,\,m=n+\ordo(n).
\label{eq-limit}
\end{equation}
Note that by \eqref{eq-estbelowstrong} we only need to show that
$\limsup$ converges to a number $\le-\beta\gamma(Q)/4$. To this end, 
we begin by establishing that for $0<\theta<1$, we have
\begin{multline}
-\frac{1}{2}L^{\langle\!\langle n\rangle\!\rangle}_Q(z)
+(n-m-1)Q^{\langle n\rangle}(z)
\\
\le -\frac{1-\theta}{2}n(n-1)\gamma(Q)-
\bigg[(n-1)\frac{\delta_0\theta}{1+\delta_0}-
(n-m-1)\bigg]Q^{\langle n\rangle}(z)+\frac{C_1\theta}{2}n(n-1),
\label{eq-LQ1}
\end{multline}
by forming a convex combination of \eqref{eq-Fekete1} and \eqref{eq-prop-2}.
By applying \eqref{eq-exgrow} to \eqref{eq-LQ1}, we get that 
(since the expression in front of $Q^{\langle n\rangle}(z)$ is negative for 
big $m,n$ with $m=n+\ordo(n)$) 
\begin{multline}
-\frac{1}{2}L^{\langle\!\langle n\rangle\!\rangle}_Q(z)
+(n-m-1)Q^{\langle n\rangle}(z)
\\
\le -\frac{1-\theta}{2}n(n-1)\gamma(Q)-\bigg[(n-1)\delta_0\theta-
(1+\delta_0)(n-m-1)\bigg]\bigg(\Lambda^{\langle n\rangle}(z)
-\frac{C_0n}{1+\delta_0}\bigg)
+\frac{C_1\theta}{2}n(n-1)
\\
=-\frac{1-\theta}{2}n(n-1)\gamma(Q)-\big[(n-1)\delta_0\theta-
(1+\delta_0)(n-m-1)\big]\Lambda^{\langle n\rangle}(z)
+C_0\theta n(n-1)-C_0 n(n-m-1).
\label{eq-LQ2}
\end{multline}
We multiply by $\beta/2$ on the left and right hand sides, to get
\begin{multline}
Z_{m,n}=\int_{\C^n}\exp\bigg\{
-\frac{\beta}{4}L^{\langle\!\langle n\rangle\!\rangle}_Q(z)
+\frac{\beta}{2}(n-m-1)Q^{\langle n\rangle}(z)\bigg\}\diff\vol_{2n}(z)
\\
\le\e^{-\frac{\beta}{4}(1-\theta)n(n-1)\gamma(Q)
+\frac{\beta}{2}C_0\theta n(n-1)-C_0 n(n-m-1)}\bigg\{
\int_\C(1+|\xi|^2)^{-\frac{\beta}{2}[\theta\delta_0(n-1)-(1+\delta_0)(n-m-1)]}
\diff\vol_{2}(\xi)\bigg\}^n,
\label{eq-LQ3}
\end{multline}
so that in view of \eqref{eq-polcoord}, we have
\begin{equation*}
Z_{m,n}
\le\exp\bigg\{-\frac{\beta}{4}(1-\theta)\gamma(Q)\,n(n-1)
+\frac{C_0\beta}{2}\theta n(n-1)-C_0 n(n-m-1)\bigg\},
\end{equation*}
provided \eqref{eq-prop-8} is assumed. Taking logarithms, we find that
\begin{equation*}
\frac{1}{n(n-1)}\log Z_{m,n}
\le-\frac{\beta}{4}(1-\theta)\gamma(Q)
+\frac{C_0\beta}{2}\theta-C_0\bigg(1-\frac{m}{n-1}\bigg),
\end{equation*}
for big enough $m,n$ with $m=n+\ordo(n)$, since \eqref{eq-prop-8} is
fulfilled then. As $\theta$, $0<\theta<1$, can be taken as close to $0$ 
as we like, it follows that 
\begin{equation*}
\limsup\frac{1}{n(n-1)}\log Z_{m,n}
\le-\frac{\beta}{4}\gamma(Q).
\end{equation*}
The claim is an immediate consequence.
\qed

\subsection{The proof of Johansson's marginal probability theorem}
%By \eqref{eq-LQest3}, we have
%\begin{equation*}
%\frac{1}{n(n-1)}
%G^{\langle\!\langle n\rangle\!\rangle}(z)
%+\frac{2\delta_0}{n}\Lambda^{\langle n\rangle}(z)-2C_0
%\le \frac{1}{n(n-1)}
%L^{\langle\!\langle n\rangle\!\rangle}_Q(z),
%\end{equation*}
%so that by the definition \eqref{eq-setA} of the set $\calA(n,\epsilon)$,
%with $\epsilon=1$, we have
%\begin{equation}
%\label{eq-est1-1}
%\frac{1}{n(n-1)}
%G^{\langle\!\langle n\rangle\!\rangle}(z)
%+\frac{2\delta_0}{n}\Lambda^{\langle n\rangle}(z)
%\le\gamma(Q)+2C_0+1,\qquad z\in\calA(n,1).
%\end{equation}
For a positive real $R$ (a radius), we put
\[
n_R(z)=\sharp\big\{j\in\{1,\ldots,n\}:\,|z_j|\ge R\big\},
\]
where $\sharp$ counts the number of elements, and $z=(z_1,\ldots,z_n)$, as 
before. 
%We clearly have
%\[
%\Lambda^{\langle n\rangle}(z)\ge n_R(z)\log(1+R^2),
%\]
%and so \eqref{eq-est1-1} gives us
%\begin{equation}
%\label{eq-est1-2}
%\frac{2\delta_0}{n}n_R(z)\log(1+R^2)
%\le\gamma(Q)+2C_0+1,\qquad z\in\calA(n,1).
%\end{equation}
We let $R_0$ be a positive real with
\begin{equation}
\label{eq-R0}
\delta_0\log(1+R_0^2)\ge\gamma(Q)+2C_0+1.
\end{equation}

\begin{prop}
We have the estimate
$$\frac{n_{R_0}(z)}{n}\le\epsilon,\qquad z\in\calA(n,\epsilon).$$
\label{prop-eps}
\end{prop}

\begin{proof}
We split the integer interval:
\[
\{1,\ldots,n\}={\mathfrak n}(z,R_0)\cup{\mathfrak m}(z,R_0),\quad
{\mathfrak n}(z,R_0)\cap{\mathfrak m}(z,R_0)=\emptyset,
\]
where 
\[
{\mathfrak n}(z,R_0)=\big\{j\in\{1,\ldots,n\}:\,|z_j|\ge R\big\},
\]
so that $n_{R_0}(z)=\sharp[{\mathfrak n}(z,R_0)]$. We split the sum defining
$L^{\langle\!\langle n\rangle\!\rangle}(z)$ accordingly (we use the symmetry
$L_Q(\xi,\eta)=L_Q(\eta,\xi)$):
\begin{multline*}
L^{\langle\!\langle n\rangle\!\rangle}(z)=
L^{\mathrm I}_Q(z)+2L^{\mathrm{II}}_Q(z)+L^{\mathrm{III}}_Q(z)
\\=
\sum_{j,k\in{\mathfrak m}(z,R_0):\,j\neq k}L_Q(z_j,z_k)
+2\sum_{j\in{\mathfrak m}(z,R_0),\,k\in {\mathfrak n}(z,R_0)}
L_Q(z_j,z_k)
+\sum_{j,k\in{\mathfrak n}(z,R_0):\,j\neq k}L_Q(z_j,z_k),
\end{multline*} 
with the obvious interpretation of $L^{\mathrm I}_Q(z)$, 
$L^{\mathrm{II}}_Q(z)$, and $L^{\mathrm{III}}_Q(z)$. From the extra growth 
condition \eqref{eq-Q1}, we see that
\begin{equation*}
L_Q(\xi,\eta)\ge G(\xi,\eta)+\delta_0\log[(1+|\xi|^2)(1+|\eta|^2)]-2C_0,
\end{equation*} 
so that by \eqref{eq-R0}, 
\begin{equation*}
L_Q(\xi,\eta)\ge\delta_0\log[(1+R_0^2)]-2C_0\ge\gamma(Q)+1,\quad\text{if}
\quad |\eta|\ge R_0.
\end{equation*} 
%while
%\begin{equation*}
%L_Q(\xi,\eta)\ge2\delta_0\log[(1+R_0^2)]-2C_0\ge2[\gamma(Q)+1],\quad\text{if}
%\quad |\xi|,|\eta|\ge R_0.
%\end{equation*}
This allows us to conclude that
\[
2L^{\mathrm{II}}_Q(z)+L^{\mathrm{III}}_Q(z)\ge 2n_{R_0}(z)[n-n_{R_0}(z)]
(\gamma(Q)+1)+n_{R_0}(z)[n_{R_0}(z)-1](\gamma(Q)+1),
\]
As regards the term $L^{\mathrm I}_Q(z)$, we may apply \eqref{eq-Fekete1} 
to the
remaining $(n-n_{R_0}(z))$-tuple:
\[
L^{\mathrm I}_Q(z)\ge (n-n_{R_0}(z))(n-n_{R_0}(z)-1)\gamma(Q).
\]
By adding up the terms, we find that
\begin{equation*}
L^{\langle\!\langle n\rangle\!\rangle}_Q(z)=
L^{\mathrm I}_Q(z)+2L^{\mathrm{II}}_Q(z)+L^{\mathrm{III}}_Q(z)
\ge n(n-1)\gamma(Q)+(n-1)n_{R_0}(z).
\end{equation*} 
For $z\in\calA(n,\epsilon)$, we then get
\begin{equation*}
\gamma(Q)+\frac{n_{R_0}(z)}{n}\le
\frac{1}{n(n-1)}L^{\langle\!\langle n\rangle\!\rangle}(z)\le
\gamma(Q)+\epsilon,
\end{equation*} 
from which the assertion is immediate. 
\end{proof}

For a point $z\in\C^n$, we define the associated weighted
sum of point masses $\sigma_z\in\calP_c(\C)$ by the 
formula
\begin{equation}
\diff\sigma_z(\xi)=\frac1{n}\sum_{j=1}^n\diff\delta_{z_j}(\xi),
\qquad \xi\in\C,
\label{eq-sigmaN}
\end{equation}
where $\delta_w$ means the Dirac point mass at $w\in\C$. Also, let
$C_b(\C)=C(\C)\cap L^\infty(\C)$ denote the space of bounded complex-valued
continuous functions on $\C$. 

\begin{prop}
Suppose $\sigma_n=\sigma_z$ is as above, with $z=(z_1,\ldots,z_n)\in\C^n$.
Suppose, moreover, that 
\[
I^\sharp_Q[\sigma_n]=\frac{1}{n(n-1)}L^{\langle\!\langle n\rangle\!\rangle}
(z)\to\gamma(Q)
\]
as $n\to+\infty$. Then, as $n\to+\infty$, we have $\sigma_n\to\hat\sigma_Q$ 
weakly-star. In other words, for each $f\in C_b(\C)$, we have
$$\int_\C f\,\diff\sigma_n\to \int_\C f\,\diff\hat\sigma_Q\quad
\text{as}\quad n\to+\infty.$$
\label{prop-limitmeas}
\end{prop}

\begin{proof}
The proof is standard. We choose a weakly-star convergent subsequence, and
call the limit $\sigma^*$. From the assumptions on the probaility measure
$\sigma_n$, we find that almost all its mass is concentrated to a fixed compact
subset of $\C$ (cf. Proposition \ref{prop-limitmeas}), and that 
$I_Q[\sigma^*]\le I_Q[\hat\sigma_Q]=\gamma(Q)$, by considering a cut-off of the
logarithmic kernel. We leave the details to the interested reader.
\end{proof}

Let $\omega:\{1,\ldots,n\}\to\{1,\ldots,n\}$ be a permutation. For $z=(z_1,
\ldots,z_n)\in\C^n$ we let $z^\omega=(z_{\omega(1)},\ldots,z_{\omega(n)})\in
\C^n$ be point induced by the permutation. Suppose for the moment that
$f\in C_b(\C^n)$, and write $f^\omega(z)=f(z^\omega)$.
By symmetry, we then have
\[
\int_{\C^n}f\,\diff\Pi_{mQ,n}=\int_{\C^n}f^\omega\,\diff\Pi_{mQ,n},
\]
which gives
\begin{equation*}
\int_{\C^n}f\,\diff\Pi_{mQ,n}
=\frac1{n!}
\int_{\C^n}\sum_\omega f^\omega\diff\Pi_{mQ,n},
\end{equation*}
where the sum runs over all permutations $\omega$. 
We next split the integral:
\begin{equation}
\int_{\C^n}f\diff\Pi_{mQ,n}
=\frac{1}{n!}\int_{\calA(n,\epsilon)}\sum_\omega 
f^\omega\diff\Pi_{mQ,n}
+\frac{1}{n!}\int_{\C^n\sm\calA(n,\epsilon)}\sum_\varsigma
f^\omega\diff\Pi_{mQ,n}.
\label{eq-splitint}
\end{equation}
By Propositions \ref{prop-3.1} and \ref{prop-eps}, the last term is 
$\ordo(1)$ as $m,n\to+\infty$ while $m=n+\ordo(n)$.
In order to understand the remaining term, we should study
\begin{equation}
\frac{1}{n!}
\sum_\omega f^\omega \quad\text{on}\,\,\,\calA(n,\epsilon).
\label{eq-sum}
\end{equation}
We now focus on the $k=1$ case of Johansson's theorem, and restrict our 
attention to $f$ which only depend on the first coordinate, $f(z)=f(z_1)$ 
with some slight abuse of notation.
Then (\ref{eq-sum}) amounts to the linear statistic
\begin{equation}
\frac1{n}\sum_{j=1}^n f(z_{j}),\qquad(z_1,\ldots,z_n)\in\calA(n,\epsilon).
\label{eq-sum2}
\end{equation}
By Proposition \ref{prop-eps}, only an $\epsilon$ proportion of the points
$z_j$ may fall outside the disk $\D(0,R_0)$, and by Proposition 
\ref{prop-limitmeas}, the expression \eqref{eq-sum2} is close to 
(the constant!)
\[
\int_\C f\diff\hat\sigma_Q
\]
for small $\epsilon$ and large $n$. The weak-star convergence 
$\Pi^{(1)}_{mQ,n}\to\hat\sigma_Q$ follows, if we let $\epsilon$ approach $0$ 
slowly as $n\to+\infty$.
The remaining case $k>1$ is analogous. 
\qed

\section{An obstacle problem. Smooth potentials}
\label{sec-obst}

\subsection{Equilibrium measure in terms of an obstacle problem} 
\label{subsec-4.1}
We consider the cone $\Sub(\C)$ of all subharmonic functions in the
plane $\C$, and its convex subset ($0<t<+\infty$ is assumed fixed)
\[
\Sub_t(\C):=\bigg\{v\in\Sub(\C):
~\limsup_{|z|\to+\infty}~[v(z)-t\log|z|^2]<+\infty  \bigg\}.
\]
Given $Q:\C\to\R\cup\{+\infty\}$, the obstacle problem is to find
\begin{equation}
\Obs_t[Q](z):=\sup\big\{v(z):\,\,v\in\Sub_t(\C)\,\text{and}
~v\le Q\,\,\text{on}\,\,\C\big\}.
\label{eq-4.1}
\end{equation}
Here, we assume of $Q$ -- as before -- that it is lower semi-continuous,
bounded on a set of positive area, and 
that
\begin{equation}
\lim_{|z|\to+\infty}~[Q(z)-t\log|z|^2]=+\infty.
\label{eq-4.2}
\end{equation}
We think of both $Q$ and $t$ as fixed; we observe, however, that if 
\eqref{eq-4.2} is fulfilled for one value of $t$, then any smaller positive 
value works as well.
It is easy to check that the supremum in \eqref{eq-4.1} is taken over a 
non-empty collection of functions $v$ (e.g., a large negative constant
will satisfy the requirements). See e.g. Doob \cite{Doob} for the potential
theory pertaining to obstacle problems of this type. For instance, after 
possibly redefining the function $\Obs_t[Q]$ on a negligible set (here, this
is a set of logarithmic capacity $0$), we get a subharmonic function.  
We need to connect the obstacle problem \eqref{eq-4.1} with the equilibrium
measure theory of Subsections \ref{sub-equil} and \ref{sub2.9}. 
To this end, let 
\[
\hat\sigma_t=\hat\sigma_t[Q]:=t\hat\sigma_{Q/t},\qquad S_t=S_t[Q]:=S_{Q/t}=
\supp\hat\sigma_t,
\]
 be the scaled equilibrium measure of Subsection \ref{sub2.9} and its 
associated support set. We write
\[
\gamma_t(Q)=t\gamma(Q/t)\quad\text{and}\quad
\gamma^*_t(Q)=\gamma_t(Q)-\frac1t\int_\C Q\diff\hat\sigma_t.
\]
For a compactly supported finite positive Borel measure $\sigma$, let 
$U^{\sigma}$ denote the logarithmic potential
\[
U^{\sigma}(\xi)=
\int_\C\log\frac{1}{|\xi-\eta|^2}\diff\sigma(\eta),
\]
and put
\begin{equation}
\label{eq-Qt}
{\widehat Q}_t(\xi)=\gamma^*_t(Q)-U^{\hat\sigma_t}(\xi).
\end{equation}
The function ${\widehat Q}_t$ is then subharmonic in $\C$, and harmonic in 
$\C\setminus S_t$, where $S_t=\supp\hat\sigma_t$. 
Moreover, as it is the total mass of the measure which 
determines the decay of the logarithmic potential at infinity, we have
\begin{equation}
{\widehat Q}_t(z)=t\log|z|^2+\Ordo(1)\quad\text{as}\,\,\,|z|\to+\infty.
\label{eq-4.3}
\end{equation}

The following lemma supplies a criterion which allows us to solve the 
obstacle problem. We recall that the logarithmic energy $I_0[\sigma]$ is given 
by \eqref{eq-intQ} with $Q$ replaced by $0$.

\begin{lem} Let $\sigma$ be a compactly supported finite positive Borel 
measure in $\C$ of finite logarithmic energy $I_0[\sigma]<+\infty$ with 
total mass $\|\sigma\|=t$. Suppose $W=c-U^\sigma$, 
where $c\in\R$ is a constant. If $W$ has both $W\le Q$ q.e. on $\C$ and 
$W=Q$ q.e. on $\supp\sigma$, then $W=\Obs_t[Q]$ q.e.
\label{lm-4.1}
\end{lem}

\begin{proof} 
Without loss of generality, we may assume that $Q\ge1$ on $\C$.
The function $W$ is in $\Sub_t(\C)$ while $W\le Q$ q.e. on $\C$. So $W\le Q$
on $\C\setminus E$, where $E\subset \C$, is polar (i.e., has logarithmic 
capacity $0$). Let $\rho$ be a compactly supported Borel measure on $\C$
such that the corresponding potential has $U^\rho=+\infty$ on $E$ (see, e.g.
\cite{Doob}). 
Put $W':=(1-\epsilon)W-\epsilon'U^\varrho$ where $\epsilon,\epsilon'$ are 
two small positive numbers. If $\epsilon'$ is very small (also relative to 
$\epsilon$), we can make sure that $W'\le Q$ on $\C\setminus E$, by using
that $W\le Q$ on $\C\setminus E$, the standard properties of potentials,
and the given properties of $Q$. 
Then $W'\le Q$ throughout $\C$ automatically 
as $W'=-\infty$ on $E$. 
We conclude that $W'\le\Obs_t[Q]$ on $\C$. By letting $\epsilon'\to0$ first 
and then $\epsilon\to0$ second, we get that $W\le\Obs_t[Q]$ on 
$\C\setminus E$, and consequently, $W\le\Obs_t[Q]$ q.e. 
on $\C$. As a side remark, we observe that $W$ is locally of Sobolev class
$W^{1,2}$, which means that in the sense of distributions, its gradient is 
locally in $L^2$ with respect to area measure.
It remains to show the reverse inequality $\Obs_t[Q]\le W$ q.e. 
on $\C$. To this end, we pick a function $v\in\Sub_t(\C)$ with $v\le Q$ 
q.e. on $\C$. It will suffice to show that $v\le W$ on $\C$.
The potential $U^\sigma$ is harmonic in $\C\setminus S$, where 
$S:=\supp\sigma$, and therefore $W$ is harmonic there as well. The assumption 
on the total mass of $\sigma$ gives that
\begin{equation}
W(z)=t\log|z|^2+\Ordo(1)\quad\text{as}\,\,\,|z|\to+\infty.
\label{eq-4.4}
\end{equation}
Next, we consider the difference $u=v-W$, which is subharmonic in 
$\C\setminus S$ and has $u\le0$ q.e. on $S$. Moreover, the assumption that 
$v\in\Sub_t(\C)$ together with \eqref{eq-4.4} shows that $u$ is bounded from
above near infinity. We should like to apply the maximum principle 
in the open set $\C\setminus S$ and obtain that $u\le0$ on $\C\setminus S$
since $u\le0$ q.e. on the boundary. However, this is a little delicate as
the functions are not necessarily continuous up to the boundary. 
The so-called Principle of Domination \cite{ST}, p. 104,
is a good substitute. To apply it, we need to make the technical assumption 
that $v$ is harmonic in a punctured neighborhood of infinity, because it 
allows us to represent $v$ q.e. in the form $b-U^\nu$, where $b$ is a constant 
and $\nu$ is finite compactly supported positive Borel measure. The assumption
$v\in\Sub_t(\C)$ then gives that $\nu$ has total mass $\|\nu\|\le t$, so
that $\|\nu\|\le\|\sigma\|$. From the
assumptions we read off that $U^\sigma\le U^\nu+c-b$ holds  q.e. on 
$S=\supp\sigma$, and hence $\sigma$-a.e. (because $\sigma$ has finite 
logarithmic energy), so by the  Principle of Domination (which again uses 
that $\sigma$ has finite logarithmic energy, and that 
$\|\nu\|\le\|\sigma\|$), we find that 
$U^\sigma\le U^\nu+c-b$ holds throughout 
$\C$. The desired conclusion that $v\le W$ follows. Next, to justify the
conclusion that $v\le W$ on $\C$ holds when we only assume that  
$v\in\Sub_t(\C)$ with $v\le Q$ q.e. on $\C$, we proceed as follows.
%First, we observe that $Q$ is bounded below, so there is no loss of 
%generality to assume that $Q\ge0$ on $\C$. 
If we let $\varepsilon$ be a small positive real number, and put
\[
\tilde v(z):=\max\bigg\{\frac{v(z)}{1+\varepsilon},t\log|z|^2-C\bigg\},
\]
then $\tilde v\in\Sub_t(\C)$ is harmonic in a punctured neighborhood of
the point at infinity, and $\tilde v\le Q$ q.e. on $\C$ holds if the
constant $C$ is big enough positive. So we have the conclusion $\tilde v\le W$
on $\C$ from the previous argument. Finally, we first let $C\to+\infty$
and afterwards let $\varepsilon\to0$, and obtain $v\le W$, as claimed.
\end{proof}

We have the following characterization.

\begin{prop}
We have
\[
\Obs_t[Q](z)={\widehat Q}_t(z),\qquad \text{q.e.--}\,\,z\in\C.
\]
\label{prop-4.2}
\end{prop}

\begin{proof}
In view of Lemma \ref{lm-4.1}, we just need to check that ${\widehat Q}_t=Q$
q.e. on $S_t$ while ${\widehat Q}_t\le Q$ q.e. on $\C$. By Frostman's Theorem
\ref{thm-frost3}, we have $U^{\hat\sigma_{Q/t}}+Q/t\ge\gamma^*(Q/t)$
q.e. on $\C$, while $U^{\hat\sigma_{Q/t}}+Q/t=\gamma^*(Q/t)$ q.e. on $S_{Q/t}$.
Since $\hat\sigma_t=t\hat\sigma_{Q/t}$, $S_t=S_{Q/t}$, and $\gamma_t^*(Q)=
t\gamma^*(Q/t)$, this means that $U^{\hat\sigma_{t}}+Q\ge\gamma_t^*(Q)$
q.e. on $\C$, while $U^{\hat\sigma_{t}}+Q=\gamma_t^*(Q)$ q.e. on $S_t$.
With $\widehat Q_t=\gamma_t^*(Q)-U^{\hat\sigma_{t}}$, this is the same as 
having ${\widehat Q}_t\le Q$  q.e. on $\C$, while ${\widehat Q}_t=Q$ q.e. on 
$S_t$, as needed.
\end{proof}

\begin{rem}
The assertion of Proposition \ref{prop-4.2} is essentially equivalent to that 
of Theorem I.4.1 \cite{ST}.
\end{rem}

We easily recover the density from the potential; we write $\diff A:=\pi^{-1}
\diff\vol_2$ for normalized area measure. 

\begin{cor} We have, in the sense of distribution theory, 
$\diff\hat\sigma_t=\Delta{\widehat Q}_t\diff A$.
\label{cor-4.4}
\end{cor}

\subsection{The super-coincidence and coincidence sets}
We keep the setting of the previous subsection, and assume 
$Q:\C\to\R\cup\{+\infty\}$ is lower semi-continuous, and bounded on a set
of positive area, subject to the growth condition \eqref{eq-4.2}.
The potential $U^\sigma$ is superharmonic for a given finite positive
compactly supported measure $\sigma$, and therefore the function 
$\widehat Q_t$ defined by 
\eqref{eq-4.3} is automatically subharmonic. In particular, $\widehat Q_t$
is upper semi-continuous, and we find that the difference $Q-\widehat Q_t$
is lower semi-continuous. 
It follows that the {\em super-coincidence set}
\begin{equation}
\label{eq-St}
S^*_t=S_t^*[Q]=\big\{z\in\C:\,\widehat Q_t(z)\ge Q(z)\big\}
\end{equation}
is compact: it is closed by semi-continuity, while 
\eqref{eq-4.3} and \eqref{eq-4.4} show that it is bounded. We note that 
by Proposition \ref{prop-4.2}, $\widehat Q_t\le Q$ quasi-everywhere, so that
$S^*_t$ equals, up to a set of logarithmic capacity zero, the {\em coincidence 
set}
\[
\big\{z\in\C:\,\widehat Q_t(z)= Q(z)\big\}.
\]
In all cases when we have a little regularity, $\widehat Q_t$ is continuous, 
and then the super-coincidence set $S_t^*$ is the same as the coincidence set.
We therefore refrain from introducing separate notation for the coincidence
set.

\begin{prop} 
The function $\widehat Q_t$ is harmonic in $\C\sm S^*_t$, and as a consequence,
$S_t\subset S_t^*$. In particular, $S_t^*$ is non-empty.
\label{prop-4.5}
\end{prop}

\begin{proof}
We pick a point $z_0\in\C\setminus S_t^*$, so that $\widehat Q_t(z_0)<Q(z_0)$. 
By semi-continuity, we get that $\widehat Q_t<Q$ in a neighborhood of $z_0$.
We claim that $\widehat Q_t$ is harmonic near $z_0$. If not, we could use
Perron's lemma and replace $\Obs_t[Q]$ (which equals $\widehat Q_t$ q.e., by
Proposition \ref{prop-4.2}) on a small disk around $z_0$ by  
the harmonic function which has the same boundary values, and get a function
which is in $\Sub_t(\C)$, and bigger than $\Obs_t[Q]$ while being $\le Q$.
This violates the extremality of $\Obs_t[Q]$, and the claim follows.
Next, by Corollary \ref{cor-4.4}, we see that $z_0\in\C\setminus S_t$. Since 
$z_0$ was an arbitrary point in $\C\setminus S_t^*$, the proof is complete.
\end{proof}

\subsection{A priori smoothness for the obstacle problem for smooth 
potentials}
\label{subsec-smooth}
As before, $Q:\C\to\R\cup\{+\infty\}$ is lower semi-continuous with 
\eqref{eq-4.2} where $t$ is a (fixed) positive real.
If $Q$ has some degree of smoothness, say, e.g., $Q:\C\to\R$ is $C^2$-smooth, 
it is natural to wonder to what extent that carries over to $\Obs_t[Q]=
{\widehat Q}_t$. Since we sometimes need to work with slightly less smooth
weights, the (local) Sobolev classes $W^{2,p}$ are sometimes more appropriate. 
By Sobolev imbedding, functions in $W^{2,p}$ are continuous provided 
$1<p<+\infty$.
The case $p=+\infty$ gives the space $W^{2,\infty}=C^{1,1}$ of functions
whose first order partial derivatives are locally Lipschitz continuous. 
We use the notation $W^{2,p}$ and $C^{1,1}$ for {\em local} classes unless
otherwise stated.  

The following a priori smoothness result is standard in connection with 
the constrained obstacle problem discussed below \cite{Fried}, and associated
with the names such as Lewy, Stampacchia, Brezis, Lions, Kinderlehrer, and
Caffarelli.
We present the elementary approach recently found by Berman \cite{Ber}, which
gives the $C^{1,1}$-smoothness part. Berman's approach also applies in the 
several complex variables context.

\begin{prop}
If $Q\in C^2$, then $\Obs_t[Q]\in C^{1,1}$. More generally, if 
$Q\in C^{1,1}$, we still have $\Obs_t[Q]\in C^{1,1}$. Finally, if 
$Q\in W^{2,p}$ for some $p$, $1<p<+\infty$, then $\Obs_t[Q]\in W^{2,p}$.
\label{prop-4.6}
\end{prop} 

\begin{proof}
We first show how $Q\in C^{1,1}$ implies that $\Obs_t[Q]\in C^{1,1}$.
We begin by noting that by \eqref{eq-4.3} and Proposition \ref{prop-4.2},
\[
\Obs_t[Q](z)={\widehat Q}_t(z)\le t\log(1+|z|^2)+C,\qquad z\in\C,
\]
for a suitable real constant $C$. 
%As a consequence, we get
%\begin{multline*}
%\Obs_t[Q](z+h)\le \frac{t}{2}\log(1+|z+h|^2)+C=\frac{t}{2}\log(1+|z|^2)
%+\frac{t}{2}\log\bigg(1+\frac{2\re\bar zh+|h|^2}{1+|z|^2}\bigg)+C
%\\
%\le\frac{t}{2}\log(1+|z|^2)+\frac{t}{2}\,\frac{2\re\bar zh+|h|^2}{1+|z|^2}+C
%\le\frac{t}{2}\log(1+|z|^2)+\frac{t}{2}\,\frac{2|zh|+|h|^2}{1+|z|^2}+C,
%\qquad z,h\in\C,
%\end{multline*}
%where we used the elementary estimate $\log(1+a)\le a$ for real $a>-1$. 
By \eqref{eq-4.2}, the growth of $Q(z)$ is faster than that of $\Obs_t[Q]$,
which we can use to show that for some possibly big value of the radius $r_0$,
%\[
%\frac{t}{2}\log(1+|z|^2)+C\le Q(z),\qquad |z|\ge r_0,
%\]
\begin{equation}
\Obs_t[Q](z+w)\le Q(z),
%+\frac{t}{2}\,\frac{2|zh|+|h|^2}{1+|z|^2},
\qquad |z|\ge r_0,\,|w|\le1.
\label{eq-4.5}
\end{equation}
Let 
\[
M_r:=\sup\big\{|\partial_t^2\{ Q(z+t\zeta)\}|:\,|z|\le r,\,\,|\zeta|=1,\,
0\le t\le1\big\},
\]
which is finite for each radius $r$ due to the assumption that $Q\in C^{1,1}$, 
and note that by Taylor's formula,
\begin{equation}
\Obs_t[Q](z+w)\le Q(z+w)\le Q(z)+2\re[w\partial Q(z)]+
\tfrac12 M_{r}|w|^2,\qquad |z|\le r,\,\,|w|\le1.
\label{eq-4.6}
\end{equation}
We fix $w\in\C$ with $|w|\le1$, and put 
\[
\tilde Q_w(z):=\tfrac12\Obs_t[Q](z+w)+\tfrac12\Obs_t[Q](z-w)-
\tfrac12 M_{r_0}|w|^2.
\]
By a combination of \eqref{eq-4.5} and \eqref{eq-4.6}, $\tilde Q_w\le Q$ on
$\C$, while it is obvious that $\tilde Q_w\in\Sub_t(\C)$. So, from the
definition of the obstacle problem, we see that
$\tilde Q_w\le\Obs_t[Q]$ on $\C$. In other words,
\begin{equation}
\Obs_t[Q](z+w)+\Obs_t[Q](z-w)-2\Obs_t[Q](z)\le M_{r_0}|w|^2,\qquad z\in\C,\,\,
|w|\le1.
\label{eq-4.7}
\end{equation}
Next, if we divide both sides of \eqref{eq-4.7} by $|w|^2$ and then let 
$w\to0$, we get
\begin{equation*}
\partial_t^2\Obs_t[Q](z+t\zeta)\Big|_{t=0}\le M_{r_0},\qquad z\in\C,\,\,
|\zeta|=1.
%\label{eq-4.8}
\end{equation*}
In particular, if $z=x+\imag y$, we have
\[
\partial_x^2\Obs_t[Q](z)\le M_{r_0},\quad
\partial_y^2\Obs_t[Q](z)\le M_{r_0}.
\]
Since $\Obs_t[Q]$ is subharmonic, that is,
\[
\partial_x^2\Obs_t[Q](z)+\partial_y^2\Obs_t[Q](z)\ge0.
\]
holds in the sense of distribution theory, we must then also have
\[
-M_{r_0}\le\partial_x^2\Obs_t[Q](z)\le M_{r_0},\quad
-M_{r_0}\le\partial_y^2\Obs_t[Q](z)\le M_{r_0}.
\]
In particular, then, $\Obs_t[Q]\in C^{1,1}$. 

As for the remaining case when we have less smoothness, that is, 
when $Q\in W^{2,p}$, the assertion follows from the smoothness theory of
constrained obstacle problems (see Lemma \ref{lem-4.6} and Theorem 
\ref{thm-4.8} below). 
\end{proof}

\subsection{A constrained obstacle problem}
\label{subsec-4.2.2}
Let $Q:\C\to\R\cup\{+\infty\}$ be lower semi-continuous with \eqref{eq-4.2}
where $t$ is a (fixed) positive real, as before. 
Let $\Omega$ be a (bounded) Jordan domain, and $\varrho:\partial\Omega\to\R$  
a continuous function with $\varrho\le Q|_{\partial\Omega}$. 
Consider the constrained obstacle problem
\[
\Obs_{\Omega,\varrho}[Q](z)
:=\sup\big\{v(z):~ v\in\Sub(\Omega), ~ v\le Q\,\,\text{on}\,\,\Omega, 
~v=\varrho \;{\rm on}\; \partial\Omega\big\},\qquad z\in\bar\Omega.
\]
We would like to model the obstacle problem associated with $\Obs_t[Q]$ 
in the form of such a constrained obstacle problem.
The natural way to do this is to put $\varrho:=\Obs_t[Q]
\big|_{\partial\Omega}$. 

\begin{lem} If $\Omega$ is a $C^\infty$-smooth bounded Jordan domain and 
$\varrho=\Obs_t[Q]\big|_{\partial\Omega}$, then 
\[
\Obs_{\Omega,\varrho}[Q]=\Obs_t[Q] \quad\text{on}\,\,\,\Omega.
\]
\label{lem-4.6}
\end{lem}

\begin{proof} 
We put $R_0=\Obs_t[Q]$, $R_1=\Obs_t[Q]\big|_{\bar\Omega}$, and 
$R_2=\Obs_{\Omega,\varrho}[Q]$.
The function $R_1$ is subharmonic with $R_1\le Q$
in $\Omega$,  and has boundary values $R_1\big|_{\partial\Omega}=\varrho$. 
It is now immediate that $R_1\le R_2$. We proceed to show that $R_2\le R_1$.
To this end, we let $v\in\Sub(\Omega)$ have $v\le Q$ on $\Omega$ and 
boundary data $v=\varrho$ on $\partial\Omega$; we are to check that 
$v\le R_1$. Next, we put 
$\tilde v=\max\{v,R_1\}$; the function $\tilde v$ is in $\Sub(\Omega)$, has
$R_1\le \tilde v\le Q$ on $\Omega$, and boundary data 
$\tilde v|_{\partial\Omega}=\varrho$. 
We consider its extension
\[
V=\begin{cases}\tilde v\quad{\rm in}\;\Omega,\\  R_0\quad{\rm in}\;\C\setminus
\Omega.
\end{cases}
\]
The way things are set up, $R_0\le V\le Q$ in $\C$, with $V=R_0$ on 
$\partial\Omega$. We claim that $V\in\Sub(\C)$. It is enough to check the
mean value inequality along $\partial\Omega$. 
For points $a\in\partial\Omega$, we have ($\epsilon>0$ is a small real 
parameter)
\[
V(a)=R_0(a)\le \frac{1}{2\pi}
\int_{-\pi}^{\pi}R_0(a+\epsilon\e^{\imag\theta})\diff\theta
\le
\int_{-\pi}^{\pi}V(a+\epsilon\e^{\imag\theta})\diff\theta.
\]
It follows that $V\in\Sub(\C)$ and {\em a fortiori} $V\in\Sub_t(\C)$ (because
of the growth at infinity). We conclude that $V$ is a function which we
may plug into the optimization problem defining $R_0=\Obs_t[Q]$, and so
$V\le R_0$ on $\C$. In fact, due to the reverse inequality, we must have
$V=R_0$. In particular, $\tilde v=V|_{\bar\Omega}=R_1$, and so $v\le R_1$.  
\end{proof}

\subsection{Kinderlehrer-Stampacchia-Caffarelli theory}
 For our purposes it would be enough to consider the case of $C^2$ or even 
$C^\infty$ potentials $Q$ but since we sometimes have to modify them 
(see e.g. \cite{AHM2}),  the Sobolev classes $W^{2, p}$ seem to be more 
appropriate.  
%We use the notation $W^{2, p}$ for {\it local} classes unless otherwise 
%stated.
We generally assume that $Q:\C\to\R$ is continuous subject to the growth
condition \eqref{eq-4.2} for some (fixed) positive real $t$.

We start with a simple observation.

\begin{lem} 
\label{lem-4.8}
Let $S_t=S_t[Q]$ and  suppose $Q\in W^{2,1}(\Int\,S_t)$. Then 
$\hat\sigma_t$ is absolutely continuous in $\Int\,S_t$ and in fact
\[
\diff\hat\sigma_t=\Delta Q\,\diff A\quad\text{on}\,\,\,\,\Int\,S_t.
\]
\label{lem-4.7}
\end{lem}

\begin{proof} 
As we know that $\diff\sigma_t=\Delta\widehat Q_t\diff A$ in the sense of 
distributions, and so the same is true if we restrict the distributions to 
the open set $\Int\, S_t$, where  $\widehat Q_t=Q$, and therefore 
$\Delta\widehat Q_t=\Delta Q$ as distributions.
\end{proof}

 The following two theorems are adapted from the theory of constrained obstacle
problems (variational inequalities); this theory is, as mentioned previously, 
associated with the names of Lewy, Stampacchia, Brezis, Lions, Kinderlehrer, 
and Caffarelli, et al. A standard references is \cite{Fried}, Chapter 1 
(see also \cite{CK}). 
%The adaptation is based on Lemma \ref{lem-4.6} above.

\begin{thm} 
Fix $p$, $1<p<+\infty$, and let $\Omega$ be a $C^\infty$-smooth 
bounded Jordan domain. We suppose $Q$ is $W^{2,p}$-smooth in $\C$, and 
that $\varrho:\partial\Omega\to\R$ is a function which is the restriction 
to $\partial\Omega$ of a function in $W^{2,p}(\C)$, with $\varrho\le Q$ on
$\partial\Omega$. Then $\Obs_{\Omega,\varrho}[Q]\in W^{2,p}(\Omega)$. 
\label{thm-4.8}
\end{thm} 

\begin{proof}
This is explained in Chapter 1 of Friedman's book \cite{Fried}, see Theorem 
1.3.2 and Problem 1 on p. 29. 
\end{proof}

Together with Lemma \ref{lem-4.6}, this justifies the $W^{2,p}$ part of the 
assertion of Proposition \ref{prop-4.6}.
Now, in view of Proposition \ref{prop-4.6}, if we suppose $Q\in W^{2,p}$
for some $1<p<+\infty$, the function $\widehat Q_t=\Obs_t[Q]$ is in $W^{2,p}$,
and by Corollary \ref{cor-4.4}, the measure $\hat\sigma_t$ is absolutely
continuous (with respect to area), and the density is locally in $L^p$.
By Lemma \ref{lem-4.7}, we get
\[
\diff\hat\sigma_t=\Delta Q\diff A\quad\text{on}\,\,\,\Int\, S_t.
\]
If the boundary $\partial S_t$ has zero area, we can conclude that 
$\diff\hat\sigma_t=1_{S_t}\Delta Q\diff A$. It is remarkable 
that this conclusion holds even when $\partial S_t$ has positive area.
%We need to introduce the {\em coincidence set}
%\begin{equation}
%\label{eq-St*}
%S^*_t=S_t^*[Q]:=\big\{z\in\C:\,\widehat Q_t(z)=Q(z)\big\},
%\end{equation}
%which is compact, for $Q\in W^{2,p}$ for some $p$, $1<p<+\infty$, subject to
%\eqref{eq-4.2}. Indeed, it is closed, being the set where two continuous 
%functions coincide, and bounded, which follows from \eqref{eq-4.2}. 

\begin{thm} 
If, for some $1<p<+\infty$, we have $Q\in W^{2,p}$, then $\widehat Q_t\in 
W^{2,p}$ and 
\[
\diff\hat\sigma_t=1_{S_t}\Delta Q\diff A.
\]
\label{thm-4.9}
\end{thm}

\begin{proof}
As $S_t$ is the support of $\hat\sigma_t$, the measure $\hat\sigma_t$ vanishes
off $S_t$. On $S_t$, however, the two $W^{2,p}$-smooth functions $Q$ and 
$\widehat Q_t$ coincide, and by \cite{KS}, p. 53, this entails that 
their partial derivatives of order $\le2$ coincide almost everywhere on $S_t$.
In particular, $\Delta\widehat Q_t=\Delta Q$ on $S_t$ as $L^p$ functions.
In view of Corollary \ref{cor-4.4}, the assertion is immediate.
\end{proof}

\begin{rem}
\label{rem-4.10}
\noindent(a) In the context of Proposition \ref{prop-4.6}, it does not help to
add more smoothness to $Q$. E.g., if $Q\in C^\infty$ is assumed, we still
cannot do better than $\Obs_t[Q]\in C^{1,1}$, at least near $\partial S_t$. 

\noindent(b) The smoothness assumptions of this subsection are excessive, 
in the sense
that it suffices to have the required smoothness of $Q$ in a neighborhood of
the droplet $S_t$. All the statements are valid under this weaker assumption.

\noindent(c) By the properties of the 2D Hilbert transform, the assertion that 
$\widehat Q_t\in W^{2,p}$ is equivalent to the property that the density
of the absolutely continuous measure $\diff\hat\sigma_t$ is in $L^p$ (locally).
\end{rem}

%{\bf Notes.} (a) In the first theorem (i) should hold for $p=\infty$, so (ii) 
%would be unnecessary (?) We can not get any better that $V\in W^{2,\infty}$ 
%even if $Q\in C^\infty$. The case $p=1$?

%(b) The conditions $\sigma\in L^p_{\rm loc}$ and $V\in W^{2,p}$ are 
%equivalent for $1<p<\infty$ as follows from the standard properties of 2D 
%Hilbert transform.

\subsection{The coincidence set and shallow points}
% and the dynamics of droplets} 
\label{subsec-shallow}
As in the previous subsection, $Q:\C\to\R$ is assumed to be of (local) 
Sobolev class $W^{2,p}$, with $1<p<+\infty$.
We assume that $Q$ meets the growth assumption \eqref{eq-4.2} for all
$t$ with $0<t<T$, where $T=T(Q)$ has $0<T\le+\infty$. 
    
We recall that we introduced the parameter $t$ to consider the evolution of 
the renormalized equilibrium measures $\hat\sigma_t=t\hat\sigma_{Q/t}$ as $t$ 
moves. The conclusion of Theorem \ref{thm-4.9} allows us to reduce the 
complexity and just study the evolution of the droplets $S_t=S_t[Q]=S_{Q/t}$.
The super-coincidence set $S_t^*=S^*_t[Q]$ defined by \eqref{eq-St} will be
referred to as the coincidence set, because the smoothness of $Q$ makes 
$\widehat Q_t$ continuous.

We should explain the relationship between the sets $S_t$ and $S_t^*$ (we 
already know that $S_t\subset S_t^*$). 
To this end, we say that a point $z_0\in S_t^*$ is $Q$-{\em shallow} 
(with respect to $S_t^*$) if there exists an open disk $D$ centered at $z_0$ 
such that 
\[
\int_{S_t^*\cap D}|\Delta Q|\,\diff A=0.
\] 
The $Q$-shallow point in $S_t^*$ form a relatively open subset.
We mention in passing that it follows from Theorem \ref{thm-4.9} that 
$\Delta Q\ge0$ a.e. on $S_t$.

\begin{prop}
The set $S_t$ is obtained from $S_t^*$ by removal of all the $Q$-shallow 
points.
\label{prop-4.14}
\end{prop}

\begin{proof}
Since $\widehat Q_t$ and $Q$ are both in $C^{1,1}$ and coincide on $S_t^*$,
we get from \cite{KS}, p. 53, that $\Delta\widehat Q_t=\Delta Q$ holds a.e. 
on $S_t^*$, so that (in the same way as Lemma \ref{lem-4.7} was obtained)
\[\diff\hat\sigma_t=1_{S_t^*}\Delta Q\,\diff A.\]
By comparing with Lemma \ref{lem-4.7}, we see that $\Delta Q=0$ a.e. on 
$S_t^*\setminus S_t$. To calculate the support of $\hat\sigma_t$, we must 
remove all the points of $S_t^*$ where there is no $|\Delta Q|\diff A$-mass 
nearby, that is, the $Q$-shallow points.
\end{proof}

\subsection{Coincidence sets and the dynamics of droplets}
\label{subsec-coinc}
As in the previous subsection, $Q:\C\to\R$ is assumed to be of class 
$W^{2,p}$, with $1<p<+\infty$.
We assume that $Q$ meets the growth assumption \eqref{eq-4.2} for all
$t$ with $0<t<T$, where $T=T(Q)$ has $0<T\le+\infty$. 

The coincidence set $S_t^*=S_t[Q]$ defined by \eqref{eq-St} is just a little
bigger than $S_t$ (we remove the $Q$-shallow points), but it contains 
essential information which helps us understand the evolution of $S_t$ as $t$ 
grows.  

We begin with some elementary properties.

\begin{lem} 
If $0<t_1\le t_2<T$, then $\widehat Q_{t_1}\le\widehat Q_{t_2}$ 
and, in particular, 
$S^*_{t_1}\subset S^*_{t_2}$.
\label{lem-4.12}
\end{lem}

\begin{proof} Since $\Sub_{t_1}(\C)\subset\Sub_{t_2}(\C)$, we clearly have
$\Obs_{t_1}[Q]\le\Obs_{t_2}[Q]$, from the definition of the obstacle problem.
The first assertion, $\widehat Q_{t_1}\le\widehat Q_{t_2}$, now follows from 
Proposition \ref{prop-4.2}. The second assertion, the inclusion 
$S^*_{t_1}\subset S^*_{t_2}$, is an easy consequence of the first assertion.
\end{proof}

Let $S_0^*$ denote the (nonempty compact) set where the global minimum of $Q$ 
is attained. 

\begin{prop} We have $S_0^*\subset S_t^*$ for all $0<t<T$.
\end{prop} 

\begin{proof}
Pick a point $a\in S_0^*$, and observe that the function $\max\{Q(a),\widehat
Q_t\}$ is subharmonic (in fact, in $\Sub_t(\C)$) and therefore competes with 
$\widehat Q_t$ for the obstacle problem. We conclude that 
$Q(a)\le\widehat Q_t$. As $\widehat Q_t\le Q$, it follows that 
$\widehat Q_t(a)=Q(a)$, so that $a\in S_t^*$. The proof is complete.
\end{proof}

\begin{prop} If $0<t_1\le t_2<T$, we have $S_{t_1}\subset S_{t_2}$ and 
$\hat\sigma_{t_1}\le\hat\sigma_{t_2}$. 
\label{prop-4.15}
\end{prop}

\begin{proof}
If a point $a\in S_{t_1}^*\cap S_{t_2}^*$ is $Q$-shallow with respect to 
$S_{t_2}^*$, then it is also $Q$-shallow with respect to $S_{t_1}^*$, since 
$S_{t_1}^*\subset S_{t_2}^*$, by Lemma \ref{lem-4.12}. The first assertion,
$S_{t_1}\subset S_{t_2}$, now follows from a second application of 
$S_{t_1}^*\subset S_{t_2}^*$. The second assertion, 
$\hat\sigma_{t_1}\le\hat\sigma_{t_2}$, is a consequence of the first assertion
combined with Theorem \ref{thm-4.9}.
\end{proof}

\begin{rem}
This supplies the potential theoretical proof of Corollary \ref{cor-2.12} 
alluded to in Remark \ref{rem-2.13}.
\end{rem}

%\subsubsection{Continuity properties}  
We need the following two lemmas from \cite{ST}, pp. 227-228.

\begin{lem} 
{\rm($0<t_0<T$)} The map $t\mapsto S_t$ is monotonically increasing and 
left-continuous in the Hausdorff metric:
\[
S_t\nearrow S_{t_0}\quad\text{as}\,\,\,t \nearrow t_0.
\]
\end{lem}

This means that $S_{t_0}$ is in a small neighborhood of $S_t$ for 
$t<t_0$ close to $t_0$ or, equivalently, that
\[
S_{t_0}=\clos \bigcup_{t<t_0} S_t.
\]

\begin{lem}
{\rm($0<t_0<T$)}
We have
\[\bigcap_{t_0<t<T} S_t\subset S_{t_0}^*.
\]
In particular, if $S_{t_0}=S_{t_0}^*$, then  
\[S_{t_0}=\bigcap_{t>t_0} S_t.\]
\label{lem-4.18}
\end{lem}

It is easy to construct examples which show how $S_{t_0}^*$ may contain
``seed points'' outside the main body of $S_{t_0}$ which grow into (small)
components of $S_t$ for $t>t_0$. The next lemma gives a criterion which 
guarantees that this phenomenon takes place. For a compact set $E\subset\C$, 
let $\Phull(E)$ denote its polynomially convex hull, that is,
\[
\Phull(E):=\big\{z\in\C:\,|p(z)|\le\max_E|p|\,\,\,\text{for all polynomials}
\,\,p\big\}.
\]
The (compact) set $\Phull(E)$ adds to $E$ all the points of $\C\setminus E$ 
which belong to bounded connectivity components of $\C\setminus E$ 
(i.e., points invisible to Brownian motion in $\C\setminus E$ starting at 
$\infty$).  

\begin{lem} 
For all $t_0,t$ with $0<t_0<t<T$, we have the inclusion
\[
\partial[\Phull(S^*_{t_0})]\subset S_t.
\]
\label{lem-4.19}
\end{lem}

\begin{proof}
The standard geometric interpretation of the polynomially convex hull gives
that 
\[\partial[\Phull(S^*_{t_0})]\subset\partial S^*_{t_0}\subset S^*_{t_0}.\] 
We need to show that $\partial[\Phull(S^*_{t_0})]\subset S_t$ for all $t$, 
$t_0<t<T$. We argue by contradiction, and suppose that there exists a point
$a\in \partial[\Phull(S^*_{t_0})]$ such that $a\in\C\setminus S_{t_1}$ for 
some $t_1$, with $t_0<t_1<T$. Then 
$a\in\C\setminus S_{t}$ for all $t$ with $t_0<t\le t_1$, 
and if $t>t_0$ is sufficiently close to $t_0$, the point $a$ belongs to the 
unbounded component of $\C\sm S_t$. 
Indeed, choose a small open neighborhood $U$ of $a$ avoiding $S_{t_1}$ and 
since, by assumption, $a\in\partial[\Phull(S^*_{t_0})]$, we may assured that
there exists a point $b\in U\sm \Phull(S^*_{t_0})$. The point $b$ belongs
to the unbounded component of $\C\setminus S^*_{t_0}$, so we may connect $b$ 
with $\infty$ by a curve $\gamma$ in $\C\sm  S^*_{t_0}$; $\gamma$ is at a 
positive distance from $S^*_{t_0}$. 
By Lemma \ref{lem-4.18}, 
$\gamma\subset \C\sm S_t$ for  all $t>t_0$ close to $t_0$, and so the point
$b$ -- and a fortiori $a$ --  is in the unbounded component of $\C\sm S_t$.
Next, we consider (for $t$ with $t>t_0$ close to $t_0$) the function 
$u=\widehat Q_{t_0}-\widehat Q_t$. Then, by Lemma \ref{lem-4.12}, we have 
$u\le0$. Moreover, since $a\in S_{t_0}^*\subset S_t^*$, we have
$\widehat Q_{t_0}(a)=\widehat Q_t(a)=Q(a)$, and therefore, $u(a)=0$. The
function $\widehat Q_{t_0}$ is harmonic in $\C\setminus S_{t_0}$, and, 
likewise, $\widehat Q_{t}$ is harmonic in $\C\setminus S_{t}$, so we conclude
that $u$ is harmonic in $\C\setminus S_t$. The function $u$ then has a local
maximum at the interior point $a$, so by the strong maximum principle,
we get that $u=0$ throughout $\C\setminus \Phull(S_{t})$. This does not 
agree with the known asymptotics \eqref{eq-4.3}. 
We conclude that the initial assumption must be false, so that 
$a\in S_{t_1}$ for all $t_1$ with $t_0<t_1<T$.
\end{proof}

\begin{rem} 
The above assertions extend to the case $t_0=0$ if as before $S_0^*$ is 
the set where the global minimum of $Q$ is attained, and we put
$S_0=\emptyset$.
\end{rem}

\subsection {Subharmonic potentials} As before, $Q:\C\to\R$ is assumed to be of
class $W^{2,p}$, so that e.g. $\Delta Q\in L^p_{\rm loc}(\C)$.  
We suppose there exists $T=T(Q)$ with $0<T\le+\infty$ such that 
\eqref{eq-4.2} holds for $0<t<T$ while it fails for $t>T$. 

\begin{lem} 
{\rm($0<t<T$)} Let $D$ be a bounded domain in $\C$ and 
suppose $\Delta Q\ge0$ in $D$. Then $\partial D\subset S_t^*$ implies 
$D\subset S_t^*$. 
\label{lem-4.21}
\end{lem}

\begin{proof} The assumption $\partial D\subset S_t^*$ means that 
$\widehat Q_t=Q$ on $\partial D$. We write $R_0=\widehat Q_t$, and let $R_1$
be the function which equals $Q$ in $D$ and equals $\widehat Q_t$ elsewhere.
We observe that $R_0\le R_1\le Q$ on $\C$, while $R_0=R_1=Q$ on $\partial D$. 
Also, the function $R_1$ is subharmonic. Indeed, $\Delta R_1=\Delta Q\ge0$ on 
$D$ (by assumption), and $\Delta R_1=\Delta R_0\ge0$ on $\C\setminus\bar D$. 
It remains to observe that for $a\in\partial D$,  
\[
R_1(a)=R_0(a)\le\frac{1}{2\pi}\int_{-\pi}^{\pi}
R_0(a+\varepsilon \e^{\imag\theta})\,\diff\theta\le 
\frac{1}{2\pi}\int_{-\pi}^{\pi}R_1(a+\varepsilon \e^{\imag\theta})\,
\diff\theta,\qquad 0<\varepsilon<+\infty.
\]
We see that $R_1$ is subharmonic in $\C$, and the conclusion $R_1=R_0$
follows. 
\end{proof}

\begin{cor} If $Q$ is subharmonic in $\C$, then $\C\sm S^*_t$ is connected.
\label{cor-4.22}
\end{cor}

A continuous function $h:\C\to\R$ is said to be {\em nowhere harmonic} if
for every open set $D\subset\C$ the restriction $h|_D$ fails to be harmonic. 

\begin{cor} Suppose $Q$ is subharmonic in $\C$, and that
$Q$ is nowhere harmonic. Then 
$S^*_{t_0}\subset S_t$ for all $t_0,t$ with $0<t_0<t<T$.
\end{cor}

\begin{proof}
By Corollary \ref{cor-4.22}, the set $\C\setminus S_{t_0}^*$ is connected, and
so $\Phull(S_{t_0}^*)=S_{t_0}^*$. By Lemma \ref{lem-4.19}, then, we arrive at
$\partial S_{t_0}^*\subset S_t$ for all $t$ with $t_0<t<T$. It remains to 
check that $\Int S_{t_0}^*\subset S_t$ for all $t$ with $t_0<t<T$. By 
Proposition \ref{prop-4.14}, we just need to show that no point in 
$\Int S_{t_0}^*$ is $Q$-shallow with respect to $S_t^*$. This is guaranteed by 
the requirement that $Q$ be nowhere harmonic.
\end{proof}

\subsection{Convex potentials}
We say that a convex function $q:\C\to\R$ is {\em locally uniformly convex} if
\[
|\xi|^2\Delta q(z)+\re[\xi^2\partial^2q(z)]\ge \epsilon(z)|\xi|^2,\qquad
\xi\in\C,
\]
for some continuous $\epsilon:\C\to]0,+\infty[$. For $C^2$-smooth $q$, this 
just says that the Hessian of $q$ is (strictly) positive definite everywhere.

In \cite{KS}, Chapter V, coincidence sets for constrained obstacle problems 
are considered, and under suitable convexity assumptions, the coincidence set
is simply connected with $C^{1,\alpha}$-smooth boundary (here, $0<\alpha<1$). 
The setting is the following. Suppose $\Omega$ is a strictly convex bounded
$C^\infty$-smooth domain, and let $q:\bar\Omega\to\R$ be $C^2$-smooth and 
locally uniformly convex, with $q>0$ on $\partial\Omega$ and 
$\min_\Omega q<0$. Then, if we put
$\varrho=0$ in the constrained obstacle problem (see Subsection 
\ref{subsec-4.2.2}), the coincidence set
\[
S^*_{\Omega,q}:=\big\{z\in\Omega:\,\Obs_{\Omega,0}[q](z)=q(z)\big\} 
\] 
is non-empty, compact, simply connected, and equal to the closure of its 
interior. Moreover, if $q$ is $C^{2,\alpha}$-smooth for some $\alpha$, 
$0<\alpha<1$, then the boundary $\partial S^*_{\Omega,q}$ is a  
$C^{1,\alpha'}$-smooth Jordan curve, for some $\alpha'$, $0<\alpha'<1$.

Applied to our setting (cf. Subsection \ref{subsec-smooth}), we get Theorem 
\ref{thm-4.11} below. Before we formulate the theorem, we note that if 
$Q:\C\setminus\R$ is convex, and \eqref{eq-4.2} holds for some positive $t$, 
then $Q$ must grow faster (radially, the growth is at least linear), so that 
\eqref{eq-4.2} holds for all positive reals $t$ (which makes 
$T=T(Q)=+\infty$). 

\begin{thm} 
{\rm ($0<t<T=+\infty$)} 
%Suppose $\Omega$ is strictly convex, $q$ is strictly concave, 
%$q<0$ on $\partial \Omega$ and $q(a)>0$ at some $a\in \Omega$. Then the 
%coincidence set $\{\hat V=q\}$ 
%of the obstacle problem in $\Omega$ is connected.
Suppose $Q:\C\to\R$ is $C^2$-smooth and locally uniformly convex with 
\eqref{eq-4.2}. Then the droplet $S_t$ is simply connected, and equal to the
closure of its interior. Moreover, 
if $Q$ is $C^{2,\alpha}$-smooth for some $\alpha$, $0<\alpha<1$, then 
$\partial S_t$ is a  $C^{1,\alpha'}$-smooth Jordan curve, for some $\alpha'$, 
$0<\alpha'<1$.
\label{thm-4.11}
\end{thm}

\begin{proof}
We claim that for big enough $c$, the compact set
\[
\bar\Omega_c:=\big\{z\in\C:\widehat Q_t(z)\le c\big\}
\]
is strictly convex with $C^\infty$-smooth boundary. In fact, we know from 
\eqref{eq-4.3} and the fact that $\widehat Q_t$ has the form
\[
\widehat Q_t(z)=t\log|z|^2+h(z),
\]
where $h$ is real-valued, bounded, and harmonic in a neighborhood of infinity.
As $c$ increases the sets $\bar\Omega_c$ cover bigger and bigger portions of 
the plane $\C$, and the boundary $\partial\bar\Omega_c$ is contained in a fixed
neighborhood of infinity for big enough $c$. The equation defining the boundary
is 
\[
|z|\,\e^{h(z)}=\e^c,
\]
and an argument using the harmonic conjugate of $h$ shows this equation may
be written in the form
\[
|z+a_0+a_1z^{-1}+a_2z^{-2}+\ldots|=\e^c,
\] 
where the series converges for big $|z|$. In other words, 
using the inverse mapping, $\partial\bar\Omega_c$ is (for big $c$) the 
image of the circle $|z|=\e^c$ under a mapping
\[
z\mapsto z+b_0+b_1z^{-1}+b_2z^{-2}+\ldots,
\]
which also converges for big $|z|$. After rescaling by a factor of $\e^{-c}$,
we are talking about the image of the unit circle $|z|=1$ under the mapping
\[
z\mapsto z+b_0\e^{-c}+b_1\e^{-2c}z^{-1}+b_2\e^{-3c}z^{-2}+\ldots,
\]
which for large values of $c$ constitutes a very slight perturbation of 
the circle $|z|=1$, and it is then easy to check that the domain inside the 
curve is strictly convex with $C^\infty$-smooth boundary. As a consequence, 
$\bar\Omega_c$ is strictly convex with $C^\infty$-smooth boundary for big $c$.
To finish the proof, we observe that (cf. Lemma \ref{lem-4.6})
\[
\Obs_{\Omega_c,0}[q]+c=\Obs_t[Q]\,\,\,\text{on}\,\,\,\bar\Omega_c,
\]
if $q=Q-c$. It is immediate that $S_t=S_{\Omega_c,q}$. The rest follows 
from Chapter V of \cite{KS}.  
\end{proof}

%Describe  pathological dynamics in $C^\infty$ case. Interpreting LG could be 
%tricky in general.

%This discussion could be quite sexy.

\section{Local droplets}
\label{sec-ld}

\subsection{Localization}
\label{subsec-5.1} 
We often localize the field $Q:\C\to\R\cup\{+\infty\}$ (which we assume to be
lower semi-continuous) to a closed set $\Sigma\subset\C$ and write 
\[
Q_\Sigma=\begin{cases}
Q\quad\,\,\,\,\,\text{on}\,\,\,\Sigma
\\
+\infty\quad\text{on}\,\,\,\C\sm\Sigma.
\end{cases}
\]
The function $Q_\Sigma$ is then also lower semi-continuous.
We will assume that $Q_\Sigma$ meets the growth condition \eqref{eq-4.2},
(which is the $t$-scaled version of \eqref{eq-2.6}; in case $\Sigma$ is 
compact, this 
is automatically so irrespective of the behavior of $Q$ near infinity). To 
avoid triviality, we also need to require that $Q_\Sigma<+\infty$ on a set 
of positive area. We will refer to the closed set $\Sigma$ as a 
{\em localization}. 

We will use the notation (which corresponds to the special parameter choice 
$t=1$)
\[
\hat\sigma_{Q,\Sigma}=\hat\sigma[Q,\Sigma]:=\hat\sigma_{Q_\Sigma},\quad
S_{Q,\Sigma}=S[Q,\Sigma]:=\supp\hat\sigma[Q,\Sigma];
\]
this conforms with the convention to write $\hat\sigma_Q=\hat\sigma[Q]$
and $S_Q=S[Q]$. We will focus on the $t$-scaled variants 
($0<t<+\infty$)
\[
\hat\sigma_t[Q,\Sigma]:=t\,\hat\sigma_{Q/t,\Sigma},\quad
S_t[Q,\Sigma]:=\supp\hat\sigma_t[Q,\Sigma]=\supp\hat\sigma[Q/t,\Sigma]=
S[Q/t,\Sigma].
\]
We shall also need the modified Robin constant 
\[
\gamma^*_t(Q,\Sigma)=\gamma^*_t(Q_\Sigma)
\]
from Subsection \ref{subsec-4.1}.

\begin{lem} 
{\rm($0<t<+\infty$)}
Suppose $\Sigma\subset\C$ is closed, and that $Q_{\Sigma}$ meets the growth 
condition \eqref{eq-4.2}, while $Q_{\Sigma}<+\infty$ holds on a set
of positive area. Then $S_t[Q,\Sigma]\subset\Sigma$.
\label{lem-5.1}
\end{lem}

\begin{proof}
If the probability measure $\sigma_0:=\hat\sigma[Q/t,\Sigma]$ were 
to have support outside $\Sigma$, the corresponding energy 
\[
I_{Q_\Sigma/t}[\sigma_0]=\int_{\C^2}\log\frac{1}{|\xi-\eta|^2}
\diff\sigma_0(\xi)\diff\sigma_0(\eta)+\frac{2}{t}\int_\C Q_\Sigma\diff\sigma_0
\]
would necessarily equal $+\infty$, which does not agree with the energy 
minimizing property of the equilibrium measure.
\end{proof}

We now compare two different localizations, one contained in the other.

\begin{lem} 
{\rm($0<t<+\infty$)}
Suppose $\Sigma_1,\Sigma_2\subset\C$ are closed with 
$\Sigma_1 \subset\Sigma_2$, and that $Q_{\Sigma_2}$ meets the growth 
condition \eqref{eq-4.2}, while $Q_{\Sigma_1}<+\infty$ holds on a set
of positive area. If $S_t[Q,\Sigma_2]\subset\Sigma_1$, then
\[
\hat\sigma_t[Q,\Sigma_1]=\hat\sigma_t[Q,\Sigma_2],\quad 
S_t[Q,\Sigma_1]=S_t[Q,\Sigma_2].
\]
\label{lem-5.2}
\end{lem}

\begin{proof} For $j=1,2$, we write $\sigma_j=\hat\sigma[Q/t,\Sigma_j]$ and 
$Q_j=Q_{\Sigma_j}/t$. If $L_Q$ is as in \eqref{eq-LQ}, we then get that (since
$\Sigma_1\subset\Sigma_2$)
\[
L_{Q_1}(\xi,\eta)=L_{Q_2}(\xi,\eta),\qquad (\xi,\eta)\in\Sigma_1\times\Sigma_1,
\]
and so (since $S_t[Q,\Sigma_2]=\supp\hat\sigma[Q/t,\Sigma_2]=\supp\sigma_2
\subset\Sigma_1$)
\begin{equation}
\label{eq-5.1}
I_{Q_1}[\sigma_2]=\int_{\C^2}L_{Q_1}(\xi,\eta)
\diff\sigma_2(\xi)\diff\sigma_2(\eta)=\int_{\C^2}L_{Q_2}(\xi,\eta)
\diff\sigma_2(\xi)\diff\sigma_2(\eta)=I_{Q_2}[\sigma_2],
\end{equation}
where we have used the identity \eqref{eq-intQ2}. As $Q_1\ge Q_2$, we have
\[
I_{Q_1}[\sigma_1]=\inf_{\sigma} I_{Q_1}[\sigma]
\ge\inf_{\sigma} I_{Q_2}[\sigma]=I_{Q_2}[\sigma_2], 
\]
where both infima run over $\sigma\in\probc$. Combined with \eqref{eq-5.1}, 
this gives $I_{Q_1}[\sigma_1]\le I_{Q_1}[\sigma_1]$, which is only possible
if $\sigma_1=\sigma_2$, by Frostman's theorem (Theorem \ref{thm-frost1}).
Finally, if the measures coincide, their supports coincide as well.
\end{proof}

The typical application of Lemma \ref{lem-5.1} will be when both 
$\Sigma_1$ and $\Sigma_2$ are compact. However, already the case when 
$\Sigma_1=S_Q$ and $\Sigma_2=\C$ is interesting.

\begin{cor} 
{\rm($0<t<+\infty$)}
If $Q$ meets the growth condition \eqref{eq-4.2}, then with 
$\Sigma:=S_t[Q]=S_{Q/t}$, we have
\[
\hat\sigma_t[Q]=\hat\sigma_t[Q,\Sigma],\quad S_t[Q,\Sigma]=S_t[Q].
\]
\end{cor}

%\begin{cor} If $Q_1=Q_2$ on $S[Q_1]\cup S[Q_2]$, then  
%$\sigma[Q_1]=\sigma[Q_2]$. 
%\end{cor}
%\begin{proof} Denote $\Sigma=S[Q_1]\cup S[Q_2]$, so 
%$\sigma[Q_1, \Sigma]=\sigma[Q_2, \Sigma]$. By the lemma, 
%$\sigma[Q_j] =\sigma[Q_j, \Sigma]$.
%\end{proof}

\subsection{Local droplets and the obstacle problem}

We let $\Sigma$ be a localization, and suppose that $Q<+\infty$ on a subset 
of $\Sigma$ with positive area. We require that $Q_\Sigma$ meets the growth 
condition \eqref{eq-4.2} for a positive $t$, which is kept fixed for the 
moment (this requirement is void if $\Sigma\subset\C$ is compact). 
Then the Borel measure $\hat\sigma_t[Q,\Sigma]$ is a well-defined 
positive measure of total mass $t$, and its support $S_t[Q,\Sigma]$ is 
compact with $S\subset\Sigma$.
The following lemma is immediate from Lemma \ref{lem-4.7}.

\begin{lem} 
Suppose $Q\in W^{2,1}(\Int\,S)$ where $S=S_t[Q,\Sigma]$. Then 
$\hat\sigma_t[Q,\Sigma]$ is absolutely continuous in $\Int\,S$ and 
in fact
\[
\diff\hat\sigma=\Delta Q\,\diff A\quad\text{on}\,\,\,\,\Int\,S.
\]
\label{lem-5.4}
\end{lem}

\begin{proof}
This follows from Lemma \ref{lem-4.8}. 
\end{proof}

%Consider the open set and closed sets
%\[
%\calN_+=\calN_+[Q]:=\{z\in\C:\,\Delta Q(z)>0\},\quad
%\calN_{+,0}=\calN_{+,0}[Q]:=\{z\in\C:\,\Delta Q(z)\ge0\}.
%\]
%In view of Lemma \ref{lem-5.4}, the probability measure 
%$\hat\sigma=\hat\sigma[Q,\Sigma]$ has $\Int S\subset\calN_{+,0}$, where $S=
%S[Q,\Sigma]=\supp\hat\sigma$, and
%\[
%\hat\sigma(\Int\, S)=\int_{\Int\, S}\diff\hat\sigma=
%\int_{\Int\, S}\Delta Q\diff A.
%\]
%If $\Sigma$ is small in the sense that
%\[
%\int_{\calN_+\cap\Sigma}\Delta Q\,\diff A<1,
%\]
%then, as $\Int\,S\subset\Sigma\cap\calN_{+,0}$ we find that 
%$\hat\sigma(\Int\, S)<1$, and thus $\hat\sigma(\partial S)>0$. This means that
%$\hat\sigma$ has mass on the boundary of $S$, which is a situation we
%would prefer to avoid. 

We are led to the following three definitions.

\begin{defn}
Suppose $Q$ is in $W^{2,1}$ on a neighborhood of $S=S_t[Q,\Sigma]$. We say 
that $S$ is a {\em local $(Q,t)$-droplet with localization $\Sigma$} if the 
following equality holds ($\hat\sigma=\hat\sigma_t[Q,\Sigma]$):
\[
\diff\hat\sigma=1_{S}\Delta Q\diff A.
\] 
\label{def-5.4}
\end{defn}

\begin{defn}
A compact set $S\subset\C$ is a {\em local $(Q,t)$-droplet} if it is 
a local $(Q,t)$-droplet with respect to some localization $\Sigma$. 
\label{def-5.5}
\end{defn}

\begin{defn}
A compact set $S\subset\C$ is a {\em global $(Q,t)$-droplet} if it is 
a local $(Q,t)$-droplet with respect to the localization $\Sigma=\C$. 
\label{def-5.6}
\end{defn}

\begin{rem}
(a) There is at most one global $(Q,t)$-droplet $S$, as it is given by 
$S=S_t[Q]$. Theorem \ref{thm-4.9} guarantees that it exists under the 
growth requirement \eqref{eq-4.2} and the additional regularity $Q\in W^{2,p}$ 
(this is true also if $Q$ is in $W^{2,p}$ only in a neighborhood of $S$). 
In contrast, there may exist several local $(Q,t)$-droplets.

\noindent(b) If the boundary $\partial S$ of the set $S=S[Q,\Sigma]$ has 
zero area, then $S$ is a local $(Q,t)$-droplet if and only if and only if 
$\hat\sigma=\hat\sigma_t[Q,\Sigma]$ is absolutely continuous.

\noindent(c) Two local $(Q,t)$-droplets $S_1,S_2$ cannot have 
the containment $S_1\subset S_2$ unless $S_1=S_2$.

\noindent (d) The point with the definition of local droplets is that we may 
focus on 
the support $S=\supp\hat\sigma$ rather than the (generally more complicated) 
equilibrium measure $\hat\sigma$ (with respect to the weight $Q_\Sigma$). 

\noindent(e) In the above definition, it is possible to weaken the 
smoothness assumption on $Q$ to $W^{\Delta,1}$ smoothness, which just asks
that the function and its Laplacian are both locally integrable.
\end{rem}

\begin{prop}
If $S$ is a local $(Q,t)$-droplet with respect to some localization 
$\Sigma=\Sigma_0$, then it is a local $(Q,t)$-droplet with respect to the 
minimal localization $\Sigma=S$. 
\label{prop-5.9}
\end{prop}

\begin{proof}
This follows from Lemma \ref{lem-5.2} with $\Sigma_1=S$ and 
$\Sigma_2=\Sigma_0$.
\end{proof}

\begin{rem}
If $S=S_t[Q,\Sigma]$ is a local $(Q,t)$-droplet, then the associated 
measure $\hat\sigma=\hat\sigma_t[Q,\Sigma]$ is absolutely 
continuous. It is possible that the converse might be true (cf. Lemma 
\ref{lem-5.4} above). 
For the moment, we have a weaker statement. Suppose first that $Q$ is in
$W^{2,p}$ in a neighborhood of $S$, for some $p$, $1<p<+\infty$. 
The statement now runs as follows: if $\hat\sigma$ is absolutely
continuous with density in $L^p$ for some $p$, $1<p<+\infty$, then $S$ is a 
local $(Q,t)$-droplet. Indeed, from the properties of the 2D Hilbert transform,
we get that the function 
\[
\widehat{(Q_{\Sigma})}_t(\xi)=\gamma^*_t(Q,\Sigma)-U^{\hat\sigma}(\xi),
\]
is in $W^{2,p}$ and from Proposition \ref{prop-4.5} we have that 
\[
\widehat{(Q_{\Sigma})}_t(\xi)=Q(\xi),\qquad \xi\in S,
\]
so that by \cite{KS}, p. 53, we get
\[
\Delta\widehat{(Q_{\Sigma})}_t(\xi)=\Delta Q(\xi),\qquad \xi\in S,
\]
as distributions, which leads to the desired result. 
\end{rem}

For a compact $S\subset\C$, we define the corresponding (weighted) 
logarithmic potential
\begin{equation}
U^{Q,S}(\xi)=\int_{S}\log\frac{1}{|\xi-\eta|^2}\,\Delta Q(\eta)\diff A(\eta).
\label{eq-UQS}
\end{equation}
%For $Q\in W^{2,1}$, we have $\Delta Q\in L^{1}_{\rm loc}$, which does not
%necessarily force the potential $U^{S,Q}$ to be in $W^{2,1}$. However, with 
%some slight extra regularity, say $Q\in W^{2,p}$ with $1<p<+\infty$, we have 
%$\Delta Q\in L^{p}_{\rm loc}$ and automatically $U^{S,Q}\in W^{2,p}$ from the
%known properties of the 2D Hilbert transform. 

We have the following characterization of local $(Q,1)$-droplets. We recall the
notion of $Q$-shallow points from Subsection \ref{subsec-shallow}.

\begin{thm}
Suppose that $S\subset\Sigma\subset\C$, where $S$ is compact and $\Sigma$ 
closed, and that $Q$ is in $W^{2,1}$ in a 
neighborhood of $S$. Then $S$ is a local $(Q,t)$-droplet with localization
$\Sigma$ if and only if: 

\noindent {\rm (i)} $\Delta Q\ge0$ a.e. on $S$, 

\noindent{\rm(ii)} $S$ contains no $Q$-shallow points,
\[
\int_S \Delta Q\diff A=t,
\tag{\rm  iii}
\]
\noindent {\rm(iv)} 
$U^{S,Q}+Q=\gamma^*_t(Q,S)$ q.e. on $S$, for some real constant
$\gamma^*_t(Q,S)$ (the modified Robin constant), and

\noindent{\rm(v)} $U^{S,Q}+Q\ge\gamma^*_t(Q,S)$ q.e. on $\Sigma$. 
\label{thm-5.7}
\end{thm}

\begin{proof}
We first establish the necessity of conditions (i)-(iv). So, we suppose that
$S$ is a local $(Q,t)$-droplet. As $\diff\hat\sigma=1_S\Delta Q\diff A$ is 
positive
with mass $t$, and $S$ is its support set, conditions (i)-(iii) are necessary. 
The necessity of condition (iv) and (v) follows from Frostman's Theorem 
\ref{thm-frost3} (with $Q_S/t$ in place of $Q$, where $S$ is used as a 
localization).

We turn to the sufficiency of the conditions (i)-(v). We write 
$W:=c-U^{Q,S}$, where the constant $c=\gamma^*_t(Q,S)$ is as in (iv). By 
(iv), we then have $W=Q_\Sigma$ q.e. on $S$ while (v) gives 
$W\le Q_\Sigma$ q.e. on $\C$. By Lemma \ref{lm-4.1},
we get that $W=\Obs_t[Q_\Sigma]$. Next, Proposition \ref{prop-4.2} and 
Corollary \ref{cor-4.4} show that 
\[
\diff\hat\sigma_t[Q,S]=\Delta W\diff A=-\Delta U^{Q,S}\diff A
=1_S\Delta Q\diff A.
\]
So we have a local $(Q,t)$-droplet with localization $\Sigma$.
\end{proof}

\begin{rem}
To characterize the local $(Q,t)$-droplets, we use the minimal localization
$S$. We see that condition (v) becomes vacuous and may be removed.  
\end{rem}

\begin{cor}
Suppose that $S\subset\Sigma\subset\C$, where $S$ is compact and $\Sigma$ 
closed, and that $Q$ is in $W^{2,1}$ in a 
neighborhood of $S$. Consider the function
\[
\widehat Q_S:=\gamma^*_t(Q,S)-U^{Q,S},
\] 
where $\gamma^*_t(Q,S)$ is the constant in Theorem \ref{thm-5.7}. 
Then $S$ is a local $(Q,t)$-droplet with localization
$\Sigma$ if and only if: 

\noindent{\rm (i)} $\Delta Q\ge0$ a.e. on $S$, 

\noindent{\rm(ii)} $S$ contains no $Q$-shallow points,
\[
\int_S \Delta Q\diff A=t,
\tag{\rm  iii}
\]
\noindent{\rm(iv)} $\widehat Q_S=Q$ q.e. on $S$, and

\noindent{\rm(v)} $\widehat Q_S\le Q$ q.e. on $\Sigma$. 

\noindent Moreover, if {\rm(i)-(v)} are assumed, then
$\widehat Q_S\in\Sub_t(\C)$ is harmonic on $\C\sm S$, with asymptotics 
\[
\widehat Q_S(z)=t\log|z|^2+\Ordo(1)\quad\text{as}\,\,\,|z|\to+\infty.
\]
As a consequence, we have q.e.
\[
\widehat Q_S=\widehat{(Q_\Sigma)}_t=\Obs_t[Q_\Sigma].
\]
Moreover, if, for some $p$ with $1<p<+\infty$, we have $Q\in
W^{2,p}$ in a neighborhood of $S$, then $\widehat Q_S\in W^{2,p}$ as well.
\label{cor-5.13}
\end{cor}

\begin{proof}
It is clear from the properties of logarithmic potentials that 
$U^{Q,S}$ is subharmonic in $\C$ and harmonic in $\C\sm S$, with the 
corresponding asymptotics at infinity as a consequence of condition (ii) 
of Theorem \ref{thm-5.7}. Moreover, the properties of the 2D Hilbert 
transform show that if $Q\in W^{2,p}$ in a neighborhood of $S$, then
$U^{Q,S}\in W^{2,p}$, for $1<p<+\infty$.  
These properties are then inherited by $\widehat Q_S$. 
\end{proof}

If there is some room to wiggle between the set $S_t[Q,\Sigma]$ and the
localization $\Sigma$, then the set $S_t[Q,\Sigma]$ is automatically a 
local $(Q,t)$-droplet: 

\begin{thm}
Suppose $Q\in W^{2,p}$ for some $p$, $1<p<+\infty$. 
If for a localization $\Sigma$ we have 
$S=S_t[Q,\Sigma]\subset\Int\,\Sigma$, then $S$ is a local $(Q,t)$-droplet
with localization $\Sigma$.
\end{thm}

\begin{proof}
This is Theorem \ref{thm-4.9} for $Q_\Sigma$ in place of $Q$. 
\end{proof}

%Let $\OO$ be an open set and $Q\in C^2(\OO)$.
%By definition, a compact set $S\subset\OO$ is a local droplet for $Q$, or a 
%$Q$-droplet if $\exists~ \Sigma\subset\OO$ such that $S$ is a $(Q,\Sigma)$ 
%droplet. 

%For a local $Q$-droplet $S$ we write $t_S$, $\sigma_S$, $U_S$, 
%$V_S$, etc; they don't depend on $\Sigma$.

% 
%As before, $Q\in C^2(\OO)$,  $S$ be a compact subset of $\OO$ such that 
%$\Delta Q\ge 0$ on $S$ and 
%$$S=\supp ~\sigma_S, \qquad \sigma_S:=\Delta Q\cdot \chi_S.$$ Then  
%$$U_S\equiv U^{\sigma_S}\in W^{2,p}$$
%for all $p>\infty$, and so
%$$\partial U_S(z)=
%\int_S \frac{d\sigma_S(\zeta)}{\zeta-z}\in  W^{1,p}\subset C.$$

\begin{rem}

\end{rem} The modified Robin constant $\gamma^*(Q,S)$ may be written out 
explicitly:
\begin{equation}
\label{eq-5.2}
\gamma_t^*(Q,S)=\frac{1}{t}\int_{S\times S}\log\frac{1}{|\xi-\eta|^2}
\Delta Q(\xi)\Delta Q(\eta)\diff A(\xi)\diff A(\eta)+\int_S Q\Delta Q\diff A.
\end{equation}

\subsection{Characterization of local droplets}
We need the concept of local $Q$-droplets. We consider compact localizations 
$\Sigma$ only, which means that no requirement on $Q$ near infinity is 
needed, just that
$Q:\C\to\R\cup\{+\infty\}$ is lower semi-continuous and has $Q<+\infty$ 
on a subset of $\Sigma$ with positive area. We recall the concept of a 
$(Q,t)$-droplet, which presupposed that $Q$ was $W^{2,1}$-smooth near $S$. 

\begin{defn}
A compact set $S\subset\C$ is a (local) $Q$-droplet if it is a local 
$(Q,t)$-droplet for some $t$ with $0<t<+\infty$.
\end{defn}

We see that Theorem \ref{thm-5.7} has the following consequence.

\begin{cor}
Suppose that $S\subset\C$ is compact, and that $Q$ is in $W^{2,1}$ in a 
neighborhood of $S$. Then $S$ is a local $Q$-droplet if and only if: 

\noindent {\rm (i)} $\Delta Q\ge0$ a.e. on $S$, 

\noindent {\rm(ii)} $S$ contains no $Q$-shallow points, and

\noindent{\rm(iii)} $U^{S,Q}+Q$ is constant q.e. on $S$. 
\label{cor-5.16}
\end{cor}

By Sobolev imbedding, we have $W^{2,p}\subset C^1$ for $2<p\le+\infty$. The 
following characterization will prove useful later.

\begin{prop} 
{\rm($0<t<+\infty$)} 
Suppose $S\subset\C$ is compact with $S=\clos\,\Int\,S$, and 
that $Q\in W^{2,p}$ in a neighborhood of $S$, for some $p$, $2<p<+\infty$. 
We then have:

\noindent{\rm(i)} If $S$ is a local $Q$-droplet, then
$\bar\partial(U^{Q,S}+Q)=0$ on $S$.

\noindent{\rm (ii)} If $S$ is connected and $\bar\partial(U^{Q,S}+Q)=0$  
on $\partial S$, and if $S$ has no $Q$-shallow points, then $S$ is a 
local $Q$-droplet.
\end{prop}

\begin{proof} We first treat part (i). So, we assume that $S$ is a local 
$Q$-droplet. By Corollary \ref{cor-5.16}, $U^{Q,S}+Q$ is constant q.e. on $S$. 
As both $Q$ and $U^{Q,S}$ are in $W^{2,p}$ in a neighborhood of $S$,
we conclude from \cite{KS}, p. 53, that $\bar\partial(U^{Q,S}+Q)=0$
a.e. on $S$. By Sobolev imbedding, $\bar\partial(U^{Q,S}+Q)$ is continuous
in a neighborhood of $S$, and so $\bar\partial(U^{Q,S}+Q)=0$ on $\Int\, S$ and
a fortiori (by the topological assumption) on $S$.

We turn to part (ii). 
Consider the function  $F:=\bar\partial(U^{Q,S}+Q)$, which is in $W^{1,p}$ in
a neighborhood of $S$, and therefore continuous. We have 
$$\partial F=\Delta(U^{Q,S}+Q)=-1_S\Delta Q+\Delta Q=0
\quad\text{a.e. on}\,\,\,S. $$
Hence $F$ is conjugate holomorphic in the interior of $S$ and since $F=0$ on 
the boundary, we have $F\equiv0$ on $S$. If  $S$ is connected, then this 
implies that $U_S+Q$ is constant on $S$, so by Corollary \ref{cor-5.16}, 
$S$ is a local $Q$-droplet.
\end{proof}

\section{Chains of local droplets}
\label{sec-cld}

%\section{A partial ordering of local droplets}

\subsection{A partial ordering of local droplets}
We recall that $S$ is a local $Q$-droplet if it is a local $(Q,t)$-droplet for
some $t$ with $0<t<+\infty$. For the concept to make sense, we need to 
ask that $Q:\C\to\R\cup\{+\infty\}$ is lower semi-continuous and 
$W^{2,1}$-smooth near $S$. 
Given a local $Q$-droplet $S$, the corresponding value of (the evolution 
parameter) $t$ is easily calculated:
\[
t=t(Q,S):=\int_S\Delta Q\diff A.
\] 
We note that by Corollary \ref{cor-5.16}, $\Delta Q\ge0$ on $S$.
To simplify the presentation, we shall {\em assume that $Q$ is $W^{2,p}$-smooth
in $\C$ for some $p$, $1<p<+\infty$}.

\begin {lem} Let $S_2$ be a local $Q$-droplet, with $t_2=t(Q,S_2)$. If 
$t_1$ has $0<t_1<t_2$ we put $S_1:=S_{t_1}[Q, S_2]$. Then $S_1$ is a local
$Q$-droplet, with $t_1=t(Q,S_1)$.
\label{lem-6.1}
\end{lem}

\begin{proof} 
We should study the measure $\sigma_1:=\sigma_{t_1}[Q, S_2]$, which by
Proposition \ref{prop-4.2} and Corollary \ref{cor-4.4} is obtained
from $\Obs_{t_1}[Q_{S_2}]$ by applying the Laplacian. 
From $t_1<t_2$ and the definition of the obstacle problem, we see that
\[
\Obs_{t_1}[Q_{S_2}]\le\Obs_{t_2}[Q_{S_2}]=\widehat Q_{S_2},
\]
where we use Corollary \ref{cor-5.13} to get the rightmost identity.
A moments reflection, using that $\widehat Q_{S_2}\le Q$, reveals that in fact
\begin{equation*}
\Obs_{t_1}[Q_{S_2}]=\Obs_{t_2}[\widehat Q_{S_2}].
%\label{eq-6.1}
\end{equation*}
Since $\Delta Q$ is in $L^p$ locally, $\widehat Q_{S_2}$ is $W^{2,p}$-smooth,
and by Theorem \ref{thm-4.9} with $\widehat Q_{S_2}$ in place of $Q$, we
get that $S_1$ is a local $Q$-droplet. 
\end{proof}

%We notice that by Corollary \ref{cor-5.13}, we have in the setting of
%\eqref{eq-6.1} that
%\begin{equation}
%\Obs_{t_1}[Q_{S_2}]=\widehat Q_{S_1}.
%\label{eq-6.2}
%\end{equation}
Lemma \ref{lem-6.1} allows us to introduce a partial ordering in the set of 
all local $Q$-droplets. 

\begin{defn}
Let $S_1,S_2$ be two local $Q$-droplets, and write $t_j=t(Q,S_j)$, $j=1,2$.
We write $S_1\prec S_2$ if $S_1\subset S_2$ and $S_1=S_{t_1}[Q, S_2]$.
\end{defn}

\begin{rem}
(a) In other words, $S_1\prec S_2$ if $S_1,S_2$ are local $Q$-droplets and
$S_1$ is a local $(Q,t_1)$-droplet with localization $S_2$, where 
$t_1=t(Q,S_1)$.
  
\noindent(b) It follows from the definition that if $S_1,S_2$ are 
$Q$-droplets with $S_1\prec S_2$, then $t(Q,S_1)\le t(Q,S_2)$. 

\noindent(c) If $S_1\prec S_2$ and $S_2\prec S_1$ for two local $Q$-droplets, 
then $S_1\subset S_2$ and $S_2\subset S_2$, and so $S_1=S_2$.
\end{rem}

\begin{prop}
Let $S_1,S_2$ be two local $Q$-droplets with $S_1\subset S_2$. 
Then $S_1\prec S_2$ if and only if $\widehat Q_{S_1}\le Q$ holds on $S_2$. 
\label{prop-6.4}
\end{prop}

\begin{proof}
This follows from Corollary \ref{cor-5.13}. After all, for local $Q$-droplets
we do not need to check conditions (i)-(iv); only (v) remains. Moreover, by
continuity and the fact that $Q$-droplets lack $Q$-shallow points, the q.e. 
statements hold everywhere.
\end{proof}

\begin{lem}
For a local $Q$-droplet $S$, we have $S\prec S$. 
\end{lem}

\begin{proof}
This follows from the definition of the ``$\prec$'' relation together with
Proposition \ref{prop-5.9}. 
\end{proof}

There is one more property we need to check to show that ``$\prec$'' defines a
partial ordering.

\begin{lem} 
If $S_1,S_2,S_3$ are three local $Q$-droplets with $S_1\prec S_2$ and 
$S_2\prec S_3$,  then $S_1\prec S_3$.
\end{lem}

\begin{proof} 
%By Lemma \ref{lem-6.1} 
%In view of Proposition \ref{prop-6.4}, we have that $\widehat Q_{S_1}\le Q$ on
%$S_2$, and $\widehat Q_{S_2}\le Q$ on $S_3$. We are to show that 
%$\widehat Q_{S_1}\le Q$ on $S_3$. Clearly, it suffices to show that 
%$\widehat Q_{S_1}$
If we use that $S_2\prec S_3$, we see from Lemma \ref{lem-6.1} that
 $S_{t_1}[Q,S_3]$ is a local $Q$-droplet with
\[
S_{t_1}(Q, S_3)\subset S_2=S_{t_2}(Q, S_3).
\] 
Using that $S_1\prec S_2$, we appeal to Lemma \ref{lem-5.2}, and get
\[
S_{t_1}[Q, S_3]=S_{t_1}[Q, S_2]=S_1,
\]
so that $S_1\prec S_3$, as claimed.
\end{proof}

\begin{rem}  
$S_1\subset S_2$ does not imply $S_1\prec S_2$. For example, 
suppose $Q$ has two global minima at the points $0$ and $2$, and suppose the 
minima are non-degenerate. Consider
\[
S_1:=S_{t_1}(Q, \Sigma_1),\qquad S_2:=S_{t_2}(Q,\Sigma_2),\qquad \text{where}
\quad 0<t_1\ll t_2\ll 1,
\]
with $\Sigma_1=\bar\D(0,1)$ and $\Sigma_2=\bar\D(0,3)$. 
Then $S_1\subset S_2$ but $S_1\not\prec S_2$. This is easy to see using the
characterization of Proposition \ref{prop-6.4}.
\label{rem-6.7}
\end{rem}

\subsection{A comparison principle}
We keep the setting of the previous subsection.
We recall the definition of the polynomially convex hull $\Phull(E)$ of a 
compact set $E$ from Subsection \ref{subsec-coinc}. The set $\Phull(E)\sm E$
is the union of all the bounded components of $\C\sm E$. 

\begin{prop}
Suppose $S_1,S_2$ are two local $Q$-droplets with $S_1\subset S_2$. We then
have $\widehat Q_{S_2}\le \widehat Q_{S_1}$ on $\Phull(S_1)$, with equality on 
$S_1$. Moreover, if for some $z_0\in\Int[\Phull(S_1)]$ we have 
$\widehat Q_{S_2}(z_0)=\widehat Q_{S_1}(z_0)$, then $\widehat Q_{S_2}=
\widehat Q_{S_1}$ holds on the component of $\Int[\Phull(S_1)]$
that contains $z_0$. 
\label{prop-6.15}
\end{prop}

\begin{proof}
The difference $\widehat Q_{S_2}-\widehat Q_{S_1}$ is in $W^{2,p}$ and 
therefore continuous, and it is subharmonic, as 
\[
\Delta [\widehat Q_{S_2}-\widehat Q_{S_1}]=1_{S_2\sm S_1}\Delta Q\ge0\quad
\text{a.e. on}\,\,\,\C.
\]
Moreover, by Corollary \ref{cor-5.13}, $\widehat Q_{S_2}=Q$ on $S_2$ and
$\widehat Q_{S_1}=Q$ on $S_1$, and so $\widehat Q_{S_2}-\widehat Q_{S_1}=0$ 
on $S_1$ as $S_1\subset S_2$. The inequality 
$\widehat Q_{S_2}-\widehat Q_{S_1}\le0$ now follows
from the maximum principle. The last assertion follows from the strong maximum
principle.
\end{proof}

We see that a local $Q$-droplet $S_2$ with $S_1\prec S_2$ does not grow
in the direction of the interior holes of $S_1$:

\begin{cor}
Suppose $S_1,S_2$ are two local $Q$-droplets with $S_1\prec S_2$. We then have 
$[S_2\sm S_1]\cap\Phull(S_1)=\emptyset$ and 
$\widehat Q_{S_2}=\widehat Q_{S_1}$ on $\Phull(S_1)$. 
\label{cor-6.16}
\end{cor}

\begin{proof}
If $S_1\prec S_2$ we have $\widehat Q_{S_1}\le Q$ on $S_2$ (cf. Proposition 
\ref{prop-6.4}), and we also have $\widehat Q_{S_2}=Q$ on $S_2$ (cf. Corollary
\ref{cor-5.13}). In view of Proposition \ref{prop-6.15}, it follows that if 
$z_0\in S_2\cap\Phull(S_1)$, then $\widehat Q_{S_2}(z_0)=
\widehat Q_{S_1}(z_0)$. So, if $z_0\in S_2\cap\Int[\Phull(S_1)]$,
another application of Proposition \ref{prop-6.15} shows that 
$\widehat Q_{S_2}=\widehat Q_{S_1}$ holds on the component
of $\Int[\Phull(S_1)]$ which contains $z_0$. 
Taking the Laplacian, we find that $1_{S_1}\Delta Q=1_{S_2}\Delta Q$ a.e.
on the component $\text{Comp}(z_0)$ of 
$\Int[\Phull(S_1)]$ which contains $z_0$, which leads to
\[
S_1\cap \text{Comp}(z_0)=S_2\cap\text{Comp}(z_0).
\]
Since $z_0\in S_2$ we also must have $z_0\in S_1$. 
We conclude that $[S_2\setminus S_1]\cap\Int[\Phull(S_1)]=\emptyset$, and
{\em a fortiori} $[S_2\setminus S_1]\cap\Phull(S_1)=\emptyset$. But then
$\widehat Q_{S_2}-\widehat Q_{S_1}$ is harmonic in $\Int[\Phull(S_1)]$ and
vanishes on $\partial[\Phull(S_1)]\subset\partial S_1\subset S_1$, and the 
conclusion $\widehat Q_{S_2}-\widehat Q_{S_1}=0$ on $\Phull(S_1)$ is immediate.
\end{proof}

\subsection{Domination chains of local droplets}
We are interested in chains of local $Q$-droplets. 

\begin{defn}
A {\em domination chain} of local 
$Q$-droplets is a (continuously indexed) family  of $Q$-droplets 
$\{S_t\}_t$, where the index $t$ ranges over a nonempty interval 
$I\subset\R_+$, with left endpoint $0$, such that $t=t(Q,S_t)$ and
\[
t_1\le t_2\quad \Longleftrightarrow\quad S_{t_1}\prec S_{t_2}.
\]
The domination chain is {\em terminating} if the interval $I$ is given 
by $0<t\le t_*$, 
for some $T_*$ with $0<t_*<+\infty$, and {\em non-terminating} if it is given 
by $0<t<t_*$ for some $t_*$ with $0<t_*\le+\infty$. In case the domination 
chain is terminating, we say that it {\em terminates at} $S_{t_*}$.
\end{defn}

\begin{lem}
Given a local $Q$-droplet $S_*$, there is exactly one domination chain of 
local $Q$-droplets that terminates at $S_*$. 
\label{lem-6.11}
\end{lem}

\begin{proof}
By Lemma \ref{lem-6.1}, $S_t:=S_t[Q,S_*]$ for $0<t\le t_*:=t(Q,S_*)$ defines 
a continuously indexed collection of local $Q$-droplets, and by Lemma 
\ref{lem-5.2} it is a (terminating) domination chain. Finally, if 
$S_\sharp$ is a local  $Q$-droplet with $S_\sharp\prec S_*$, then by 
definition, it is of the form $S_\sharp=S_{t_\sharp}[Q,S_*]$ with 
$t_\sharp:=t(Q,S_\sharp)\le t_*$, so the domination chain is unique.
\end{proof}

\subsection{Maximal domination chains of local $Q$-droplets}
We keep the setting of the previous subsection.
We shall need the concept of a maximal domination chain of local $Q$-droplets.

\begin{defn}
A domination chain of $Q$-droplets is {\em maximal} if it is contained in no 
larger domination chain of local $Q$-droplets.
\end{defn}

Maximal domination chains of $Q$-droplets can be either terminating or 
non-terminating.
If the chain is indexed by the unbounded interval $I=\R_+$ then it is 
automatically non-terminating. If the chain is indexed by a bounded interval, 
then it can be non-terminating only if the droplets develop ``arms'' or 
``islands'' that tend to infinity:

\begin{thm} 
Let $\{S_t\}_{t\in I}$ be a maximal non-terminating domination chain of local 
$Q$-droplets. Then the union
\[
S_\cup:=\bigcup_{t\in I} S_t
\]
is an unbounded subset of $\C$. 
\label{thm-6.11}
\end{thm}

\begin{proof} 
We suppose $S_\cup$ is bounded, and form $S_*=\clos\,S_\cup$, which is then 
compact. We are to show that the non-terminating domination chain 
$\{S_t\}_{t\in I}$ cannot be maximal. The interval $I$ is given by 
$0<t<t_*$ for some $t_*$
with $0<t_*<+\infty$. For $t\in I$, we let $\sigma_t$ be the positive measure 
$\diff\sigma_t=1_{S_t}\Delta Q\diff A$, which has total mass $\|\sigma_t\|=t$.
Let $\sigma_*$ be given by 
$\diff\sigma_*=1_{S_\cup}\Delta Q\diff A$, which has total mass 
$\|\sigma_*\|=t_*$. Then $\sigma_t\to\sigma_*$ in norm as $t\to t_*$, and
in fact the corresponding densities converge in $L^p$:
\[
1_{S_t}\Delta Q\diff A\,\to\,1_{S_\cup}\Delta Q\diff A\quad\text{in}\,\,\,\,\,
L^p(\C)\,\,\,\,\,\text{as}\,\,\,\,t\to t_*.
\]
By the well-known properties of the 2D Hilbert transform, we find that
the associated potentials converge in $W^{2,p}$: $U^{Q,S_t}\to U^{Q,S_\cup}$
as $t\to t_*$. Also, we easily check that if the constants $\gamma^*(Q,S_t)$ 
and $\gamma^*(Q,S_\cup)$ are as in \eqref{eq-5.2}, we have 
$\gamma^*(Q,S_t)\to\gamma^*(Q,S_\cup)$ as $t\to t_*$. 
As a consequence,
\[
\widehat Q_{S_t}=\gamma^*_{t}(Q,S_t)-U^{Q,S_t}\to \widehat Q_{S_\cup}=
\gamma^*_{t_*}(Q,S_t)-U^{Q,S_\cup}\quad\text{in}\,\,\,\,\,
W^{2,p}\,\,\,\,\,\text{as}\,\,\,\,t\to t_*.
\]
By Sobolev imbedding the convergence is locally uniform. Since 
$\widehat Q_{S_t}=Q$ on $S_t$ we get that $\widehat Q_{S_\cup}=Q$ on $S_\cup$.
By continuity, then, we find that $\widehat Q_{S_\cup}=Q$ on 
$S_*=\clos\, S_\cup$. Next, by \cite{KS}, p. 53, we see that
$\Delta\widehat Q_{S_\cup}=\Delta Q$ a.e. on $S_*$, that is, 
$1_{S_\cup}\Delta Q=\Delta Q$ a.e. on $S_*$. Expressed differently, we have
$1_{S_\cup}\Delta Q=1_{S_*}\Delta Q$ as elements of $L^p(\C)$. In particular,
$\Delta Q\ge0$ holds a.e. on $S_*$. By construction, $S_*$ has no $Q$-shallow
points, a property this set inherits from the individual droplets $S_t$, 
$t\in I$. In view of Corollary \ref{cor-5.16}, $S_*$ is a local
$Q$-droplet. It remains to show that we may add $S_*$ as a terminal local
$Q$-droplet for the domination chain, thereby defeating the maximality of the 
non-terminating domination chain. To this end, it suffices to obtain that 
$S_t\prec S_*$ for $t\in I$. We pick a $t'$ with $t<t'<t_*$, and use 
$S_t\prec S_{t'}$ to deduce that $\widehat Q_{S_t}\le Q$ on $S_{t'}$ 
(Proposition \ref{prop-6.4}). 
By letting $t'\to t_*$, we get that $\widehat Q_{S_t}\le Q$ on $S_{\cup}$, 
and by continuity that $\widehat Q_{S_t}\le Q$ on $S_{*}$. By Proposition
\ref{prop-6.4} this means that $S_t\prec S_*$. The proof is finished. 
\end{proof}

The following definition is useful.

\begin{defn}
A local $Q$-droplet $S$ is {\em maximal} if for any other local $Q$-droplet 
$S'$ the relation $S\prec S'$ implies that $S=S'$. 
\end{defn}

\begin{cor}
A maximal domination chain $\{S_t\}_{t\in I}$ of local $Q$-droplets either 
terminates at a maximal local $Q$-droplet, or is non-terminating, in which 
case the set $S_\cup$ of Theorem \ref{thm-6.11} is unbounded.
\end{cor}

\subsection{Richardson's formula} 
We keep the setting of the previous two subsections.
We would like to understand the flow evolution $t\mapsto S_t$ of domination
chains (or containment chains, see the next subsection) of local $Q$-droplets.
A natural way to do this is to analyze the effect of the flow when we use
harmonic functions as test function (i.e., we calculate 
``harmonic moments''). 

\begin{prop} 
Suppose $S,S'$ are two local $Q$-droplets with $S\subset S'$. 
Then for all $h\in W^{2,1}(\C)$ (local Sobolev class) that are harmonic in 
$\C\sm S$
and bounded near infinity, we have (with $t=t(Q,S)$ and $t'=t(Q,S')$)
\[
\int_{S'\sm S} h\Delta Q\diff A=(t'-t)\,h(\infty).
\]
\label{prop-6.14}
\end{prop}

\begin{proof} The formula holds for constant $h$, by the choice of 
$t,t'$. So, by subtracting a constant, we may take $h(\infty)=0$. 
We have
\begin{equation}
\int_S h\Delta Q\diff A=\int_{\C}h\,1_S\Delta Q\diff A=\int_{\C} h
\Delta\widehat Q_S\diff A=\int_{\C}\widehat Q_S\,\Delta h\diff A,
\label{eq-6.1}
\end{equation}
and the analogous identity holds for $S'$ as well. By forming the difference 
between \eqref{eq-6.1} for $S$ and $S'$ we see that
\begin{equation*}
\int_{S'\sm S} h\Delta Q\diff A=\int_{\C}
\big[\widehat Q_{S'}-\widehat Q_S\big]\,\Delta h\diff A=0,
%\label{eq-6.1}
\end{equation*}
because $\Delta h=0$ on $\C\sm S$ while $\widehat Q_{S'}-\widehat Q_S=Q-Q=0$
on $S$ (see, e.g., Corollary \ref{cor-5.13}). To finish the proof, we just
need to justify \eqref{eq-6.1}. By Green's formula, we have 
\begin{equation}
\int_{\D(0,R)}\big[h\Delta\widehat Q_S-\widehat Q_S\Delta h\big]\diff A
=2\int_{\T(0,R)}
\big[h\partial_n\widehat Q_S-\widehat Q_S\partial_n h\big]\diff s,
\label{eq-6.2}
\end{equation}
where $\diff s$ is normalized arc length (i.e., arc length divided by $2\pi$) 
and $\partial_n$ is the exterior normal 
derivative. Next, we observe that as $|z|\to+\infty$, we have the asymptotics
\[
h=\Ordo(|z|^{-1}),\quad |\nabla h|=\Ordo(|z|^{-2}),\quad 
\widehat Q_S=\Ordo(\log |z|),\quad 
|\nabla\widehat Q_S|=\Ordo(|z|^{-1}),
\]
because both $h$ and $\widehat Q_S$ are harmonic in $\C\sm S$ with given
asymptotical behavior. By letting $R\to+\infty$ in \eqref{eq-6.2}, we obtain 
\eqref{eq-6.1}.
The proof is complete.
\end{proof}

\begin{cor} Suppose $S,S'$ are two local $Q$-droplets with $S\subset S'$. 
Also suppose that the interior $\Int\, S$ has finitely many components. 
Then for all $h$ continuous and bounded in $\C\sm\Int\,S$, which are harmonic
in $\C\sm\clos\,\Int\,S$, we have (with $t=t(Q,S)$ and $t'=t(Q,S')$)
\[
\int_{S'\sm S} h\Delta Q\diff A=(t'-t)\,h(\infty).
\]
\label{cor-6.15}
\end{cor}

\begin{proof} 
By  Mergelyan-type approximation we can find a sequence of 
bounded $C^\infty$-smooth functions $h_n$ that are are harmonic in 
$\C\sm\clos\,\Int\,S$, such that $h_n\to h$ uniformly on $\C\sm\Int\,S$
as $n\to+\infty$. The assertion now follows from Proposition \ref{prop-6.14}.
\end{proof}

\subsection{Differential form of Richardson's formula}
We keep the setting of the previous subsections, and introduce the concept of
a containment chain. 

\begin{defn}
A {\em containment chain} of local $Q$-droplets is a (continuously indexed) 
family  of $Q$-droplets $\{S_t\}_t$, where the index $t$ ranges over 
a nonempty interval  $I\subset\R_+$, with left endpoint $0$, such that 
$t=t(Q,S_t)$ and
\[
t_1\le t_2\quad \Longleftrightarrow\quad S_{t_1}\subset S_{t_2}.
\]
\end{defn}

Let $\{S_t\}_{t\in I}$ be a containment chain of local $Q$-droplets, and
let $t_*$ denote the right endpoint of $I$. Let $I^-$ be 
the interval obtained from $I$ by removal of $t_*$ (if $t_*\notin I$, we put 
$I^-:=I$).

\begin{lem}
The map $t\mapsto S_t$, $t\in I$, is continuous in the Hausdorff metric 
except for a countable subset of the interval $I$. 
\end{lem}

\begin{proof}
For $t_0\in I^-$, we form the compact sets
\[
S_{t_0}^-=\clos \bigcup_{t:t<t_0}S_t,\quad S_{t_0}^+=\bigcap_{t:t>t_0}S_t.
\]
Then $S_t^-\subset S_t\subset S_t^+$ holds for each for $t\in I^-$. 
Note that $S_{t_0}^-$ is well-defined also when $t_0=t_*$. 
For $t\in I^-$, we put 
\[
\delta^-(t):=\max_{z\in S_{t}}\dist_\C(z,S_{t}^-),\quad
\delta^+(t):=\max_{z\in S_{t}^+}\dist_\C(z,S_{t}).
\]
Next, for positive $\epsilon$, we consider the sets
\[
D^-_\epsilon:=\{t\in I^-:\,\delta^-(t)>\epsilon\},\quad
D^+_\epsilon:=\{t\in I^-:\,\delta^+(t)>\epsilon\}.
\]
We argue that for each positive $\epsilon$, the sets 
$D^-_\epsilon,D^+_\epsilon$ are countable and that the only possible 
accumulation point is $t_*$ (and if $t_*$ is an accumulation point, then
the set $S_{t_*}^-$ must be unbounded). Indeed, if, for some $t'$ with 
$0<t'<t_*$, the set $D^-_\epsilon\cap]0,t']$ has $N=N(t',\epsilon)$ elements,
then the local $Q$-droplet $S_{t'}$ contains $N$ points which are 
$\epsilon$-separated (the distance between any two different points is at
least $\epsilon$). We get an effective bound on $N$ in terms of the diameter
of the compact set $S_{t'}$. The analogous argument applies to $D^+_\epsilon$
in place of $D^-_\epsilon$.
\end{proof}

We let $\omega_\infty^{(t)}$ denote harmonic measure for the 
open set $\C\sm S_t$ with respect to the point at infinity. This is a 
probability measure whose support is contained in $\partial\Phull(S_t)\subset 
\partial S_t$ (the effect on a test function is that we get the value at 
infinity of the harmonic extension).

\begin{prop} 
Suppose the map $t\mapsto S_t$, $t\in I$, is right continuous 
at $t_0\in I^-$. Suppose moreover that $\Int\, S_{t_0}$ has finitely many 
components and that $S_{t_0}=\clos\,\Int\, S_{t_0}$. Then  for all  
$g\in C(\C)$ we have 
\[
\lim_{t\to t_0^+}~\frac1{t-t_0}\int_{S_{t}\sm S_{t_0}} g\,\Delta Q\diff A
=\int_{\partial S_{t_0}} g\, d\omega_\infty^{(t_0)},
\]
that is, we have the weak-star convergence of measures
\[
\lim_{t\to t_0^+}~\frac1{t-t_0}1_{S_t\sm S_{t_0}}\Delta Q\diff A
=\diff\omega_\infty^{(t_0)}.
\]
\label{prop-6.18}
\end{prop}

\begin{proof}  
Let $h$ denote the function which coincides with $g$ on $S_{t_0}$ and extends
harmonically (and boundedly) to $\C\sm S_t$, so that in particular
\[
h(\infty)=\int_{\partial S_{t_0}} g d\omega_\infty^{(t_0)}.
\]
Then $h$ is continuous and bounded in $\C$ (see, e.g. \cite{GM} for a
discussion of the Dirichlet problem). Since $S_{t_0}=\clos\,\Int\,S_{t_0}$,
Corollary \ref{cor-6.15} applied to to $h$ gives
\[
\frac1{t-t_0}\int_{S_{t}\sm S_{t_0}}g\Delta Q\diff A=
h(\infty)+\frac1{t-t_0}\int_{S_{t}\sm S_{t_0} } (g-h)\Delta Q\dA.
\]
It remains to show that the last term on the right hand side tends to zero 
as $t\to t_0$. This follows from the fact that $h(z)-g(z)\to 0$ as 
$z\to\partial S_{t_0}$ and that $S_{t}\searrow S_{t_0}$ by the right 
continuity assumption.
\end{proof}

\begin{rem}
Proposition \ref{prop-6.18} states that (under regularity assumptions) the
infinitesimal growth of the local $Q$-droplets is in the exterior direction
only.
If the containment chain of local $Q$-droplets were to grow in the direction of
the internal holes, the containment chain could not possibly be a domination 
chain (cf. Corollary \ref{cor-6.16}).
\end{rem}

\subsection{Richardson's inequality} We now show that under modest regularity
conditions, containment chains of local $Q$-droplets are in fact domination
chains. 

\begin{thm} Suppose $S,S'$ are two local $Q$-droplets, with $S\subset S'$.
Then the following are equivalent:

\noindent{\rm(i)} $S\prec S'$. 

\noindent{\rm(ii)} For all real-valued functions $h\in W^{2,1}(\C)$ 
(local Sobolev class) that are 
subharmonic in $\C\sm S$, harmonic in $\C\sm S'$, and bounded near infinity, 
we have (with $t=t(Q,S)$ and $t'=t(Q,S')$)
\[
(t'-t)h(\infty)\le\int_{S'\sm S} h\Delta Q\diff A.
\]
\label{thm-6.19}
\end{thm}

\begin{proof} 
We first show that (i)$\implies$(ii). 
We note that the inequality is an equality when $h$ is constant (see, e.g., 
Proposition \ref{prop-6.14}). This allows us to restrict our
attention to $h$ with $h(\infty)=0$.
As in the proof of Richardson's formula (Proposition \ref{prop-6.14}), we 
find that
\[
\int_{S'\sm S} h\Delta Q\diff A=\int_{\C} 
[\widehat Q_{S'}-\widehat Q_S]\Delta h\diff A
=\int_{S'\sm S}[\widehat Q_{S'}-\widehat Q_S]\Delta h\diff A\ge
0,
\]
since $\widehat Q_S\le Q=\widehat Q_{S'}$, by Proposition \ref{prop-6.4} and
Corollary \ref{cor-5.13}. 

We turn to the implication (ii)$\implies$(i). We take $h(\infty)=0$, and get
(as above) from (ii) that
\begin{equation}
0\le\int_{S'\sm S} h\Delta Q\diff A=
\int_{S'\sm S}[\widehat Q_{S'}-\widehat Q_S]\Delta h\diff A.
\label{eq-6.3}
\end{equation}
The question now is what kind of functions $\Delta h$ are possible here.
We have automatically $\Delta h\in L^1(S')$ while $\Delta h\ge0$ a.e.
on $S'\sm S$. We also need to impose that
\[
\int_{S'}\Delta h\diff A=0,
\]
as a consequence of the behavior of $h$ near infinity. In fact, any real-valued
function $g\in L^q(S')$ for some $q$ with $1<q<+\infty$ with 
$g\ge0$ a.e. on $S'\sm S$ and 
\begin{equation}
\label{eq-6.4}
\int_{S'}g\diff A=0,
\end{equation}
is of the form $g=\Delta h$ for an $h$ as in (ii) with $h(\infty)=0$. 
It follows from \eqref{eq-6.3} that
\begin{equation}
0\le\int_{S'\sm S}[\widehat Q_{S'}-\widehat Q_S]g\diff A.
\label{eq-6.5}
\end{equation}
As it is easy to fulfill \eqref{eq-6.4} by placing an $L^q$-integrable 
negative mass on $S$ to compensate for the positive mass on $S'\sm S$, on 
$S'\sm S$ the function $g$ is basically any positive $L^q$ function on 
$S'\sm S$. This is
only possible if $\widehat Q_{S}\le \widehat Q_{S'}$ on $S'\sm S$, and as 
$\widehat Q_{S'}=Q$ on $S'$, we get $\widehat Q_{S}\le Q$ on $S'\sm S$. Since
$\widehat Q_{S}\le Q$ holds automatically on $S$, we see that 
$\widehat Q_{S}\le Q$ on $S'$. The conclusion $S\prec S'$ now follows from 
Proposition \ref{prop-6.4}. 
The proof is complete.
\end{proof}

The proof of Theorem \ref{thm-6.19} has the following consequence.

\begin{cor}
Suppose $S'$ is a local $Q$-droplet, and that $S\subset S'$, where $S$ is
compact and lacks $Q$-shallow points. If condition {\rm(ii)} of Theorem 
\ref{thm-6.19} is fulfilled, then $S$ is a local $Q$-droplet, and $S\prec S'$.
\label{cor-6.23}
\end{cor}

\begin{proof}
As in the proof of Theorem \ref{thm-6.19}, we get from condition (ii) of that
theorem
\[
0\le\int_{S'\sm S}h\Delta Q\diff A=\int_{\C}[U^{Q,S}-U^{Q,S'}]\Delta h\diff A
=\int_S[U^{Q,S}-U^{Q,S'}]\Delta h\diff A+\int_{S'\sm S}
[U^{Q,S}-U^{Q,S'}]\Delta h\diff A,
\] 
provided that $h\in W^{2,1}(\C)$ (local Sobolev class)
is subharmonic in $\C\sm S$ and harmonic in
$\C\sm S'$, and bounded near infinity, with $h(\infty)=0$.  
As in the proof of Theorem \ref{thm-6.19}, we choose $h$ as (minus) the 
logarithmic potential of $g$, where $g\in L^q(S')$ has $g\ge0$ on $S'\sm S$
and 
\[
\int_{S'}g\diff A=0.
\]
so that
\begin{equation}
0\le\int_S[U^{Q,S}-U^{Q,S'}]g\diff A+\int_{S'\sm S}
[U^{Q,S}-U^{Q,S'}]g\diff A,
\label{eq-6.6.0}
\end{equation}
If we choose $g$ such that $g=0$ on $S'\sm S$, we have equality (since then
the inequality applies to $-g$ as well):
\[
\int_S[U^{Q,S}-U^{Q,S'}]g\diff A=0.
\] 
As $g$ is now arbitrary except that its integral over $S$ vanishes, 
we conclude that $U^{Q,S}-U^{Q,S'}$ is constant on $S$. Call the constant $c$:
$U^{Q,S}=c+U^{Q,S'}$ on $S$. Since $S'$ is a local $Q$-droplet, the
function $U^{Q,S'}+Q$ is constant on $S'$ (cf. Corollary \ref{cor-5.16})
and consequently, $U^{Q,S}$ is constant on $S$ (after all, $S\subset S'$). 
We conclude that $S$ is a local $Q$-droplet. That $S\prec S'$ now follows from
Theorem \ref{thm-6.19}.
\end{proof}

\begin{thm} 
Let $\{S_t\}_{t\in I}$ be a containment chain of local $Q$-droplets. If, for 
almost every $t\in I$, the set $\Int\, S_t$ has finitely many components and 
$S_t=\clos\,\Int\,S_t$, then  $\{S_t\}_{t\in I}$ is a domination chain. 
\label{thm-6.21}
\end{thm}

\begin{proof} 
We consider the $W^{2,p}$-smooth function 
\[
V(\xi,\eta;t):=U^{Q,S_t}(\xi)-U^{Q,S_t}(\eta)=\int_{S_t}\log
\bigg|\frac{z-\xi}{z-\eta}\bigg|^2\Delta Q(z)\diff A(z),\qquad t\in I,
\]
and note that
\begin{equation}
V(\xi,\eta;t)=\widehat Q_{S_t}(\eta)-\widehat Q_{S_t}(\xi)
,\qquad t\in I.
\label{eq-6.6}
\end{equation}
Let $\mu,\nu$ be two compactly supported Borel probability measures which are 
absolutely continuous with densities in $L^q$ for some $q$, $1<q<+\infty$. 
We need the expression
\[
V^{\mu,\nu}(t):=\int_{C^2}V(\xi,\eta;t)\diff\mu(\xi)\diff\nu(\eta)
=\int_{S_t}U^{\nu-\mu}\Delta Q\diff A,\qquad t\in I,
\]
where $U^{\nu-\mu}:=U^\nu-U^\mu$, and $U^\mu,U^\nu$ are the usual logarithmic
potentials. The functions $U^\mu,U^\nu$ are in $W^{2,q}$ and therefore 
continuous (and bounded). 
The function $U^{\nu-\mu}$ is harmonic off $\supp(\nu-\mu)$, and its value 
at infinity is $U^{\nu-\mu}(\infty)=0$. 
We have 
\begin{equation}
V^{\mu,\nu}(t')-V^{\mu,\nu}(t)=
\int_{S_{t'}\sm S_t}U^{\nu-\mu}\Delta Q\diff A,\quad\text{for}\,\,\, 
t,t'\in I \,\,\,\text{with}\,\,\,t<t',
\label{eq-6.7}
\end{equation}
which gives
\[
|V^{\mu,\nu}(t')-V^{\mu,\nu}(t)|\le\|U^{\nu-\mu}\|_{L^\infty(\C)}
\int_{S_{t'}\sm S_t}\Delta Q\diff A=(t'-t)
\|U^{\nu-\mu}\|_{L^\infty(\C)},\quad\text{for}\,\,\, 
t,t'\in I \,\,\,\text{with}\,\,\,t<t',
\]
since $\Delta Q\ge0$ a.e. on a local $Q$-droplet.
It follows that the function $V^{\mu,\nu}$ is Lipschitz continuous, and
therefore differentiable almost everywhere. In view of \eqref{eq-6.7}, its 
right derivative is
\[
[V^{\mu,\nu}]'(t^+)=\lim_{t'\to t^+}\frac1{t'-t}
\int_{S_{t'}\sm S_{t}}U^{\nu-\mu}\Delta Q\diff A
=\int_{\partial S_{t}^\infty}U^{\nu-\mu}\diff\omega_\infty^{(t)},
\]
by Proposition \ref{prop-6.18}, with the possible exception of a countable set
of $t$'s. If now $\supp\nu\subset S_t$, the function $U^{\nu-\mu}$ becomes
subharmonic (and bounded) in $\C\sm S_t$, so by the maximum principle
\[
0=U^{\nu-\mu}(\infty)\le\int_{\partial S_{t}^\infty}
U^{\nu-\mu}\diff\omega_\infty^{(t)}.
\]
We conclude that $[V^{\mu,\nu}]'(t)\ge0$ for a.e. $t$ with 
$\supp\nu\subset S_t$. Put 
\[
t_\nu:=\inf\{t\in I:\,\supp\nu\subset S_t\},
\]
and note that for $t\in I$ with $t>t_\nu$ we have $[V^{\mu,\nu}]'(t)\ge0$
almost everywhere, and hence $V^{\mu,\nu}$ is increasing on that sub-interval:
\begin{equation}
V^{\mu,\nu}(t)\le V^{\mu,\nu}(t')\quad \text{for}\,\,\,t,t'\in I\,\,\,
\text{with}\,\,\,t_\nu<t<t'.
\label{eq-6.8}
\end{equation}
Next, we let the probability measures $\mu,\nu$ get more and more concentrated,
so that $\supp\mu\to\{\xi\}$ and $\supp\nu\to\{\eta\}$. The inequality 
\eqref{eq-6.8} survives the limit process, and we obtain that
\begin{equation}
V(\xi,\eta;t)\le V(\xi,\eta;t')\quad \text{for}\,\,\,t,t'\in I\,\,\,
\text{with}\,\,\,t_\xi<t<t',
\label{eq-6.9}
\end{equation}
where
\[
t_\xi:=\inf\{t\in I:\,\xi\in S_t\}.
\]
The short argument which justifies this involves choosing the support of $\nu$
cleverly, and this is made possible by the fact that a local 
$Q$-droplet lacks $Q$-shallow points. If we use \eqref{eq-6.6}, we see that
\eqref{eq-6.9} expresses that
\begin{equation}
\widehat Q_{S_{t}}(\eta)-\widehat Q_{S_{t}}(\xi)
\le \widehat Q_{S_{t'}}(\eta)-\widehat Q_{S_{t'}}(\xi)
\quad \text{for}\,\,\,t,t'\in I\,\,\,
\text{with}\,\,\,t_\xi<t<t'.
\label{eq-6.10}
\end{equation} 
Since for $t_\xi<t<t'$ we have $\xi\in S_t\subset S_{t'}$, we get that
(cf. Proposition \ref{prop-4.5})
\[
\widehat Q_{S_{t}}(\xi)=\widehat Q_{S_{t'}}(\xi)=Q(\xi),
\]
so that \eqref{eq-6.10} simplifies:
\begin{equation*}
\widehat Q_{S_{t}}(\eta)
\le \widehat Q_{S_{t'}}(\eta)
\quad \text{for}\,\,\,t,t'\in I\,\,\,
\text{with}\,\,\,t_\xi<t<t'.
%\label{eq-6.11}
\end{equation*} 
By making clever choices of the point $\xi$ we can get $t_\xi$ to be as close
to $0$ as we need, and so
\begin{equation*}
\widehat Q_{S_{t}}(\eta)
\le \widehat Q_{S_{t'}}(\eta)
\quad \text{for}\,\,\,t,t'\in I\,\,\,
\text{with}\,\,\,t<t'.
%\label{eq-6.11}
\end{equation*} 
For $\eta\in S_{t'}$ we have $\widehat Q_{S_{t'}}(\eta)=Q(\eta)$, and we 
derive that for $t,t'\in I$ with $t<t'$, we have
\begin{equation*}
\widehat Q_{S_{t}}(\eta)
\le Q(\eta)
\qquad \eta\in S_{t'}.
%\label{eq-6.11}
\end{equation*}
By Proposition \ref{prop-6.4}, we get $S_{t}\prec S_{t'}$ for all 
$t,t'\in I$ with $t<t'$, and $\{S_t\}_{t\in I}$ is a domination chain.
\end{proof}

\section{The Hele-Shaw equation}
\label{sec-hse}

%\subsection{Weighted Hele-Shaw equation} 
%We first recall the classical equation.

\subsection{Smooth curve families (laminations)}
We need the following definition.

\begin{defn}
A family  of simple curves $\Gamma_t$ (where $t$ runs over some interval) 
in $\C$ is a $C^\infty$-smooth {\em lamination} if 

\noindent(i) $\Gamma_t\cap\Gamma_{t'}=\emptyset$ holds for $t\neq t'$, and

\noindent(ii) Each curve has a local parametrization $z=\gamma_t(\theta)$
($\theta$ runs over some interval), such that the function 
$\gamma(\theta,t):=\gamma_t(\theta)$ is a local $C^\infty$-diffeomorphism.
\end{defn}

We will alternatively use the term {\em $C^\infty$-smooth curve family} as
synonymous to $C^\infty$-smooth lamination. We mention that it is of course 
also possible to define laminations with a lower degree of smoothness than 
$C^\infty$. 
The normal velocity $v_n=v_n(z)$, $z\in \Gamma_t$, may be defined as follows:
\[
v_n:=\langle\partial_t\gamma, n\rangle=
\frac1{|\partial_\theta\gamma|}\,
\im[\partial_t\gamma\partial_\theta\bar\gamma],
\]
where the inner product is that of $\C\cong\R^2$ and $n$ is a unit normal 
to $\Gamma_t$. 
It is easy to see that the definition does not depend on the choice 
of parametrization $\gamma$. Indeed, if we write
\[
\ti\gamma_t(\vartheta)=\ti\gamma(\vartheta,t):=\gamma(\theta(\vartheta,t),t),
\] 
where $\vartheta\mapsto \theta(\vartheta,t)$ is a local diffeomorphism, then
\[
\partial_t \ti\gamma=\partial_\theta\gamma\,\partial_t \theta+
\partial_t\gamma,\qquad 
\partial_\vartheta \ti\gamma=\partial_\theta\gamma\,\partial_\vartheta \theta,
\]
so that
\begin{multline*}
\frac1{|\partial_\vartheta\ti\gamma|}\,
\im[\partial_t\ti\gamma\partial_\vartheta
\overline{\ti\gamma}]=\frac1{|\partial_\theta\gamma\partial_\vartheta\theta|}\,
\im\big[|\partial_\theta\gamma|^2\partial_t \theta\partial_\vartheta\theta+
\partial_t\gamma\partial_\theta\bar\gamma\partial_\vartheta\theta\big]
\\
=\frac{\partial_\vartheta\theta}{|\partial_\vartheta\theta|}\,
\frac1{|\partial_\theta\gamma|}\,\im[\partial_t\gamma\partial_\theta\bar\gamma]
=\pm \frac1{|\partial_\theta\gamma|}\,
\im[\partial_t\gamma\partial_\theta\bar\gamma],
\end{multline*}
where there is a sign change if the coordinate change reverses the direction
of the unit normal vector.

\begin{lem} Let $\Gamma_t$ be a $C^\infty$-smooth lamination of Jordan curves,
such that the domain $D_t$ interior to $\Gamma_t$ increases with 
$t$. Then, for continuous $f:\C\to\C$, we have
\[
\frac\diff{\diff t}\int_{D_t}f\diff A=2\int_{\Gamma_t}fv_n~\diff s.
\]
\end{lem}

\begin{proof} 
We identify the area form with the area measure according to, e.g., 
$\diff z\wedge\diff\bar z=2\pi\imag\diff A(z)$. We may assume that for $t,t_0$
close to one another with $t_0<t$, $D_{t}\sm D_{t_0}$ is parametrized by 
$\gamma_\tau(\theta)=\gamma(\theta,\tau)$ where $0\le\theta\le1$ and 
$t_0\le\tau<t$, with
periodicity boundary conditions in $\theta$: $\gamma(0,\tau)=\gamma(1,\tau)$.
We let $R(t_0,t)$ denote the rectangle $[0,1]\times[t_0,t]$, so that
\[
\int_{D_t\sm D_{t_0}}f\diff A=\frac{1}{2\pi\imag}\int_{R(t_0,t)}
f(\gamma(\theta,\tau))\,\diff\gamma\wedge\diff\bar\gamma.
\]
We calculate:
\[
\diff\gamma\wedge d\bar\gamma=[\partial_\theta\gamma\partial_t\bar\gamma-
\partial_\theta\bar\gamma\partial_t\gamma]\,\diff\theta\wedge\diff t
=2\imag\im[\partial_\theta\gamma\partial_t\bar\gamma]\,\diff\theta\wedge\diff t
=2\imag|\partial_\theta\gamma|v_n\diff\theta\diff t,
\]
where we have identified a form with the corresponding measure. We identify
$|\partial_\theta\gamma|\diff\theta$ as arc length along $\Gamma_t$, so that
$|\partial_\theta\gamma|\diff\theta=2\pi\diff s(\theta)$, and therefore,
\[
\frac{1}{2\pi \imag}\diff\gamma\wedge\diff\bar\gamma=
2 v_n\diff s(\theta)\diff t.
\]
The assertion is now immediate.
\end{proof}

\subsection{The Hele-Shaw flow equation}
We assume we have a $C^\infty$-smooth lamination of Jordan curves $\Gamma_t$, 
and let $D_t$ denote the interior domain while $\Omega_t$ is the exterior
(unbounded) domain. We also write $K_t:=\clos\,D_t=\C\sm\Omega_t$, so that 
$K_t$ is compact.
The classical {\em Hele-Shaw equation} relates the
normal velocity $v_n$ to the normal derivative of the Green function (for the
Laplacian) of the exterior domain $\Omega_t$ when one of the two coordinates
is the point at infinity (the factor $\frac14$ comes from our choice of 
normalizations):
\begin{equation}
v_n=\tfrac14\partial_n G_t\quad\text{on}\,\,\,\Gamma_t,\,\,\,\text{where}\,\,\,
G_t=G(\cdot,\infty;\Omega_t).
\label{eq-7.1}
\end{equation}
The Green function $G_t$ is always positive in $\Omega_t$ and vanishes 
along the boundary $\Gamma_t$, and $n$ is taken in the exterior direction, so
that $\partial_n G_t>0$ on $\Gamma_t$. Actually, $\partial_n G_t$ is the
Poisson kernel of $\Omega$ for the point at infinity, so that 
$\frac12\partial_n G_t$ times normalized arc length measure has the 
interpretation of $\diff\omega^{(t)}_\infty$, 
harmonic measure at infinity for the domain $\Omega_t$.
There is also a weighted analog of \eqref{eq-7.1}: the 
{\em weighted Hele-Shaw equation} is
\begin{equation}
\rho\,v_n=\tfrac14\partial_n G_t\quad\text{on}\,\,\,\Gamma_t.
\label{eq-7.2}
\end{equation}
The function $\rho$ is the weight, and it is assumed to be $C^\infty$-smooth
with $\rho>0$ point-wise. It is possible to interpret the introduction of the 
weight as a change of the geometry (cf. \cite{HS}, \cite{HO}, \cite{HP}). 
In the sequel, we will use $\rho=\Delta Q$. 

\begin{defn}
We say that an increasing family of compact sets $\{K_t\}_t$ (where $t$ 
ranges over some interval) is a {\it generalized} solution of the weighted 
Hele-Shaw equation with weight $\rho=\Delta Q$ if 

\noindent(i)
$\Delta Q\ge0$ on $\cup_t K_t$, if 

\noindent(ii) for each $t\in I$, $K_t$ lacks 
$Q$-shallow points, and if, 

\noindent(iii) for all $f\in C(\C)$, the function
\[
t\mapsto \int_{K_t}f\Delta Q\diff A
\]
is absolutely continuous and for a.e. $t$ we have 
($\Gamma_t=\partial\Omega_t^\infty$ where $\Omega_t^\infty:=\C\sm\Phull(K_t)$ 
is the unbounded component of the complement $\C\sm K_t$ and 
$\omega_\infty^{(t)}$ is harmonic measure at infinity for $\Omega_t^\infty$) 
\[
\frac {\diff}{\diff t}\int_{K_t}f\Delta Q\diff A=
\int_{\Gamma_t}f\diff\omega^{(t)}_\infty.
\]
\end{defn}

Note that no smoothness requirement is imposed on the compact sets $K_t$ as
in the standard formulation of the weighted Hele-Shaw equation \eqref{eq-7.2}.
The way things are set up, strong solutions of the weighted Hele-Shaw equation 
(i.e., solutions of \eqref{eq-7.2}) are automatically generalized solutions.
In short, the equation asks that the compact sets $K_t$ grow according to the 
law
\[
\frac{\diff}{\diff t}\,[1_{K_t}\Delta Q\diff A]=\diff\omega^{(t)}_\infty.
\]

\begin{prop}
Let $\{K_t\}_t$ be an increasing family of compact sets, where $t$ ranges 
over an open interval $I$, and suppose $\Delta Q\ge0$ on $\cup_{t\in I}K_t$,
and that $K_t$ lacks $Q$-shallow points, for each $t\in I$.   
Then $\{K_t\}_{t\in I}$ is a generalized solution of the weighted Hele-Shaw 
equation with weight $\Delta Q$ if and only if, for all $t,t'\in I$ with 
$t<t'$, and for all real-valued $f\in C(\C)$,  
\[
\int_{K_{t'}\sm K_t}f\Delta Q\diff A=
\int_{t}^{t'}\int_{\Gamma_\tau}f\diff\omega^{(\tau)}_\infty\diff\tau.
\]
\end{prop}

\begin{proof}
This is just an application of Calculus. 
\end{proof}

So, the weighted Hele-Shaw equation corresponds to the disintegration of
measures
\[
1_{K_{t'}\sm K_t}\Delta Q\diff A=\int_{t}^{t'}\diff\omega^{(\tau)}_\infty
\diff\tau.
\]
It follows from the standard properties of the harmonic measure that
if $f\in C(\C)$ is real-valued, bounded, and subharmonic in $\C\sm K_t$
while it is harmonic near infinity, then
\begin{equation}
\int_{K_{t'}\sm K_t}f\Delta Q\diff A=\int_{t}^{t'}
\int_{\Gamma_\tau}f\diff\omega^{(\tau)}_\infty\diff\tau
\ge\int_{t}^{t'}f(\infty)\diff\tau=(t'-t)f(\infty).
\label{eq-7.3}
\end{equation}
This strongly resembles Richardson's inequality for local $Q$-droplets 
(Theorem \ref{thm-6.19}). The comparison with Theorem \ref{thm-6.19} suggests 
the concept of a weak solution to the Hele-Shaw equation.

\begin{defn}
We say that an increasing family of compact sets $\{K_t\}_t$ (where $t$ 
ranges over some interval) is a {\it weak} solution of the weighted 
Hele-Shaw equation with weight $\Delta Q$ if 

\noindent(i)
$\Delta Q\ge0$ on $\cup_t K_t$, if 

\noindent(ii) for each $t\in I$, $K_t$ lacks 
$Q$-shallow points, and if, 

\noindent(iii) for all real-valued $f\in W^{2,1}(\C)$ (local Sobolev class), 
\[
(t'-t)f(\infty)\le \int_{K_{t'}\sm K_t}f\Delta Q\diff A\quad\text{for}\,\,\,
t,t'\in I \,\,\,\text{with}\,\,\,t<t',
\]
provided $f$ is subharmonic in $\C\sm K_t$, harmonic in 
$\C\sm K_{t'}$, and bounded near infinity.
\label{def-7.5}
\end{defn}

\begin{prop}
A generalized solution of the Hele-Shaw equation is a weak solution.
\end{prop}

\begin{proof}
It is known that it suffices to have condition (iii) of Definition 
\ref{def-7.5} fulfilled for $f\in W^{2,q}$ for some $q$ slightly bigger than 
$1$. Such functions are continuous, so the assertion is immediate from 
\eqref{eq-7.3}. 
\end{proof}

We note that the sets $K_t$ need not be local $Q$-droplets, although that 
is one
particular instance. The analogy with that case suggest the following.

\begin{defn}
An increasing family of compact sets $\{K_t\}_t$ is {\it correctly indexed} 
if 
\[
t=t(Q,K_t)=\int_{K_t}\Delta Q\diff A.
\]
\end{defn}

This is in agreement with \eqref{eq-7.3} (or with Definition \ref{def-7.5}) 
for $f\equiv1$ 
(since the inequality applies to $f\equiv-1$ as well the inequality is of 
course an equality).

It is known that the Hele-Shaw equation behaves like the heat equation, in
that one direction of time $t$ is stable and the other is unstable. Here,
the evolution $t\mapsto K_t$ is unstable when $t$ increases, and stable when
$t$ decreases. 

\begin{thm} 
Let $K_*\subset\C$ be compact, with $\Delta Q\ge0$ a.e. on $K_*$. 
We assume that $t_*:=t(Q,K_*)>0$, and that $K_*$ lacks $Q$-shallow points. 
Then there exists a correctly indexed weak solution $t\mapsto K_t$ of the 
weighted Hele-Shaw equation with weight $\Delta Q$ on the interval 
$0<t\le t_*$, such that $K_{t_*}=K_*$. The solution is unique. 
\label{thm-7.7}
\end{thm}

\begin{proof}
We consider the function $\widetilde Q_*:=-U^{Q,K_*}$, where $U^{Q,K_*}$ is as 
in \eqref{eq-UQS}. It has 
\[
\Delta \widetilde Q_*=1_{K_*}\Delta Q,\qquad 
\widetilde Q_*(z)=t_*\log|z|^2+\Ordo(1)\,\,\,\text{as}\,\,\,|z|\to+\infty,
\]
and we can define
\begin{equation}
K_t:=S_t[\widetilde Q_*,K_*],\qquad 0<t<t_*,
\label{eq-7.4}
\end{equation}
and $K_{t_*}:=K_*$. The way things are set up, $K_*$ becomes a 
$\widetilde Q_*$-droplet (cf. Corollary \ref{cor-5.16}). Moreover, in view of 
Lemma \ref{lem-6.1} (with $\widetilde Q_*$ in place of $Q$), the sets $K_t$ 
are local $\widetilde Q_*$-droplets. We see from Theorem \ref{thm-6.19} that 
the sets $K_t$ form a domination chain of local $\widetilde Q_*$-droplets 
if and only if they form a weak solution of the Hele-Shaw equation. 

It remains to establish that the weak solution $t\mapsto K_t$ 
unique. So, suppose $t\mapsto K_t$ is a weak solution, which need not be of
the form \eqref{eq-7.4}. We claim that $K_t$ is a $\widetilde Q_*$-droplet
for $0<t<t_*$. We know that $K_*$ is a local $\widehat Q_*$-droplet, 
that $K_t$ has no $Q$-shallow points, and that $K_t\subset K_*$. In addition,
the weak solution condition entails
\[
(t_*-t)f(\infty)\le \int_{K_*\sm K_t}f\Delta Q\diff A,\qquad 0<t<t_*,
\] 
provided $f\in W^{2,1}(\C)$ (local Sobolev class) is real-valued, 
subharmonic in $\C\sm K_t$, harmonic in $\C\sm K_*$, and bounded near infinity.
An application of Corollary \ref{cor-6.23} shows that $K_t$ must also be 
a local $\widetilde Q_*$-droplet, with $K_t\prec K_*$ with respect to the 
weight $\widetilde Q_*$. The uniqueness part is now a consequence of Lemma 
\ref{lem-6.11}.
\end{proof} 

\begin{rem}
(a) A key element of the proof of Theorem \ref{thm-7.7} is the identification 
of the weak solutions of the Hele-Shaw equation $t\mapsto K_t$ with domination 
chains with respect to the weight $\widetilde Q_*$.

\noindent(b)
Theorem \ref{thm-7.7} supplies existence and uniqueness in the backward 
time direction. It is not difficult to see that there is even {\em local} 
uniqueness in the backward time direction. 
However, in the forward time direction, there is generally neither existence 
nor uniqueness. An example of non-uniqueness can be based on, e.g., the 
setting of Remark \ref{rem-6.7}. We now discuss non-existence. In the context
of Theorem \ref{thm-7.7}, the difference $U^{Q,K_*}-U^{Q,K_t}$ is constant on
$K_t$ for $0<t<t_*$ (see, e.g., the proof of Corollary \ref{cor-6.23}); 
let $c(t)$ be that constant. We consider the 
functions $H_t:=Q-U^{Q,K_t}$ and $H_*:=Q-U^{Q,K_*}$, which have 
$\Delta H_t=0$ a.e. on $K_t$ and $\Delta H_*=0$ a.e. on $K_*$, respectively.
We have $H_t-H_*=U^{Q,K_*}-U^{Q,K_t}=c(t)$ on $K_t$, we we write as $H_t=c(t)+
H_*$ on $K_t$. So $H_t$ restricted to $K_t$ is supposed to have an extension 
to $K_*$ -- the function $\widetilde H_t:=c(t)+H_*$ -- with $\Delta 
\widetilde H_t=0$ a.e. on $K_*$. This adds an additional smoothness requirement
on $H_t$ for $0<t<t_*$, which suggests that $K_t$ cannot be an arbitrary 
compact subset of $\C$ with $\Delta Q\ge0$ a.e. on $K_t$ which lacks 
$Q$-shallow points. But $K_t$ is uniquely given for $0<t<t_*$ (the backward 
direction) for arbitrary compacts $K_*$ lacking $Q$-shallow points.  
So with very irregular $K_*$ we should be able to arrange that we have 
non-existence in the forward time direction. Another reason for non-existence 
in the forward direction is the existence of maximal local $Q$-droplets
(see the next section for details), at least for some $Q$ with 
$\Delta Q\equiv1$.
\end{rem}

%begin{lem}  Generalized solutions satisfy the following  properties:
%
%(i)  For all $t_1<t_2$ and all functions $h$ superharmonic, integrable in 
%K_{t_1}^c$, harmonic at infinity, we have
%$\int_{K_2\sm K_1}h~dA\le (t_2-t_2)~h(\infty).$$
%
%(ii)  For all $t_1<t_2$, we have
%$K_1=S_{t_1}(-U_{t_2}, K_2).$$
%
%(iii)  For all $t_1<t_2$, $U_{t_2}-U_{t_1}\equiv\const$ on $K_1$.
%
% The properties (i)-(iii) are equivalent.\end{lem}

A proof of the following statement can be based on Proposition \ref{prop-6.18}.
The only part that needs checking is the absolute continuity requirement, which
we leave to the interested reader.

\begin{prop}
Suppose $t\mapsto K_t$ is a weak solution to the Hele-Shaw equation and that
for almost all $t$ the sets $\Int\, K_t$ have finitely many components. 
Then $t\mapsto K_t$ is a generalized solution.  
\end{prop}

%{\bf Remark.}  IVP for HS equation with weight $\rho$. Claim:   
%existence and uniqueness of a weak solution in the backward direction.

% [Pf: if $K$ is the initial set, then set $Q=-U_K$ and consider the 
%chain of $(Q,K)$-droplets.]

\end{document}